\documentclass[11pt]{amsart}
\usepackage{amssymb,mathrsfs,graphicx,enumerate,color}

\usepackage{mathtools}
\usepackage{hyperref} 
\usepackage{comment}

\newcommand{\RR}{{\mathbb R}}

\newcommand{\Bcal}{\mathcal{B}}

\newcommand{\Dcal}{\mathcal{D}}
\newcommand{\Ecal}{\mathcal{E}}
\newcommand{\Fcal}{\mathcal{F}}
\newcommand{\Gcal}{\mathcal{G}}

\newcommand{\Ical}{\mathcal{I}}
\newcommand{\Jcal}{\mathcal{J}}

\newcommand{\Mcal}{\mathcal{M}}

\newcommand{\Ocal}{\mathcal{O}}
\newcommand{\Pcal}{\mathcal{P}}

\newcommand{\Rcal}{\mathcal{R}}

\newcommand{\Xcal}{\mathcal{X}}

\newcommand{\VBar}{\overline{V}}
\newcommand{\WBar}{\overline{W}}

\newcommand{\Ubar}{\bar{U}}

\renewcommand{\hbar}{\bar{h}}

\newcommand{\pbar}{\bar{p}}

\newcommand{\ubar}{\bar{u}}
\newcommand{\vbar}{\bar{v}}

\newcommand{\ptil}{\tilde{p}}

\newcommand{\util}{\tilde{u}}
\newcommand{\vtil}{\tilde{v}}

\newcommand{\abs}[1]{\left|#1\right|}
\newcommand{\aabs}[1]{\Big|#1\Big|}

\newcommand{\norm}[1]{\left\|#1\right\|}

\DeclareMathOperator{\supp}{supp}

\renewcommand{\a}{\alpha}
\renewcommand{\b}{\beta}
\newcommand{\g}{\gamma}
\renewcommand{\d}{\delta}
\newcommand{\e}{\varepsilon}

\newcommand{\z}{\zeta}
\newcommand{\et}{\eta}
\renewcommand{\th}{\theta}
\renewcommand{\k}{\kappa}
\renewcommand{\l}{\lambda}

\newcommand{\m}{\mu}
\newcommand{\n}{\nu}
\newcommand{\x}{\xi}

\renewcommand{\r}{\rho}
\newcommand{\s}{\sigma}
\renewcommand{\t}{\tau}

\newcommand{\ps}{\psi}

\renewcommand{\O}{\Omega}

\newcommand{\rd}{\partial}
\newcommand{\thbar}{\bar{\th}}
\newcommand{\thtil}{\tilde{\th}}

\newcommand{\one}[1]{\mathbf{1}_{\{#1\}}}

\usepackage{colortbl}
\definecolor{black}{rgb}{0.0, 0.0, 0.0}
\definecolor{red}{rgb}{1.0, 0.5, 0.5}

\title[   ]{Stability of a Riemann shock in a physical class: \\from Brenner-Navier-Stokes-Fourier to Euler}

\author[Eo]{Saehoon Eo}
\address[Saehoon Eo]
{ Department of Mathematics, \newline
Stanford University \\
CA 94305, USA}
\email{eosehoon@stanford.edu}

\author[Eun]{Namhyun Eun}
\address[Namhyun Eun]
{ Department of Mathematical Sciences, \newline
Korea Advanced Institute of
Science and Technology \\
Daejeon 34141, Korea}
\email{namhyuneun@kaist.ac.kr}
 
\author[Kang]{Moon-Jin Kang}
\address[Moon-Jin Kang]
{ Department of Mathematical Sciences, \newline
Korea Advanced Institute of
Science and Technology \\
Daejeon 34141, Korea}
\email{moonjinkang@kaist.ac.kr}

\newtheorem{theorem}{Theorem}[section]
\newtheorem{lemma}{Lemma}[section]

\newtheorem{proposition}{Proposition}[section]
\newtheorem{remark}{Remark}[section]

\newcommand{\bbr}{\mathbb R}

\newcommand{\deo}{\d_0}

\numberwithin{figure}{section}
\newcommand{\beq}{\begin{equation}}
\newcommand{\eeq}{\end{equation}}
\newcommand{\bsp}{\begin{split}}
\newcommand{\esp}{\end{split}}

\newcommand{\wc}{\rightharpoonup}

\newcommand{\vt}{{\tilde{v}}}
\newcommand{\ut}{{\tilde{u}}}
\newcommand{\pt}{{\tilde{p}}}
\newcommand{\tht}{{\tilde{\th}}}

\newcommand{\Ut}{{\tilde{U}}}

\newcommand{\vtn}{{\tilde{v}^\nu}}
\newcommand{\utn}{{\tilde{u}^\nu}}
\newcommand{\thtn}{{\tilde{\th}^\nu}}

\newcommand{\vn}{{v^\nu}}
\newcommand{\un}{{u^\nu}}
\newcommand{\thn}{{\th^\nu}}

\newcommand{\Phif}[2]{\Phi\Big(\frac{#1}{#2}\Big)}
\newcommand{\bpf}[1]{\noindent\textbf{Proof of \eqref{#1}:}} 
\newcommand{\bci}[1]{\noindent\textbf{Control of \(I_{#1}\):}} 
\newcommand{\step}[1]{\vskip0.2cm \noindent{\it Step #1:} }

\begin{document}

\date{\today}


\thanks{\textbf{Acknowledgment.} This work was supported by Samsung Science and Technology Foundation under Project Number SSTF-BA2102-01. The authors thank HyeonSeop Oh and Professor Alexis F. Vasseur for valuable comments.}

\begin{abstract}
The stability of an irreversible singularity, such as a Riemann shock to the full Euler system, in the absence of any technical conditions on perturbations, remains a major open problem even within mono-dimensional framework.
A natural approach to justify such stability is to consider vanishing dissipation (or viscosity) limits of physical viscous flows.
We prove the existence of vanishing dissipation limits, on which a Riemann shock of small amplitude is stable (up to a time-dependent shift) and unique.
Thus, a Riemann weak shock is rigid (not turbulent) under physical disturbances.
We adopt the Brenner-Navier-Stokes-Fourier system, based on the bi-velocity theory, as a physical viscous model. 
The key ingredient of the proof is the uniform stability of the viscous shock with respect to the viscosity strength.
The uniformity is ensured by contraction estimates of any large perturbations around the shock.
The absence of any restrictions on size of initial perturbations forces us to handle extreme values of density and temperature, which constitutes the most challenging part of our analysis.
We use the method of \(a\)-contraction with shifts, but we improve it by introducing a more delicate analysis of the localizing effect given by viscous shock derivatives. 
This improvement possesses a degree of robustness that renders it applicable to a wide range of models.
This is the first resolution for the challenging open problem on the ``unconditional'' stability and uniqueness of Riemann shock solutions to the full Euler system in a class of vanishing physical dissipation limits.
\end{abstract}
\maketitle \centerline{\date}


\section{Introduction}
\setcounter{equation}{0}
The compressible Euler system has been extensively studied in various directions over an extended period.
In 1965, Glimm \cite{Glimm65} proved the global existence of entropy solutions to \(n \times n\) hyperbolic system of conservation laws (especially for the Euler system) in one space dimension, for small BV initial data.
Concerning the uniqueness and stability, Dafermos and DiPerna established the Weak/Strong principle in 1979 \cite{Dafermos96,Diperna79}, to show that Lipschitz solutions (strong solutions) are unique and \(L^2\)-stable within the broader class of bounded entropic (weak) solutions.

\vspace{2mm}
For singular solutions containing jump discontinuities such as Riemann shocks, the scenario for uniqueness and stability changes significantly.
A Riemann shock is a self-similar singularity that is irreversible and planar jump discontinuous, first proposed by Riemann \cite{Riemann1860} in 1860s, as an entropy solution to the Euler system with Riemann initial data.
For the multi-dimensional isentropic case, De Lellis-Sz\'ekelyhidi \cite{DeLellisLaszlo10} and Chiodaroli-De Lellis-Kreml \cite{ChioDeLellisKreml15} showed non-uniqueness of entropy solutions.
They showed that entropy solutions to the multi-dimensional isentropic Euler systems are not unique, by constructing infinitely many entropy solutions based on the convex integration method \cite{DeLellisLaszlo09,DeLellisLaszlo10}.
We refer to \cite{ChioFeireislKreml15,KlinKremlOMM20} for more general result on the ill-posedness of the full Euler system in multi-dimension. 
These results imply that the entropy inequality (or 2nd law of thermodynamics) is not enough in general to select a unique and stable solution.
We also refer to \cite{ChenVasseurYu,Krupa24,KrupaNonT4} for the study of mono-dimensional hyperbolic conservation laws based on the convex integration method.

\vspace{2mm}
For mono-dimensional case, the uniqueness and stability of singular solutions were proved under some technical conditions. 
It was shown that Riemann solutions containing shocks are unique and stable under certain conditions on perturbations, such as locally BV regularity in space \cite{ChenFridLi02}; strong trace property \cite{KV16,Krupa21,LegerVasseur11,Vasseur16}.   
For more general solutions, it was proved in \cite{BressanCrastaPiccoli00,BressanLiuYang99,ChenKrupaVasseur22,LiuYang99} that small BV entropy solutions are unique and stable, provided that the solutions satisfy some extra conditions such as Tame Oscillation Condition; Bounded Variation Condition along space-like curves; strong trace property.
Those technical conditions for solutions are removed in \cite{BressanDeLellis23,BressanGuerra24,ChenKrupaVasseur22} only for the uniqueness results.

\vspace{2mm}
However, the study of the stability of singularities without any technical conditions for perturbations remains a challenging open problem, even in one-dimension. A natural way to justify the stability involves to consider vanishing ``physical" viscosity limits (or dissipation limits) of Navier-Stokes flows with evanescent viscosities, as ``physical" disturbances (for example, see \cite{dafermos2005hyperbolic}).
In this context, Bianchini-Bressan \cite{BianchiniBressan05} first constructed unique and stable small BV solutions from inviscid limits of parabolic flows with `artificial' viscosities (given by Laplacians for all conserved quantities).
For physical viscous system such as Navier-Stokes systems, Kang-Vasseur \cite{KV-Inven} first proved the stability of a Riemann shock of small amplitude in a class of inviscid limits of Navier-Stokes solutions. More precisely, for the isentropic Euler system, they constructed a set of inviscid limits of Navier-Stokes solutions perturbed from a Riemann weak shock, on which the shock is stable. This result was extended to the case for Riemann solution composed of two weak shocks \cite{KV-JDE}. Recently, based on those results and \cite{Bressan20,ChenKrupaVasseur22,ChenPerepelitsa10}, Chen-Kang-Vasseur \cite{ChenKangVasseur24arxiv} first constructed unique and stable small BV solutions to the isentropic Euler system from inviscid limits of Navier-Stokes solutions. We also refer to \cite{KVW-CMP} regarding  the stability result in 3D of a planar contact discontinuity without shear (i.e., an entropy wave) in the class of vanishing dissipation limits.
However, up to now, for the full Euler system, the above problem of the shock stability is open. In this paper, we present the first result on this challenging problem.

\vspace{2mm}
We aim to investigate the stability and uniqueness of a Riemann shock for the one-dimensional full Euler system in a physical class composed of vanishing dissipation limits from an associated physical viscous system.
For brevity, we consider the full Euler system in the Lagrangian mass coordinates:
\begin{align}
\left\{
\begin{aligned} \label{Euler}
    &v_t-u_x = 0, \\
    &u_t+p_x = 0, \\
    &\Big(e+\frac{u^2}{2}\Big)_t+(pu)_x = 0,
\end{aligned}
\right.
\end{align}
together with the relations (for ideal polytropic gas):
\begin{equation} \label{pressure}
p=p(v,\th)=\frac{R\th}{v}, \qquad e(\th)=\frac{R}{\g-1}\th.
\end{equation} 
Here, \(v\) denotes the specific volume, \(u\) the fluid velocity, \(\th\) the absolute temperature, \(E:=e+\frac{u^2}{2}\) the total energy, \(R>0\) the gas constant and \(\g>1\) the adiabatic constant.

\subsection{Brenner-Navier-Stokes-Fourier system}
For the physical viscous system associated to \eqref{Euler}, we consider the so-called Brenner-Navier-Stokes-Fourier(BNSF) system in the Lagrangian mass coordinates:
\begin{align}
\left\{
\begin{aligned} \label{LBNSF}
    &v_t -(u_v)_x = \Big(\t(\th)\frac{v_x}{v}\Big)_x, \\
    &(u_v)_t+p_x = \Big(\mu(\th)\frac{(u_v)_x}{v}\Big)_x, \\
    &\Big(e+\frac{(u_v)^2}{2}\Big)_t+(p u_v)_x = 
    \Big(\k(\th)\frac{\th_x}{v}\Big)_x+ \Big(\m(\th)\frac{u_v (u_v)_x}{v}\Big)_x.
\end{aligned}
\right.
\end{align}
Here, $u_v$, \(\m\) and \(\k\) respectively represent the volume velocity, the viscosity coefficient and the heat conductivity coefficient. Especially, \(\t(\th)\coloneqq\frac{\k(\th)}{c_p}\) is called the Brenner coefficient, where \(c_{p}\) denotes the specific heat capacity at constant pressure.
These are all described in detail below. 
The BNSF system has been proposed by Brenner \cite{brenner2006fluid} to improve some flaws of the compressible Navier-Stokes-Fourier(NSF) system.
He argued in \cite{brenner2005kinematics,brenner2005navier,brenner2006fluid} that the NSF system does not adequately describe compressible flows, particularly under extreme conditions such as rarefied gases, shock waves and gaseous flows through micro-channels.
In fact, there have been numerous studies pointing out and attempting to improve the shortcomings of the NSF system(see \cite{arkilic2001mass, dadzie2008continuum, dongari2009extended, greenshields2007structure, harley1995gas, klimontovich1992need, klimontovich1993hamiltonian}), but it has been accepted that Brenner's justification is the most systematic and well-established.
His idea is based on the bi-velocity theory, which claims the presence of two different velocities to describe the motion of fluid, namely the mass velocity \(u_m\) (which is identical to fluid velocity appearing in classical fluid equations) and the volume velocity \(u_v\). 
He observed that, in general, \(u_m \neq u_v\), and the inconsistency increases as the density gradient becomes larger. 
More specifically, the BNSF system is introduced in \cite{brenner2006fluid} in arbitrary dimensions and is thus originally written in Eulerian coordinates.
In the one-dimensional case, it takes the following form:
\beq\label{EBNSF}
\left\{
\begin{aligned}
    & \rd_t \rho + \rd_y (\rho u_m) = 0,\\
    & \rd_t (\rho u_v) + \rd_y (\rho u_v u_m + p) = \rd_y (\m(\th) \rd_y u_v), \\
    & \rd_t \Big(\rho\Big(\frac{u_v^2}{2}+e\Big)\Big) + \rd_y \Big(\rho\Big(\frac{u_v^2}{2}+e\Big)u_m + pu_v\Big) = \rd_y (\m(\th) u_v\rd_y u_v + \k(\th)\rd_y \th),
\end{aligned}
\right.
\eeq
where \(\r\) denotes the density and \(u_m\) the mass velocity which is used to describe the mass transportation and convection effect and is identical to the fluid velocity appearing in classical fluid equations such as the NSF system, while the volume velocity is introduced to define the momentum, energy, work and viscous stress. 
Moreover, from \cite{brenner2006fluid}, the constitutive relation between \(u_v\) and \(u_m\) is given by 
\begin{equation}\label{brenner}
u_v = u_m + \frac{\k(\th)}{\r c_p}\partial_y\ln\r.
\end{equation}
Note that $u_v=u_m$ in the Euler system by $\kappa=0$.
In his series of works \cite{brenner2005kinematics,brenner2005navier,brenner2006fluid}, he proposed \eqref{brenner} based on the theory by \"Ottinger and justified it through various ways including alignment with Burnett's solution to the Boltzmann equation and comparison with experimental results.
In \cite{brenner2006fluid}, he also demonstrated that \eqref{brenner} is consistent with linear irreversible thermodynamics, and notably, representing the no-slip boundary condition at solid surfaces using the volume velocity more accurately reflects experimental observations.

Notice that \eqref{brenner} shows that the difference between \(u_v\) and \(u_m\) is proportional to density gradient, and can be zero or negligibly small near the regime where the density is uniform in space as incompressible fluids.
This feature justifies our adoption of the BNSF system as a physical viscous model in the present paper.
To obtain an estimate of viscous perturbations around the shock that is uniform with respect to the viscous strength, which implies the existence of a set of inviscid limits for the stability of the shock, it is essential to control perturbed viscous flows with abrupt spatial variations, where the density could be arbitrarily small or large, as in our strategy for the proof (see Section \ref{sec:idea}). 

The need to introduce the concept of volume velocity is elaborated in a more intuitive manner as follows: consider the work produced by the displacement of gas particles due to the pressure within the gas.
To determine the amount of this work, it is essential to define the displacement distance.
However, it is important to note that it is not the individual gas particles that generate pressure and are displaced by it, but rather a collection of particles.
Since the boundary of such a collection is inherently ambiguous, it is intuitively evident that a rigorous definition of the displacement distance is infeasible.
This is linked to the fact that the volume of gas is not a point-based property.
The volume of gas could not be understood as an attribute of particles at a single point, but rather it could be interpreted, at best, statistically as a property of a collection of particles.
Therefore, the concept of volume velocity, as a new form of velocity, is necessary to describe the motion of particle collections, rather than relying solely on the traditional concept of velocity.

Despite the thoroughness of Brenner's theory, there are few results that analyze the BNSF system from the mathematical perspective. 
Feireisl-Vasseur \cite{FeVa} established the existence of the global weak solutions to the initial boundary value problem of the BNSF system. It is also worth mentioning that the BNSF system is used to construct measure-valued solutions of Euler in \cite{Fe1}, and the finite volume scheme for Euler in \cite{Fe2}.

\vspace{2mm}
To obtain the Lagrangian representation \eqref{LBNSF} of the BNSF system as a physical viscous model associated to \eqref{Euler}, we rewrite the (original) BNSF system \eqref{EBNSF} (in Eulerian coordinates $(t,y)$) in Lagrangian mass coordinates $(t,x)$, where the flow map is derived from the mass velocity $u_m$ as follows: 
\[
\partial_t y(t,x) = u_m (t,y(t,x)),\qquad \int_0^{y(0,x)} \rho(0,z) dz = x.
\]
Indeed, $v(t,x)=\frac{1}{\rho(t,y(t,x))}, u_v(t,x)=u_v(t,y(t,x))$ and $u_m(t,x)=u_m(t,y(t,x))$ satisfy the system composed of $v_t - (u_m)_x =0$ and the last two equations of \eqref{LBNSF}. Then, $v_t - (u_m)_x =0$ can be written as the first equation of \eqref{LBNSF} by using the constitutive relation $u_v=u_m - \frac{\k(\th)}{c_p}\frac{v_x}{v}$ from \eqref{brenner}, together with $\tau(\th)=\kappa(\th)/c_p$.

\vspace{2mm}
From the Chapman-Enskog theory or the first level of approximation in kinetic theory, the viscosity and the heat conductivity coefficients are known to be a function of temperature alone. 
According to Chapman-Cowling \cite{ChapmanCowling90} or Vincenti-Kruger \cite{VincentiKruger66}, for a gas with intermolecular potential proportional to \(r^{-a}\), where \(r\) is the intermolecular distance, \(\m(\th)\) and \(\k(\th)\) are given as follows: 
\beq\label{mkbeta}
\m(\th) = \m_0\th^\b, \qquad \k(\th) = \k_0\th^\b, \qquad \text{ where } \b = \frac{a+4}{2a}.
\eeq
Thus, the Brenner coefficient \(\t(\th) = \k(\th)/c_p\) may also degenerate as $\theta^\beta$ near absolute zero.

\vspace{2mm}
In what follows, for notational simplicity, we set $u=u_v$ in \eqref{LBNSF} as the mass velocity \(u_m\) does not appear. To construct the desired class of vanishing dissipation limits, we consider a vanishing parameter $\nu>0$ on the right-hand side of \eqref{LBNSF}:
\begin{align}
\left\{
\begin{aligned} \label{inveq}
    &v_t^\nu-u_x^\nu = \nu\Big(\t(\th^\nu)\frac{v_x^\nu}{v^\nu}\Big)_x, \\
    &u_t^\nu+p(v^\nu,\th^\nu)_x = \nu\Big(\mu(\th^\nu)\frac{u_x^\nu}{v^\nu}\Big)_x, \\
    &\Big(e^\nu+\frac{|u^\nu|^2}{2}\Big)_t+(p(v^\nu,\th^\nu)u^\nu)_x = 
    \nu\Big(\k(\th^\nu)\frac{\th_x^\nu}{v^\nu}\Big)_x+\nu \Big(\m(\th^\nu)\frac{u^\nu u_x^\nu}{v^\nu}\Big)_x.
\end{aligned}
\right.
\end{align}

\subsection{Riemann shocks and viscous shocks}
We consider a Riemann shock to the Euler system \eqref{Euler} as
\[
(\vbar,\ubar,\bar{E})(x-\s t)= \left\{
\begin{aligned}
(v_-,u_-,E_-) \quad \text{if } x-\s t<0, \\
(v_+,u_+,E_+) \quad \text{if } x-\s t>0,
\end{aligned} \right.
\]
where the constant states \((v_-,u_-,E_-)\) and \((v_+,u_+,E_+)\), and its shock speed $\s$ satisfy 
the Rankine-Hugoniot condition and the Lax entropy condition:
\begin{align}
\begin{aligned} \label{end-con}
&\exists~\s \quad \text{s.t.}~\left\{
\begin{aligned}
&-\s(v_+-v_-) -(u_+-u_-) =0, \\
&-\s(u_+-u_-) +p(v_+,\th_+)-p(v_-,\th_-) =0, \\
&-\s(E_+-E_-) +p(v_+,\th_+)u_+ -p(v_-,\th_-)u_- =0,
\end{aligned} \right. \\
&\text{and either \(v_->v_+\), \(u_->u_+\), \(\th_-<\th_+\) or \(v_-<v_+\), \(u_->u_+\), \(\th_->\th_+\) holds.}
\end{aligned}
\end{align}
Here, \(\th_-\) and \(\th_+\) satisfy \(E_- = \frac{R}{\g-1}\th_- + \frac{1}{2}u_-^2\) and \(E_+ = \frac{R}{\g-1}\th_+ + \frac{1}{2}u_+^2\).
As is common, when dealing with the full Euler system (and BNSF system), it is simpler to handle $\theta$ variable rather than the conserved variable $E$.
Hence, we use (\(\vbar,\ubar,\thbar\)) for the Riemann shock connecting the two end states \((v_-,u_-,\th_-)\) and \((v_+,u_+,\th_+)\) satisfying \eqref{end-con} as follows:
\begin{equation} \label{shock-0}
(\vbar,\ubar,\thbar)(x-\s t)= \left\{
\begin{aligned}
(v_-,u_-,\th_-) \quad \text{if } x-\s t<0, \\
(v_+,u_+,\th_+) \quad \text{if } x-\s t>0.
\end{aligned} \right.
\end{equation}
If \(v_->v_+\), (\(\vbar,\ubar,\thbar\))(\(x-\s t\)) is a \(1\)-shock wave with velocity \(\s=- \sqrt{-\frac{p_+-p_-}{v_+-v_-}}<0\), where \(p_\pm \coloneqq p(v_\pm, \th_\pm)\).
If \(v_-<v_+\), it is a \(3\)-shock wave with \(\s= \sqrt{-\frac{p_+-p_-}{v_+-v_-}}>0\).

In our asymptotic analysis, as a viscous counterpart of the Riemann shock, we will consider a viscous shock wave (\(\vtn, \utn, \thtn\))(\(x-\s t\)) connecting the two constant states satisfying the conditions in \eqref{end-con}, which is a smooth and monotone traveling wave solution to \eqref{inveq}:
\begin{align}
\left\{
\begin{aligned} \label{shock_0}
    &-\hspace{-0.2mm}\s (\vtn)'-(\utn)' = \nu\Big( \t(\thtn)\frac{(\vtn)'}{\vtn}\Big)', \\
    &-\hspace{-0.2mm}\s (\utn)'+p(\vtn,\thtn)' = \nu\Big(\m(\thtn)\frac{(\utn)'}{\vtn}\Big)', \\
    &-\hspace{-0.2mm}\s \Big(\frac{R}{\g-1}\thtn + \frac{1}{2}(\utn)^2\Big)'+(p(\vtn,\thtn)\utn)' = 
    \nu\Big(\k(\thtn)\frac{(\thtn)'}{\vtn}\Big)'+\nu \Big(\m(\thtn)\frac{\utn (\utn)'}{\vtn}\Big)', \\
    &\lim_{\x\to\pm\infty} (\vtn,\utn,\thtn)(\x)=(v_\pm,u_\pm,\th_\pm).
\end{aligned}
\right.
\end{align}
Notice that the existence and uniqueness of the viscous shocks are stated in Lemma \ref{lem-VS}.

\subsection{Hypotheses on Brenner coefficient}
In this article, we consider the case \(\b = 2\) in \eqref{mkbeta}, i.e., \(\m\) and \(\k\) are proportional to \(\th^2\). 
To discuss the Brenner coefficient \(\t(\th)=\k(\th)/c_p\), we need to observe the fact that the specific heat capacity \(c_p\), is often assumed to be constant at moderate temperatures but degenerates as the temperature approaches the absolute zero, due to loss of available degrees of freedom: since \(c_p\) is defined as \(c_p=\th\frac{dS}{d\th}\), and the entropy $S$ cannot blow up as \(\th \to 0\) by the third law of thermodynamics, \(c_p \to 0\) as \(\th \to 0\) (see \cite{fermi2012thermodynamics,landau2013statistical}).
Furthermore, it is well-known from Debye's law (see \cite{kittel1980thermal,kittel2018introduction}) that the specific heat degenerates proportionally to \(\th^3\) at low temperatures. 
Consequently, the Brenner coefficient could be large as $1/\theta$ near \(\th=0\), while it behaves like $\th^2$ at moderate temperatures.

To simplify the analysis, we assume that $\t(\th)$ converges to a constant state as \(\th\to0\) rather than diverging to infinity.
Accordingly, based on the above discussion, the Brenner coefficient is therefore assumed to satisfy
\[
\t(\th) \coloneqq \t_0 \th_0^2 + \t_1 \th^2
\]
where \(\th_0>0\) is a prescribed constant.
For the sake of simplicity, we consider a normalization \(\th_0=\t_1=\m_0=\k_0=1\) (which does not affect our analysis)
\begin{equation} \label{mu-def}
\t(\th)=\t_0 + \th^2, \qquad
\m(\th)= \th^2, \qquad \k(\th)= \th^2.
\end{equation}
We also consider the case when the left state \(\th_-\) is relatively smaller than $\tau_0$ such that
\begin{equation} \label{cold}
\th_- \le \frac{1}{3} \sqrt{\frac{R\t_0}{R + R\g + (\g-1)^2}} \,\,\Leftrightarrow \,\,
\frac{R\t_0}{R(\t_0+(\g+1)\th_-^2) + (\g-1)^2 \th_-^2} \ge 0.9.
\end{equation}
Notice that the condition \eqref{cold} could be relaxed by considering large \(\t_0\). 
The need for this condition \eqref{cold} will be explained in Section \ref{sec:idea}.

\subsection{Main results}
To measure perturbations of the Riemann shocks, we use the relative entropy associated to the entropy of \eqref{Euler} as follows: for any positive functions \(v_1,\th_1,v_2,\th_2\) and any functions \(u_1, u_2\),
\begin{equation} \label{eta_def}
\et((v_1,u_1,\th_1)|(v_2,u_2,\th_2)) \coloneqq R\Phif{v_1}{v_2}+\frac{R}{\g-1}\Phif{\th_1}{\th_2}+\frac{(u_1-u_2)^2}{2\th_2}
\end{equation}
where \(\Phi(z) \coloneqq z-1-\log z\).
Note that \(\Phi(z_1/z_2)\) is the relative functional associated with the strictly convex function \(Q(z)\coloneqq-\log z\).
That is,
\begin{align}
\begin{aligned} \label{Phi=Q}
Q(z_1|z_2)
&\coloneqq Q(z_1)-Q(z_2)-Q'(z_2)(z_1-z_2) \\
&=-\log z_1 + \log z_2 +\frac{1}{z_2}(z_1-z_2)
= \Phif{z_1}{z_2}.
\end{aligned}
\end{align}

We will consider limits of solutions to the Brenner-Navier-Stokes-Fourier system for the first components \(z_1\), i.e., \(v_1\) and \(\th_1\). However, we obtain uniform bounds at most in \(L^1\) for the solutions.
Hence, the limits could be measures on \(\RR^+\times\RR\).
Intriguingly, this could be physical, and it is related to the possible appearance of a cavitation.
For this reason, we need to generalize the notion of relative entropy to measures defined on \(\RR^+ \times \RR\).
Fortunately, it is enough to extend the definition only for the case when we compare a measure \(dv\) (or \(d\th\)) with a simple function \(\vbar\)(or \(\thbar\)) which takes only two values \(v_-\) and \(v_+\)(or \(\th_-\) and \(\th_+\)).
Let \(v_a\) be the Radon-Nikodym derivative of \(dv\) with respect to the Lebesgue measure and \(dv_s\) be its singular part, i.e., \(dv=v_a dtdx+dv_s\).
Then, we define the relative functional as 
\begin{align} \label{dQ}
d\Phif{v}{\vbar} \coloneqq \Phif{v_a}{\vbar}dtdx + \frac{1}{\VBar(t, x)} dv_s(t,x),
\end{align}
where \(\VBar\) is given by
\[
\VBar(t,x)=
\begin{cases}
\max(v_-,v_+) & \text{for } (t,x)\in \overline{\O_M} (\eqqcolon \text{the closure of }\O_M), \\
\min(v_-,v_+) & \text{for } (t,x)\notin \overline{\O_M},
\end{cases}
\]
and \(\O_M=\{(t,x)|\vbar(t,x)=\max(v_-,v_+)\}\).
We use \(\VBar\), not \(\vbar\), because it has to be defined at every point to deal with \(dv_s\).
Notice that in this case, the relative entropy is a measure itself.
Furthermore, if \(v\in L^{\infty}(\RR^+;L^\infty(\RR)+\Mcal(\RR))\) and its Radon-Nikodym derivative \(v_a\) is away from \(0\), then \(d\Phi(v/\vbar)\) does belong to \(L^{\infty}(\RR^+;L^\infty(\RR)+\Mcal(\RR))\), where \(\Mcal\) denotes the space of nonnegative Radon measures.
One last thing we mention here is that for the case of \(\th\), we extend the definition of the relative functional to measures in the same manner.

For the global-in-time existence of solutions to \eqref{inveq}, we introduce the function space:
\begin{multline*}
\Xcal_T \coloneqq
\{(v,u,\th) \mid 
v-\underline{v}, u-\underline{u}, \th-\underline{\th} \in L^\infty(0,T;H^1(\RR))\cap L^2(0,T;H^2(\RR)), \\
v^{-1}, \theta^{-1} \in L^\infty((0,T)\times\RR) \}
\end{multline*}
where \(\underline{v}, \underline{u}\), and \(\underline{\th}\) are smooth monotone functions such that
\begin{align} \label{sm-end}
\underline{v}(x)=v_{\pm}, \quad \underline{u}(x)=u_{\pm}, \quad \underline{\th}(x)=\th_{\pm} \quad \text{for } \pm x \ge 1.
\end{align}

The main theorem is on stability and uniqueness of the entropy shocks to \eqref{Euler}:
\begin{theorem}\label{thm_inviscid}
Let \(R>0,\g>1,\t_0>0\) be any constants.
For each \(\n>0\), consider the system \eqref{inveq} with \eqref{pressure} and \eqref{mu-def}.
Then, for any given constant state \((v_-, u_-, \th_-)\in \RR^+\times \RR\times \RR^+\) satisfying \eqref{cold}, there exists a constant \(\e_0>0\) such that for any \(\e<\e_0\) and any \((v_+, u_+, \th_+)\in \RR^+\times \RR\times \RR^+\) satisfying \eqref{end-con} with \(\abs{v_+-v_-} = \e\), the following holds. \\
Let  (\(\vbar,\ubar,\thbar\)) be a Riemann shock connecting the two end states \((v_-, u_-, \th_-)\) and \((v_+, u_+, \th_+)\), and \((\vtn, \utn, \thtn)\) be the associated viscous shock  as a solution of \eqref{shock_0}. \\
Then, for a given initial datum \(U^0=(v^0, u^0, \th^0)\) of \eqref{Euler} satisfying 
\begin{equation}\label{basic_ini}
\Ecal_0 := \int_\RR \et((v^0, u^0, \th^0)|(\vbar, \ubar, \thbar))dx < \infty,
\end{equation}
the following is true. \\
(i) (Well-prepared initial data) There exists a sequence of smooth functions \(\{(v_0^\nu, u_0^\nu, \th_0^\nu)\}_{\nu>0}\) on \(\RR\) such that
\begin{equation} \label{ini_conv}
\begin{aligned}
&\lim_{\n\to 0}v_0^\nu = v^0, \quad \lim_{\n\to 0}u_0^\nu = u^0, \quad \lim_{\n\to 0}\th_0^\nu = \th^0, \quad a.e., \quad v_0^\nu, \th_0^\nu > 0, \\
&\lim_{\n\to 0}\int_\RR \Big(R\Phif{v_0^\nu}{\vtn} + \frac{R}{\g-1}\Phif{\th_0^\nu}{\thtn} + \frac{(u_0^\nu-\utn)^2}{2\thtn}\Big) dx = \Ecal_0.
\end{aligned}
\end{equation}
(ii) For any given \(T>0\), let \(\{(\vn, \un, \thn)\}_{\n>0}\) be a sequence of solutions in \(\Xcal_T\) to \eqref{inveq} with the initial datum \((v_0^\nu, u_0^\nu, \th_0^\nu)\) as above. 
Then there exist inviscid limits \(v_\infty, u_\infty\), and \(\th_\infty\) such that as \(\n\to 0\) (up to a subsequence),
\begin{equation}\label{wconv}
\vn\wc v_\infty, \quad \un\wc u_\infty, \quad \thn\wc \th_\infty \quad in \,\, \Mcal_{loc}((0, T)\times \RR),
\end{equation}
where \(v_\infty\) and \(\th_\infty\) lie in \(L^\infty(0, T, L^\infty(\RR)+\Mcal(\RR))\) and \(\Mcal_{loc}((0, T)\times \RR)\) is the space of locally bounded Radon measures. 
In addition, there exist a shift \(X_\infty\in BV((0, T))\) and a constant \(C>0\) such that \(d\Phi(v_\infty/\vbar)\), \(d\Phi(\th_\infty/\thbar)\in L^\infty(0, T;\Mcal(\RR))\), and for a.e. \(t\in (0, T)\), 
\begin{align}
\begin{aligned} \label{uni-est}
&R\int_{x\in\RR} d\Phi(v_\infty/\vbar(x-X_\infty(\cdot)))(t) 
+\frac{R}{\g-1}\int_{x\in\RR} d\Phi(\th_\infty/\thbar(x-X_\infty(\cdot)))(t) \\
&\qquad
+\int_{\RR} \frac{\abs{u_\infty(t,x) - \ubar(x-X_\infty(t))}^2}{2\thbar(x-X_\infty(t))}dx \le C\Ecal_0.
\end{aligned}
\end{align}
Moreover, the shift \(X_\infty\) satisfies 
\begin{equation}\label{X-control}
\abs{X_\infty(t)-\s t} \le \frac{C(T)}{\abs{v_--v_+}}(\sqrt{\Ecal_0}+\Ecal_0^3).
\end{equation}
Therefore, entropy shocks \eqref{shock_0} (with small amplitude) of the full Euler system \eqref{Euler} are stable and unique in the class of weak inviscid limits of solutions to the Brenner-Navier-Stokes-Fourier system \eqref{inveq}. 
\end{theorem}

\begin{remark}
(1) Theorem \ref{thm_inviscid} establishes the uniqueness and stability of weak Riemann shocks in the class of  inviscid limits of solutions to the Brenner-Navier-Stokes-Fourier system.
For the uniqueness, if \(\Ecal_0=0\), then it holds from \eqref{uni-est} and \eqref{X-control} that for a.e. \(t\in(0,T)\), \(X_\infty(t)=\s t\) and hence,
\begin{align*}
&R\int_{\RR} \Phi(v_a(t,x)/\vbar(x-\s t)) dx
+\frac{R}{\g-1}\int_{\RR} \Phi(\th_a(t,x)/\thbar(x-\s t)) dx \\
&\qquad
+\int_{\RR} \frac{\abs{u_\infty(t,x) - \ubar(x-\s t)}^2}{2\thbar(x-\s t)} dx = 0
\end{align*}
where \(dv_\infty=v_a dt dx+dv_s\) and \(d\th_\infty=\th_a dt dx + d\th_s\), and the singular parts \(v_s\) and \(\th_s\) vanish.
Thus, we obtain the uniqueness of the Riemann shock as 
\[
v_\infty(t,x) = \vbar(x -\s t), \, \th_\infty(t,x) = \thbar(x -\s t), \, u_\infty(t,x) = \ubar(x -\s t) \,\, \text{ a.e. } (t,x) \in [0,T]\times\RR.
\] 
(2) The smallness of the shock amplitude does not play any role in the vanishing viscosity limit process.
The small amplitude restriction arises from Theorem \ref{thm_main} and is also used to guarantee the existence of viscous shock waves.\\
(3) Considering the function space \(\Xcal_T\) for solutions $(\vn, \un, \thn)$ is important to ensure the existence of shifts as solutions to the ODE \eqref{X-def}. 
However, until recently, nothing is known about global existence of large strong solutions to \eqref{inveq}. For the existence of desired strong solutions belonging to \(\Xcal_T\), we refer to the ongoing paper \cite{EEK}.
\end{remark}

As described in Section \ref{sec:idea}, the main part of the proof of Theorem \ref{thm_inviscid} is to derive contraction estimates on any large perturbations of viscous shocks to \eqref{inveq} with a fixed $\nu=1$: for simplicity, we handle the equivalent non-divergence form as:
\begin{equation}\label{main}
\left\{
\begin{aligned}
    &v_t-u_x = \Big(\t(\th)\frac{v_x}{v}\Big)_x, \\
    &u_t+p(v,\th)_x = \Big(\mu(\th)\frac{u_x}{v}\Big)_x, \\
    &\frac{R}{\g-1}\th_t+p(v,\th)u_x = 
    \Big(\k(\th)\frac{\th_x}{v}\Big)_x+\m(\th)\frac{(u_x)^2}{v},
\end{aligned}\right.
\end{equation}
and the corresponding viscous shock system equivalent to \eqref{shock_0} with a fixed \(\nu=1\) is as:
\begin{equation}\label{VS}
\left\{
\begin{aligned}
    &-\s \vtil'-\util' = \Big(\t(\thtil)\frac{\vtil'}{\vtil}\Big)', \\
    &-\s \util'+p(\vtil,\thtil)' = \Big(\mu(\thtil)\frac{\util'}{\vtil}\Big)', \\
    &-\s \frac{R}{\g-1}\thtil'+p(\vtil,\thtil)\util' = 
    \Big(\k(\thtil)\frac{\thtil'}{\vtil}\Big)'+\m(\thtil)\frac{(\util')^2}{\vtil}.
\end{aligned}
\right.
\end{equation}
In Theorem \ref{thm_main}, the contraction is measured by the relative entropy (as in \eqref{eta_def} or \eqref{relative_e}).

\begin{theorem}\label{thm_main}
Let \(R>0\), \(\g>1\), \(\t_0>0\) be any constants.
Consider the system \eqref{main} with \eqref{pressure} and \eqref{mu-def}. 
For a given constant state \((v_-,u_-,\th_-)\in\RR^+\times\RR\times\RR^+\) satisfying \eqref{cold}, 
there exist constants \(\e_0, \d_0>0\) such that the following is true.\\
For any \(\e<\e_0\), \(\d_0^{-1}\e<\l<\d_0\), and any \((v_+,u_+,\th_+)\in\RR^+\times\RR\times\RR^+\) satisfying \eqref{end-con} with \(\abs{v_+-v_-} = \e\), 
there exists a smooth monotone function \(a\colon \RR\to \RR^+\) with \(\lim_{x\to\pm\infty} a(x)=1+a_{\pm}\) for some constants \(a_-\) and \(a_+\) with \(\abs{a_+-a_-}=\l\) such that the followings hold.\\
Let \(\Ut\coloneqq (\vt,\ut,\tht)\) be the viscous shock connecting \((v_-,u_-,\th_-)\) and \((v_+,u_+,\th_+)\) as a solution of \eqref{shock_0} with \(\n=1\).
For a given \(T>0\), let \(U\coloneqq (v,u,\th)\) be a solution in \(\Xcal_T\) to \eqref{main} with a initial datum \(U_0\coloneqq (v_0,u_0,\th_0)\) satisfying \(\int_\RR \et(U_0|\Ut) dx<\infty\). 
Then there exist a shift \(X\in W^{1,1}((0,T))\) and a constant \(C>0\) (independent of \(\d_0\), \(\e/\l\) and \(T\)) such that 
\begin{align}
\begin{aligned}\label{cont_main}
&\int_\RR (a\tht)(x) \et\big(U(t,x+X(t))| \Ut(x)\big) dx \\
&\qquad +\d_0\frac{\e}{\l} \int_0^T \int_\RR  \abs{\s a'(x)} \Phi\left(v(s,x+X(s))/\vt(x)\right) dx ds \\
&\qquad +\d_0 \int_0^T \int_\RR a(x)
(v(\t_0+\th^2))(s, x+X(s))\abs{\rd_x\Big(\frac{1}{v(s, x+X(s))}-\frac{1}{\vt(x)}\Big)}^2 dx ds \\
&\qquad +\d_0 \int_0^T \int_\RR a(x)
\frac{1}{v}(s, x+X(s))\abs{\rd_x\left(\th(s, x+X(s))-\tht(x)\right)}^2 dx ds  \\
&\le \int_\RR (a\tht)(x) \et\big(U_0(x)|\Ut(x)\big) dx, \\
\end{aligned}
\end{align}
and 
\begin{align}
\begin{aligned} \label{est-shift}
&\abs{\dot{X}(t)}\le \frac{C}{\e^2}\left(f(t)+1 \right), \quad \text{for a.e. } t\in[0,T], \\
&\text{for some positive function \(f\) satisfying} \quad \norm{f}_{L^1(0,T)} \le\frac{\l}{\d_0\e}\int_\RR \et(U_0|\Ut) dx.
\end{aligned}
\end{align}
\end{theorem}

Note that it is enough to prove Theorem \ref{thm_main} for \(3\)-shocks. For \(1\)-shocks, the change of variables \(x \to -x, u \to -u, \s_\e\to -\s_\e\) gives the corresponding result. 
Thus, from now on, we consider a 3-shock \((\vt,\ut,\tht)\), i.e., \(v_-<v_+, u_->u_+, \th_->\th_+, \) and 
\(
\s_\e = \sqrt{-\frac{p_+-p_-}{v_+-v_-}} > 0.
\)

\begin{remark}
To prove Theorem \ref{thm_main}, we employ the method of \(a\)-contraction with shifts, that was first developed in \cite{KV16,Vasseur16} for the study of stability of Riemann shocks, and extended to the study for stability of large perturbations of viscous shock to viscous conservation laws as in  \cite{Kang19,KO,Kang-V-1,KV21,KV-Inven,KVW} (see also \cite{PaulJeffrey,Kang18} for the non-contraction property induced by the method).
However, for the proof of Theorem \ref{thm_main}, we significantly improve the method of \(a\)-contraction with shifts.
In particular, we refine the nonlinear Poincar\'e inequality with constraint and introduce a more delicate analysis of the localizing effect given by viscous shock derivatives, as in Section \ref{sec:idea}.
This improvement exhibits robustness that enables it to be applicable in a broader range of contexts.
\end{remark}

\begin{remark}
Theorem \ref{thm_main} provides a contraction estimate for a single viscous shock.
However, a series of papers concerning the isentropic Euler system \cite{KV-Inven,KV-JDE,ChenKangVasseur24arxiv} demonstrate that such contraction estimates can play a central role in proving the stability of Riemann solutions composed of two weak shocks.
Furthermore, these estimates are crucial even when combined with front tracking algorithm to obtain the stability of general small BV solutions.

While extending this from a single shock to small BV solutions in the full Euler system will undoubtedly present new challenges due to its more complicated structure, Theorem \ref{thm_main} represents a cornerstone towards proving the stability of small BV solutions in the class of vanishing physical viscosity limits.
\end{remark}

\section{Ideas of the proof}\label{sec:idea}
\setcounter{equation}{0}

The key ingredient for the proof of Theorem \ref{thm_inviscid} is to obtain a stability estimate uniform with respect to the strength of the viscosity $\nu$. This can be obtained from a contraction estimate of large perturbations only for the case of $\nu=1$, as presented in Theorem \ref{thm_main}.
Indeed, for $U^\nu:=(\vn, \un, \thn)$ solution to \eqref{inveq}, consider the following scaling:
\[
v(t,x)=v^{\nu}(\nu t,\nu x), \qquad u(t,x)=u^{\nu}(\nu t,\nu x),\quad \theta(t,x)=\thn(\nu t,\nu x).
\]
Note that $(v, u, \theta)$ is a solution to \eqref{main} (or equivalently \eqref{inveq} with \(\n=1\)).  
Once we prove the contraction estimate \eqref{cont_main} on the solution $(v, u, \theta)$, thanks to the scaling, we still have the contraction estimate for any $\nu>0$ as follows:
\[
\int_\bbr (a\tht)(x/\nu) \et\big(U^\nu (t,x+X_\nu(t))| \Ut^\nu(x)\big) dx \le \int_\bbr (a\tht)(x/\nu) \et\big(U_0^\nu (x)| \Ut^\nu(x)\big) dx,
\]
where $X_\nu(t):= \nu X(t/\nu)$.
Since $X_\nu'(t)= X'(t/\nu)$, we may use the bound of shifts \eqref{est-shift} to obtain the convergence of $X_\nu$ in $L^1(0,T)$ by the compactness of BV. 
Note that, even if the above perturbation (say $\mathcal{E}^\nu$) is small, the corresponding perturbation $\mathcal{E}$ at the normalized level $\nu=1$ is big, because of $\mathcal{E}=\mathcal{E}^\nu/\nu$.
This is why we should obtain the contraction estimate for any large solutions to \eqref{inveq} with $\nu=1$.  

The above contraction estimate would imply a weak compactness of $U^\nu$, by the well-prepared  initial data \eqref{ini_conv} and $\frac{\theta_\pm}{2}\le a\tht \le 2 \theta_\pm$.
In addition, since $U\mapsto \eta(U|\Ubar)$ is strictly convex and $\Ut^\nu \to \bar U$, we could have the stability estimates \eqref{uni-est} as desired. 

We first prove Theorem \ref{thm_main} by utilizing a carefully refined version of the method of \(a\)-contraction with shifts. 
In Sections \ref{subsec:new1}-\ref{subsec:new4}, we demonstrate the main aspects of our proof that distinguish it from the proof in \cite{KV-Inven} for the isentropic case. On the top of that, the proof of Theorem \ref{thm_inviscid} requires a more delicate approach as in Section \ref{subsec:new5}.

\subsection{Evolution of weighted relative entropy} \label{subsec:new1}
From Lemma \ref{lem-rel}, the time derivative of the weighted relative entropy with the BNSF structure involves three diffusion terms. Note that since one of the main bad terms is \((p(v,\th)-p(\vtil,\thtil))(u-\util)\) and the pressure \(p(v,\th)=R\th/v\) is inversely proportional to the specific volume \(v\), it would be better to express the diffusion on \(v\) variable in the terms of its reciprocal \(1/v\) as \(\Dcal_v\).
In addition, in Lemma \ref{lem-rel}, the diffusion on \(u\) variable is obtained after the original \(u\) diffusion from the momentum equation cancels out with the parabolic bad term from the energy equation.

\subsection{Maximization w.r.t. the velocity variable} \label{subsec:new2}
We should control all the bad terms by the good terms.
Thanks to the \(a\)-contraction framework and the BNSF structure, we now have both hyperbolic good term localized (by the weight derivative) and diffusion term for each variable \(v,u\) and \(\th\), in contrast to the isentropic case in \cite{KV-Inven}.
However, as in Lemma \ref{lem-max}, we still maximize the main hyperbolic bad terms $\int a' (p(v,\th)-p(\vtil,\thtil)) (u-\tilde u) d\x$ with respect to \(u-\util\) to eliminate the dependence on \(u\), in a region that $v$ is not too small or stays near the shock $\tilde v$. 
A motivation for this maximization is from the estimates near the shock as outlined below.

\subsection{Estimates inside truncation} \label{subsec:new3}
The choice of \(\t,\m\) and \(\k\), particularly the necessity of \(\t_0\), is strongly related to estimates near the shock waves $\tilde v$ and $\tilde\theta$.
For the inside of the truncation (denoted by \(\Ubar \coloneqq (\vbar,u,\thbar)\) as in Section \ref{section_hyperbolic} where only $v$ and $\theta$ are truncated), i.e., near $\tilde v$ and $\tilde\theta$, we need to control all leading-order hyperbolic bad terms on \(v\) and \(\th\) variables by the associated good terms \(\Gcal_v,\Gcal_\th,\Dcal_v,\Dcal_\th\).
However, as defined in \eqref{d1vdef}, only \(\Dcal_v^1\) among all the diffusion terms is not intertwined with the other variables, in this case \(u\) and \(\th\), and so it satisfies the monotonicity property \(\Dcal_v^1(\vbar) \le \Dcal_v^1(v)\).
Thus, we should control the hyperbolic bad terms by the three good terms \(\Gcal_v,\Gcal_\th, \Dcal_v^1\).
This is where the constant part \(\t_0\) is essential.
(We could not use other diffusion terms \(\Dcal_v^2\) and \(\Dcal_\th\) at this point because these do not satisfy the monotonicity property.)

In the proof of Proposition \ref{prop:main3}, from the estimate for \((\Ical_1 - \Gcal_v - \Gcal_\th)\), we obtain a new good term \(T(W_1,W_2)\) after a certain normalization with the rescaled variables \(W_1\coloneqq \frac{\l}{\e}v_-^2\left(\frac{1}{v}-\frac{1}{\vtil}\right)\) and \(W_2\coloneqq \frac{\l}{\e}\frac{v_-}{\th_-}(\th-\thtil)\).
Notice that \(T(W_1,W_2)\) has the form of a perfect square of a linear combination of \(W_1\) and \(W_2\).
This good term plays an important role in rewriting all terms of \(W_2\) in terms of \(W_1\) as in \eqref{transquad} and \eqref{transcubic}.
Thanks to that, all the terms are represented in terms of $W_1$ variables only.
Then, the estimate for all the terms boils down to the nonlinear Poincar\'e inequality, as an improved version of the previous one \cite{KV21} (in which, the diffusion coefficient requires $1$ instead of $0.9$):
\begin{eqnarray*}
&&-\frac{1}{\d}\Big(\int_0^1W_1^2\,dy+2\int_0^1 W_1\,dy\Big)^2+(1+\d)\int_0^1 W_1^2\,dy\\
&&\qquad\qquad+\frac{2}{3}\int_0^1 W_1^3\,dy +\d \int_0^1 |W_1|^3\,dy  -(0.9-\d)\int_0^1 y(1-y)|\rd_y W_1|^2\,dy\leq 0.
\end{eqnarray*}
This improvement is crucial, since we do not use all diffusion terms fully but only $\Dcal_v^1$. Indeed, except for the common factor after the normalization, the coefficient of $\Dcal_v^1$ becomes 
\[
\frac{R\t_0}{R(\t_0+(\g+1)\th_-^2) + (\g-1)^2\th_-^2},
\]
which is less than $1$ but is greater than $0.9$ under the assumption \eqref{cold}.

\subsection{Estimates outside truncation} \label{subsec:new4}

Note that the relative entropy functional \(\Phi(v/w)\) grows linearly as \(v\to\infty\) and it grows as \((-\log)\) when \(v\to0+\), provided that \(w\) is bounded.
Due to these (not fast enough) growth rates, we may encounter severe difficulties when dealing with the outside of the truncation.
To remedy this, we may use the diffusion terms, and in this case, it is necessary to exploit the exponential tails of the derivative of the weight (or the shocks) as in \eqref{tail} (or \eqref{der-scale}), and Lemma \ref{lemma_pushing} and Lemma \ref{lemma_Linfty} allow us to do so.
To be specific, Lemma \ref{lemma_pushing} and Lemma \ref{lemma_Linfty} prevent the bad terms from losing the power of the weight derivative. 
Some applications of these lemmas are elaborated below.

As an example, consider $\int |\th-\thbar|^2 d\x$ arising in the estimates for $\abs{\Bcal_1^+(U) - \Ical_1(\vbar,\thbar)}$ which is quadratic in \(\th\) (as in the proof of \eqref{bo1p}).
It could not be controlled by the hyperbolic good term $\Gcal_\th$ for big values of \(\th\).
Hence, we may use $\Dcal_\th$ to control it based on the pointwise estimate \eqref{pwth1}.
In order to use \eqref{pwth1}, we need to show that
\[
\int a'\Big(\int_{\x_0}^\x v\one{|\th-\thbar|>0}d\z\Big)d\x \ll 1.
\]
However, even though we have $v\one{|\th-\thbar|>0} \lesssim \Phi(\frac{v}{\vt})+\Phi(\frac{\th}{\tht})$, showing this smallness is not obvious because \eqref{locE} is the only small quantity we can use.
To resolve this, we crucially utilize the fact that the inner integral is calculated over the interval between \(\x\) and \(\x_0\).
Thanks to the sharp inequality of Lemma \ref{lemma_pushing} together with \eqref{locE}, we can make the above integral small.

Yet the bad terms in \(\Bcal^v(U), \Bcal^u(U)\) and \(\Bcal^\th(U)\), which arise from the parabolic parts, cannot be controlled using the strategy above. 
The reason is that some terms should be controlled by a certain term with high power of \(\th\), which cannot be bounded by \eqref{locE} alone. 
For instance, in the estimates of $\Bcal^u$, we need to control the inequality \eqref{ineq-p4}:
\[
\int_\RR \abs{a'}^2 \frac{\th^2}{v}(\th-\thtil)^2 d\x. 
\]
To prove the inequality \eqref{ineq-p4}, we first establish an \(L^\infty\) bound on \(\abs{a'}\frac{\th^3}{v}\) by combining the pointwise estimates and the sharp inequality from Lemma \ref{lemma_Linfty}.
With this, we could reduce the power of \(\th\) in the integrand which can be bounded by \eqref{locE} now.

\subsection{Estimates for shift in the inviscid limit} \label{subsec:new5}

To ensure the uniqueness of the entropy shock, we need to control the size of the shift \(X(t)\) as stated in \eqref{X-control}.
Note that the strategy used in \cite{KV-Inven} is not available because we can only guarantee weak convergence with the uniform estimate we have, but due to Brenner's correction, the mass equation is no longer a linear equation.
Therefore, there is no equation that the weak limits \((v_\infty, u_\infty, \th_\infty)\) satisfy.  
Thus, we need a different approach.

We first obtain estimates for the viscous solutions \((\vn, \un, \thn)\) from the integral representation \eqref{weak-eq}, which is uniform in the limit \(\n\to 0\).
Then, from the \(L^1\) convergence (hence the pointwise up to subsequence) of \(X_\n\) to \(X_\infty\), we obtain the desired estimate for \(X_\infty\) which guarantees the uniqueness of the solution.
Note that in the process of obtaining the estimates for the viscous solution, we need to handle the \(\th\) dependence of the Brenner coefficient \(\t(\th)\).
To this end, we require an \(L^\infty\) bound on \(\th\), which is established by the contraction estimates as in \eqref{pw-th-final}.

\section{Preliminaries} \label{sec:pre} 
\setcounter{equation}{0}

For simplicity, we rewrite the system (\ref{main}) into the following system, based on the change of variable associated with the speed of propagation of the shock \((t, x)\mapsto (t, \x=x-\s_\e t)\): 
\begin{equation}\label{NS}
\left\{
\begin{aligned}
    &v_t - \s_\e v_\x - u_\x = \Big(\t(\th)\frac{v_\x}{v}\Big)_\x
    \eqqcolon Pv,\\
    &u_t - \s_\e u_\x + p(v, \th)_\x = \Big(\m(\th)\frac{u_\x}{v}\Big)_\x 
    \eqqcolon Pu, \\
    &\frac{R}{\g-1}(\th_t-\s_\e\th_\x) + p(v, \th)u_\x = \Big(\k(\th)\frac{\th_\x}{v}\Big)_\x + \m(\th)\frac{(u_\x)^2}{v}
    \eqqcolon P\th.
\end{aligned}
\right.
\end{equation}
On the other hand, the viscous shock wave $\Ut^{-X}(t,\xi):=\Ut (\xi -X(t))$ shifted by \(-X(t)\) (to be defined in (\ref{X-def})) satisfies the following system of equations: 
\begin{equation}\label{til-system}
\left\{ 
\begin{aligned}
    &(\vt^{-X})_t-\s_\e (\vt^{-X})_\x+\dot{X}(t)(\vt^{-X})_\x-(\ut^{-X})_\x 
    = \Big(\t(\tht^{-X})\frac{(\vt^{-X})_\x}{\vt^{-X}}\Big)_\x 
    \hspace{-2mm}\eqqcolon P\vt^{-X}, \\
    &(\ut^{-X})_t-\s_\e (\ut^{-X})_\x+\dot{X}(t)(\ut^{-X})_\x+p(\vt^{-X},\tht^{-X})_\x 
    = \Big(\mu(\tht^{-X})\frac{(\ut^{-X})_\x}{\vt^{-X}}\Big)_\x
    \hspace{-2mm}\eqqcolon P\ut^{-X}, \\
    &\frac{R}{\g-1}(\tht^{-X})_t- \frac{R\s_\e}{\g-1}(\tht^{-X})_\x+\frac{R}{\g-1}\dot{X}(t)(\tht^{-X})_\x+p(\vt^{-X},\tht^{-X})(\ut^{-X})_\x \\
    &\hspace{4.8cm}
    = \Big(\k(\tht^{-X})\frac{(\tht^{-X})_\x}{\vt^{-X}}\Big)_\x+\m(\tht^{-X})\frac{((\ut^{-X})_\x)^2}{\vt^{-X}}  \eqqcolon P\tht^{-X}.
\end{aligned}
\right.
\end{equation}

\subsection{Global and local estimates on the relative quantities}
In this subsection, we provide a set of useful inequalities on the function \(\Phi\), and these inequalities demonstrate the locally quadratic structure of \(\Phi\) and others.
\begin{lemma} \label{lem-pro}
For given constants \(v_-,\th_->0\), there exist positive constants \(c_1, c_2, c_3, c_4\) such that the following inequalities hold. \\
\begin{enumerate}
    \item For any \(w\in[\frac{1}{2}v_-,2v_-]\),
    \begin{equation} \label{rel_Phi_v}
    \begin{aligned}
    c_1\abs{v-w}^2 &\le \Phif{v}{w} \le c_1^{-1} \abs{v-w}^2 \quad 
    &&\text{for all } v \in \Big(\frac{1}{3}v_-,3v_-\Big), \\
    c_2 \abs{v-w} &\le \Phif{v}{w} \quad 
    &&\text{for all } v \in \RR^+\setminus\Big(\frac{1}{3}v_-,3v_-\Big).
    \end{aligned}
    \end{equation}
    \item For any \(w\in[\frac{1}{2}\th_-,2\th_-]\),
    \begin{equation} \label{rel_Phi_th}
    \begin{aligned}
    c_3\abs{v-w}^2 &\le \Phif{v}{w} \le c_3^{-1} \abs{v-w}^2, \quad 
    &&\text{for all } v \in \Big(\frac{1}{3}\th_-,3\th_-\Big) \\
    c_4\abs{v-w} &\le \Phif{v}{w} \quad 
    &&\text{for all } v \in \RR^+\setminus\Big(\frac{1}{3}\th_-,3\th_-\Big).
    \end{aligned}
    \end{equation}
    \item Moreover, if \(0<w\le u \le v\) or \(0<v \le u \le w\), then
    \begin{equation} \label{Phi-sim}
    \Phif{v}{w} \ge \Phif{u}{w},
    \end{equation} 
    and for any \(\d_* \in (0,w/2)\), there exists a constant \(c_5>0\) such that if, in addition, \(w\in[\frac{1}{2}\min(v_-,\th_-),2\max(v_-,\th_-)]\), \(\abs{w-v} > \d_*\) and \(\abs{w-u} \le \d_*\), then
    \begin{equation} \label{rel_Phi1}
    \Phif{v}{w}-\Phif{u}{w} \ge c_5 \abs{v-u}.
    \end{equation}
\end{enumerate}
\end{lemma}
\begin{proof}
The proof follows essentially the same reasoning as in \cite[Lemma 2.4]{KV21}, despite the differences between the definitions of \(Q(v)\).
It relies on the locally quadratic nature of relative entropy and the fundamental theorem of calculus, so we omit the details.
\end{proof}

\begin{lemma} \label{lem:local}
For given constants \(v_-,\th_->0\), there exists a positive constant \(\d_*\) such that
for any \(\d \in (0,\d_*) \), the following is true. \\
1) For any \(v,w\in\RR^+\) satisfying \(\abs{1/v-1/w} \le \d\) and \(\abs{1/w-1/v_-} \le \d\),
\begin{equation} \label{Phi-est1-v}
\frac{1}{2}\left(\frac{v}{w}-1\right)^2 - \frac{1}{3} \left(\frac{v}{w}-1\right)^3 \le \Phif{v}{w} \le \frac{1}{2}\left(\frac{v}{w}-1\right)^2.
\end{equation}
2) For any \(v,w\in\RR^+\) satisfying \(\abs{v-w} \le \d\) and \(\abs{w-\th_-} \le \d\),
\begin{equation} \label{Phi-est1-th}
\frac{1}{2}\left(\frac{v}{w}-1\right)^2 - \frac{1}{3} \left(\frac{v}{w}-1\right)^3 \le \Phif{v}{w} \le \frac{1}{2}\left(\frac{v}{w}-1\right)^2.
\end{equation}
\end{lemma}
\begin{proof}
Clearly, it is enough to show \eqref{Phi-est1-v}.
Note that the function \(\Phi(z)\) is composed of a polynomial and the natural logarithm, and so it has a Taylor expansion centered at \(1\) and its radius of convergence is \(1\).
Since \(\Phi'(z)=1-1/z\), \(\Phi''(z)=1/z^2\), \(\Phi^{(3)}(z)=-2/z^3\) and \(\Phi^{(4)}(z)=6/z^4\), we have the following expansion of \(\Phi\):
\[
\Phi(z)= \frac{1}{2}(z-1)^2 - \frac{1}{3}(z-1)^3 + \frac{1}{4}(z-1)^4 + l.o.t..
\]
Then, for sufficiently small \(\d\), the conclusion holds immediately.
\end{proof}

\begin{lemma} \label{lem-quad}
For given constants \(v_-,\th_->0\), there exist positive constants \(C\) and \(\d_*\) such that for any \(\d \in (0,\d_*)\), the followings hold:
\begin{enumerate}
    \item For any \(v,w\in\RR^+\) with \(\abs{1/w-1/v_-}\le\d\), if they satisfy either \(\Phi(v/w)<\d\) or \(\abs{1/v-1/w}<\d\), then
    \begin{equation} \label{vquad}
    \abs{v-w}^2 \le C \Phif{v}{w}.
    \end{equation}
    \item For any \(v,w\in\RR^+\) with \(\abs{w-\th_-}\le\d\), if they satisfy either \(\Phi(v/w)<\d\) or \(\abs{v-w}<\d\), then
    \begin{equation} \label{thquad}
    \abs{v-w}^2 \le C \Phif{v}{w}.
    \end{equation}
\end{enumerate}
\end{lemma}
\begin{proof}
It suffices to show \eqref{vquad} only.
For small enough \(\d_*\), it holds that
\[
\min \{c_1 \abs{v-w}^2,c_2 \abs{v-w}\}\le \Phif{v}{w},
\]
and so if \(\Phi(v/w)<\d<\d_*\ll 1\), we have \(\abs{v-w}\ll 1\). 
Then again, for sufficiently small \(\d_*\), from either \(\abs{v-w}\ll 1\) or \(\abs{1/v-1/w}<\d\), we use \eqref{rel_Phi_v} to obtain the desired result.
\end{proof}

\section{Proof of Theorem \ref{thm_main}}\label{section_theo}
\setcounter{equation}{0}

Throughout this and the succeeding section, \(C\) denotes a positive constant which may differ from line to line, but which stays independent on \(\e\) (the shock strength) and \(\l\) (the total variation of the function \(a\)). We will consider two smallness conditions, one on \(\e\), and the other on \(\e/\l\). In the argument, \(\e\) will be far smaller than \(\e/\l\).

\subsection{Properties of small shock waves}
In this subsection, as mentioned above, we state the existence and uniqueness of the viscous shock wave of \eqref{inveq} with small amplitude \(\e\) and present useful properties of the 3-shock waves \((\vt, \ut, \tht)\).
Note that the following properties hold for not only the Brenner-Navier-Stokes-Fourier system but also the Navier-Stokes-Fourier system. (See \cite[Lemma 2.3]{KVW-NSF})
\begin{lemma} \label{lem-VS}
(i) (Existence and Uniqueness) For a given left-end state \((v_-,u_-,\th_-)\in\RR^+\times\RR\times\RR^+\), there exists a constant \(\e_0>0\) such that for any right-end state \((v_+,u_+,\th_+)\in\RR^+\times\RR\times\RR^+\) satisfying \eqref{end-con} and \(\e\coloneqq\abs{v_--v_+}<\e_0\), there is a unique traveling wave solution \((\vt,\ut,\tht):\RR\to\RR^+\times\RR\times\RR^+\) to \eqref{inveq}, as a monotone profile satisfying \eqref{shock_0}.\\
(ii) (Properties) Let \((\vt, \ut, \tht)\) be the 3-shock wave with amplitude \(\e = \abs{v_+-v_-} \hspace{-0.5mm}\sim \abs{u_+-u_-} \hspace{-0.5mm}\sim |\th_+-\th_-|\). Without loss of generality, we may assume that \((\vt, \ut, \tht)\) satisfies \(\vt(0)=\frac{v_-+v_+}{2}\).
Then, there exist positive constants \(C\) such that the followings are true. 
\begin{align}
\begin{aligned} \label{tail}
&\vtil' >0, \quad \util' <0, \quad \thtil' <0 \text{ for all } \x \in \RR, \\
&\abs{(\vtil(\x)-v_-, \util(\x)-u_-, \thtil(\x)-\th_-)} \le C\e e^{-C\e\abs{\x}} \text{ for all } \x \le 0, \\
&\abs{(\vtil(\x)-v_+, \util(\x)-u_+, \thtil(\x)-\th_+)} \le C\e e^{-C\e\abs{\x}} \text{ for all } \x \ge 0, \\
&\abs{(\vtil'(\x), \util'(\x), \thtil'(\x))} \le C\e^2e^{-C\e\abs{\x}} \text{ for all } \x\in \RR, \\
&\abs{(\vtil''(\x), \util''(\x), \thtil''(\x))} \le C\e\abs{(\vtil'(\x), \util'(\x), \thtil'(\x))} \text{ for all } \x\in \RR.
\end{aligned}
\end{align}
Moreover, we have
\begin{equation} \label{lower-v}
\inf_{\x\in \left[-\frac{1}{\e},\frac{1}{\e}\right]} \abs{\vt'(\x)} \ge C\e^2.
\end{equation}
It also holds that \(\abs{\vt'} \sim \abs{\ut'} \sim |\tht'|\) for all \(\x\in\RR\). More explicitly, we have the following:
\begin{align}
\abs{\util'(\x) + \s_*\vtil'(\x)}
+\Big|\thtil'(\x) + \frac{(\g-1)p_-}{R}\vtil'(\x)\Big| &\le C\e\abs{\vtil'(\x)} 
\label{ratio-vu}
\end{align}
where
\[
\s_* \coloneqq \sqrt{\frac{\g p_-}{v_-}} = \frac{\sqrt{R\g  \th_-}}{v_-}
\]
which satisfies
\begin{equation} \label{sm1}
\abs{\s_\e-\s_*} \le C \e.
\end{equation}
In addition, the following estimate holds:
\begin{equation} \label{tail-0}
C^{-1}(v_+-\vt)(\vt-v_-) \le \vt' \le C(v_+-\vt)(\vt-v_-).
\end{equation}
\end{lemma}
\begin{proof}
The existence and above estimates for viscous shocks have been established in \cite{EEKO24} when  \(\t,\m,\k\) are constants.
Following the same arguments in the proof of \cite{EEKO24}, Lemma \ref{lem-VS} can be proved, and so we omit the details and refer to \cite{EE-DBNSF}.
\end{proof}

The estimates \eqref{tail-0} imply not only the estimates \eqref{tail} on the exponential tail of the viscous shock but also the following two Lemmas that are useful outside truncation as described in Section \ref{sec:idea}.  

\begin{lemma}\label{lemma_pushing}
There is a constant \(C>0\) such that for a smooth function \(f\colon \RR\to \RR\) with 
\(
\int_\RR \vtil'(\x) \abs{f(\x)} d\x <\infty,
\)
we have the following: for any \(\x_0\in[-1/\e, 1/\e]\), 
\[
\abs{\int_\RR \vtil'(\x) \int_{\x_0}^{\x} f(\z) d\z d\x} 
\le \frac{C}{\e}\int_\RR \vtil'(\z) \abs{f(\z)} d\z.
\]
\end{lemma}
\begin{proof}
Using Fubini's theorem, we convert the given integral as follows:
\begin{align*}
\abs{\int_\RR \vt'(\x)\int_{\x_0}^\x f(\z) d\z d\x}
&\le \int_\RR\int_\RR \vt'(\z)\abs{f(\z)} \frac{\vt'(\x)}{\vt'(\z)}
\one{\{\x_0\ge\z\ge\x\}\cup\{\x\ge\z\ge\x_0\}} d\x d\z \\
&= \int_{-\infty}^{\x_0} \vt'(\z)\abs{f(\z)}
\int_{-\infty}^{\z} \frac{\vt'(\x)}{\vt'(\z)}d\x d\z 
+ \int_{\x_0}^{\infty} \vt'(\z)\abs{f(\z)}
\int_{\z}^{\infty} \frac{\vt'(\x)}{\vt'(\z)}d\x d\z \\
&= \int_{-\infty}^{\x_0} \vt'(\z)\abs{f(\z)}\frac{\vt(\z)-v_-}{\vt'(\z)}d\z 
+ \int_{\x_0}^{\infty} \vt'(\z)\abs{f(\z)}\frac{v_+-\vt(\z)}{\vt'(\z)}d\z.
\end{align*}
We analyze the two fractions in the integrands:
\(\frac{\vt(\z)-v_-}{\vt'(\z)}\) and
\(\frac{v_+-\vt(\z)}{\vt'(\z)}\).
Then, we will bound the first one for \(\z\le \x_0\), and the second one for \(\z\ge\x_0\).
Note that \(\vt\) is increasing, so these two are all positive. 
From \eqref{tail-0} and \(\vt(0) = \frac{v_-+v_+}{2}\), we have 
\[
(\vt(\z)-v_-) \ge \frac{C^{-1}\vt'(\z)}{v_+-\vt(\z)}
\ge \frac{C^{-1}}{\e} (\vt(\z)-v_-)' \text{ for } \z\le 0,
\]\[
(v_+-\vt(\z)) \ge \frac{C^{-1}\vt'(\z)}{\vt(\z)-v_-}
\ge -\frac{C^{-1}}{\e} (v_+-\vt(\z))' \text{ for } \z\ge 0,
\]
so we use Gronwall's inequality to obtain that 
\begin{align*}
\vt(\z)-v_- \ge e^{-C\e\abs{\z}}(\vt(0)-v_-) = \frac{\e}{2}e^{-C\e\abs{\z}} \text{ for } \z\le 0, \\
v_+-\vt(\z) \ge e^{-C\e\abs{\z}}(v_+-\vt(0)) = \frac{\e}{2}e^{-C\e\abs{\z}} \text{ for } \z\ge 0. 
\end{align*}
Especially, since \(\x_0\in [-1/\e, 1/\e]\), we have 
\[
\vt(\x_0)-v_- \ge \frac{\e}{2}e^{-C}, \qquad
v_+-\vt(\x_0) \ge \frac{\e}{2}e^{-C}.
\]
These lead us to the followings:
\[
\frac{\vt(\z)-v_-}{\vt'(\z)} 
\le \frac{C}{v_+-\vt(\z)} \le \frac{C}{v_+-\vt(\x_0)} 
\le \frac{2Ce^C}{\e} \text{ for } \z\le \x_0,
\]\[
\frac{v_+-\vt(\z)}{\vt'(\z)} 
\le \frac{C}{\vt(\z)-v_-} \le \frac{C}{\vt(\x_0)-v_-} 
\le \frac{2Ce^C}{\e} \text{ for } \z\ge \x_0.
\]
Therefore, the integral above can be bounded as follows:
\begin{align*}
\abs{\int_\RR \vt'(\x)\int_{\x_0}^\x f(\z) d\z d\x}
&\le \int_{-\infty}^{\x_0} \vt'(\z) \abs{f(\z)}
\frac{\vt(\z)-v_-}{\vt'(\z)}d\z
+ \int_{\x_0}^{\infty} \vt'(\z)\abs{f(\z)}
\frac{v_+-\vt(\z)}{\vt'(\z)}d\z \\
&\le \frac{C}{\e}\int_{-\infty}^{\x_0} \vt'(\z)\abs{f(\z)}d\z
+ \frac{C}{\e}\int_{\x_0}^{\infty} \vt'(\z)\abs{f(\z)}d\z 
\end{align*}
which completes the proof.
\end{proof}

\begin{lemma}\label{lemma_Linfty}
There is a constant \(C>0\) such that for any \(\x_0\in[-1/\e, 1/\e]\) and a smooth function \(f\colon \RR\to \RR\) with \(\int_\RR \vt'(\x)\abs{f(\x)}d\x < \infty\), the following holds for every \(\x\in\RR\):  
\[
\abs{\vt'(\x)\int_{\x_0}^{\x} f(\z)d\z} \le C\int_\RR \vt'(\z)\abs{f(\z)}d\z.
\]
\end{lemma}
\begin{proof}
For any \(\x\in \RR\), we have 
\[
\abs{\vt'(\x)\int_{\x_0}^{\x}f(\z)d\z} 
\le \int_{\x_0}^\x \frac{\vt'(\x)}{\vt'(\z)}\vt'(\z)\abs{f(\z)}d\z
\le \sup \frac{\vt'(\x)}{\vt'(\z)}\one{\z\in [\x_0, \x]\cup [\x, \x_0]} \cdot \int_\RR \vt'(\z)\abs{f(\z)}d\z.
\]
Hence, it suffices to show that \(\vt'(\x)/\vt'(\z)\) has an upper bound, for \(\x\) and \(\z\) satisfying \(\x_0 \le \z \le \x\) or \(\x \le \z \le \x_0\). 
Notice that under this condition, we are in one of the following three cases: i) \(\z\in [-1/\e, 1/\e]\), ii) \(1/\e \le \z \le \x\), or iii) \(\x \le \z \le -1/\e\). 
From \eqref{tail-0}, we have 
\[
\frac{\vt'(\x)}{\vt'(\z)} \le C\frac{(v_+-\vt(\x))(\vt(\x)-v_-)}{(v_+-\vt(\z))(\vt(\z)-v_-)}
\]
If \(\z\in [-1/\e, 1/\e]\), then \eqref{tail} and \eqref{lower-v} give \(\frac{\vt'(\x)}{\vt'(\z)} \le Ce^{-C\e\abs{\x}} \le C\). 
If \(1/\e \le \z \le \x\), then \(\e/2 \le \vt(\z)-v_- \le \vt(\x)-v_- \le \e\), and hence, since \(\vt\) is increasing, we have 
\[
\frac{\vt'(\x)}{\vt'(\z)} \le C\frac{v_+-\vt(\x)}{v_+-\vt(\z)} \le C.
\]
Similarly, if \(\x \le \z \le -1/\e\), then \(\e/2 \le v_+-\vt(\z) \le v_+-\vt(\x) \le \e\), and so it holds that
\[
\frac{\vt'(\x)}{\vt'(\z)} \le C\frac{\vt(\x)-v_-}{\vt(\z)-v_-} \le C.
\]
This gives an upper bound of \(\vt'(\x)/\vt'(\z)\) as we desired. This completes the proof. 
\end{proof}

\subsection{Relative entropy method}
Our analysis is based on the relative entropy.
The method is purely nonlinear, and allows us to handle rough and large perturbations.
The relative entropy method was first introduced by Dafermos \cite{Dafermos79} and Diperna \cite{Diperna79} to prove the $L^2$ stability and uniqueness of Lipschitz solutions to the hyperbolic conservation laws endowed with a convex entropy.

The system \eqref{Euler} admits a convex entropy \(\et(U)\coloneqq -R\log v - \frac{R}{\g-1}\log \th\) which comes from the Gibbs relation \(-\th d\eta = de+pdv\). (See \cite[Appendix A]{KVW-NSF})
Then, its derivative is
\begin{equation}\label{nablae}
\nabla\eta(U)=\begin{pmatrix} \displaystyle -\frac{R}{v} & \displaystyle \frac{u}{\th} & \displaystyle -\frac{1}{\th} \end{pmatrix}^T.
\end{equation}
In addition, the relative entropy is given as follows:
\begin{equation} \label{relative_e}
\begin{aligned}
\eta(U|\Ut) &= \eta(U)-\eta(\Ut) - \nabla\eta(\Ut) (U-\Ut)
= R\Phif{v}{\vt} + \frac{R}{\g-1}\Phif{\th}{\tht} + \frac{\abs{u-\ut}^2}{2\tht}.
\end{aligned}
\end{equation}

We consider a weighted relative entropy between the solution \(U\) of \eqref{NS} and the viscous shock $\Ut\coloneqq (\vt, \ut, \tht)$ to \eqref{VS} up to a shift \(X(t)\):
\[
a(\x)\tht(\x)\eta\big(U(t,\x+X(t))|\Ut(\x)\big),
\]
where $a$ is a smooth weight function to be determined later in \eqref{weight-a}.
Note that we put the additional weight $\tilde\theta$ as above, that is motivated by the Gibbs relation (see also \cite{DV}).  \\

In Lemma \ref{lem-rel}, we will derive a quadratic structure on $\frac{d}{dt}\int_{\RR} a(\xi)\eta\big(U(t,\xi+X(t))|\Ut(\xi)\big) d\xi$.\\
For that, we introduce the following notation: for any $f : \RR^+\times \RR\to \RR$ and the shift $X(t)$, 
\[
f^{\pm X}(t, \xi)\coloneqq f(t,\xi\pm X(t)).
\]
We also introduce the function space:
\begin{multline*}
\Xcal\coloneqq \{(v,u,\th) \mid  v^{-1}, v, \th \in L^{\infty}(\RR),~ u-\ut \in L^2(\RR), \\
\rd_\x \Big(\frac{1}{v}-\frac{1}{\vt}\Big), \rd_\x (u-\ut), \rd_\x (\th-\tht) \in L^2(\RR) \},
\end{multline*}
on which the functionals \(Y, \Jcal^{bad},\Jcal^{para}, \Jcal^{good}, \Dcal\) in \eqref{ybg-first} are well-defined for all \(t\in (0,T)\).

\begin{lemma}\label{lem-rel}
Let $a\colon\RR\to\RR^+$ be any positive smooth bounded function whose derivative is bounded and integrable. 
Let $\Ut\coloneqq (\vt, \ut, \tht)$ be the viscous shock in \eqref{shock_0} of size \(\e = \abs{v_+-v_-}\). 
For any solution $U\in \Xcal_T$ to \eqref{NS} with \eqref{pressure} and \eqref{mu-def}, and any absolutely continuous shift $X\colon[0,T]\to\RR$, the following holds. 
\begin{multline}\label{ineq-0}
\frac{d}{dt}\int_{\RR} (a\tht)(\x)\eta(U^X(t,\xi)|\Ut(\xi)) d\xi \\
=\dot X(t) Y(U^X) + \Jcal^{bad}(U^X) + \Jcal^{para}(U^X) - \Jcal^{good}(U^X) - \Dcal(U^X),
\end{multline}
where
\begin{equation}
\begin{aligned}\label{ybg-first}
&Y(U) \coloneqq -\int_\RR a' \tht \et(U|\Ut) \,d\x
- \int_\RR a\Big(R\Phif{v}{\vt}\,d\x
+ \frac{R}{\g-1}\Phif{\th}{\tht}\Big) \tht' \,d\x\\
&\qquad\qquad\qquad
+ \int_\RR a\Big(\frac{R\tht}{\vt^2}(v-\vt)\vt'+ \frac{R}{\g-1}\frac{1}{\tht}(\th-\tht)\tht' + (u-\ut) \ut' \Big) \,d\x\\
&\Jcal^{bad}(U) \coloneqq \int_\RR a' (p-\pt)(u-\ut) \,d\x
- \s_\e\int_\RR a\Big(R\Phif{v}{\vt}+\frac{R}{\g-1}\Phif{\th}{\tht}\Big)\tht' \,d\x \\
&\qquad\qquad\qquad
+ \int_\RR a\frac{1}{\vt}(p-\pt)(v-\vt)\ut' \,d\x\\
&\Jcal^{good}(U) \coloneqq R\s_\e\int_\RR a' \tht\Phif{v}{\vt} \,d\x
+ \frac{R\s_\e}{\g-1}\int_\RR a'\tht\Phif{\th}{\tht} \,d\x
+ \frac{\s_\e}{2}\int_\RR a' (u-\ut)^2 \,d\x\\
&\Dcal(U) \coloneqq 
R \int_\RR a\tht v(\t_0+\th^2)\abs{\Big(\frac{1}{v}-\frac{1}{\vt}\Big)_\x}^2 d\x
+ \int_\RR a\tht\frac{\th}{v}\abs{(u-\ut)_\x}^2 d\x
+ \int_\RR a\frac{\tht}{v}\abs{(\th-\tht)_\x}^2  d\x\\
&\Jcal^{para}(U) \coloneqq \Bcal^v(U) + \Bcal^u(U) + \Bcal^\th(U),
\end{aligned}
\end{equation}
where \(\pt \coloneqq p(\vt,\tht)=R\tht/\vt\).
Here, \(\Jcal^{para}(U)\) is the collection of parabolic bad terms, of which components are defined by 
\begin{equation}
\begin{aligned}\label{RHS-v}
\Bcal^v(U) &\coloneqq 
-R\int_\RR (a\tht)' v(\t_0+\th^2)\Big(\frac{1}{v}-\frac{1}{\vt}\Big) \Big(\frac{1}{v}-\frac{1}{\vt}\Big)_\x d\x\\
&\qquad
-R\int_\RR a\tht(v(\t_0+\th^2)-\vt(\t_0+\tht^2))\Big(\frac{1}{v}-\frac{1}{\vt}\Big)_\x \Big(\frac{1}{\vt}\Big)' d\x\\
&\qquad
-R\int_\RR (a\tht)' (v(\t_0+\th^2)-\vt(\t_0+\tht^2))\Big(\frac{1}{v}-\frac{1}{\vt}\Big) \Big(\frac{1}{\vt}\Big)' d\x\\
&\qquad
-R\int_\RR a\tht v\Big(\frac{1}{v}-\frac{1}{\vt}\Big)^2
\Big((\t_0+\tht^2)\frac{\vt'}{\vt}\Big)' d\x,
\end{aligned}
\end{equation}
\begin{equation}
\begin{aligned}\label{RHS-u}
\Bcal^u(U) &\coloneqq 
- \int_\RR a' \frac{\th^2}{v}(u-\ut) (u-\ut)_\x \,d\x
+2 \int_\RR a \frac{\th}{v}(\th-\tht)(u-\ut)_\x \ut' d\x \\
&\qquad
+2 \int_\RR a(u-\ut) \frac{\th}{v}(\th-\tht)_\x \ut'd\x
+ \int_\RR a(u-\ut) \th^2 \Big(\frac{1}{v}-\frac{1}{\vt}\Big)_\x \ut' d\x\\
&\qquad
+ \int_\RR a(u-\ut) \Big(\frac{\th^2}{v}-\frac{\tht^2}{\vt}\Big) \ut'' d\x
+2 \int_\RR a(u-\ut) \Big(\frac{\th}{v}-\frac{\tht}{\vt}\Big) \tht' \ut' d\x\\
&\qquad
+ \int_\RR a (\th-\tht)\Big(\frac{\th}{v}-\frac{\tht}{\vt}\Big) (\ut')^2 d\x
+ \int_\RR a(u-\ut) (\th^2-\tht^2) \Big(\frac{1}{\vt}\Big)' \ut' d\x,
\end{aligned}
\end{equation}
\begin{equation}
\begin{aligned}\label{RHS-th}
\Bcal^\th(U) &\coloneqq
- \int_\RR a' \frac{\th(\th-\tht)}{v} (\th-\tht)_\x d\x
+3 \int_\RR a \frac{\th-\tht}{v} (\th-\tht)_\x \tht' d\x \\
&\qquad
+ \int_\RR a (\th-\tht)\th \Big(\frac{1}{v}-\frac{1}{\vt}\Big)_\x \tht'd\x
+ \int_\RR a \frac{(\th-\tht)^2}{v} \tht'' d\x\\
&\qquad
+ \int_\RR a \tht(\th-\tht)\Big(\frac{1}{v}-\frac{1}{\vt}\Big) \tht'' d\x
+2 \int_\RR a (\th-\tht) \Big(\frac{1}{v}-\frac{1}{\vt}\Big) (\tht')^2 d\x\\
&\qquad
+ \int_\RR a (\th-\tht)^2 \Big(\frac{1}{\vt}\Big)' \tht'd\x.
\end{aligned}
\end{equation}
\end{lemma}

\begin{remark}
In what follows, we will define the weight function $a$ such that $\s_\e a' >0$. 
Therefore, $-\Jcal^{good}$ consists of three good terms, while $\Jcal^{bad}$ consists of bad terms. 
\end{remark}

\begin{proof}
See Appendix \ref{appendix_RHS}.
\end{proof}

\subsection{Construction of the weight function}
We define the weight function $a$ by
\begin{equation}\label{weight-a}
a(\xi)=1 + \lambda \frac{\vt(\x)-v_-}{[v]},
\end{equation}
where \([v]\coloneqq v_+ - v_- \).
The weight function \(a\) possesses some useful properties. First of all, it is positive and increasing, and satisfies \(1 \le a \le 1+\l \).
From \([v]=\e\), we can compare the scale of derivatives of the weight function and viscous shocks. For \(\vt\), we have
\begin{equation}\label{der-a}
a'(\x)= \frac{\l}{\e} \vt'(\x),
\end{equation}
and from \eqref{ratio-vu} and \eqref{tail}, we conclude that
\begin{equation}\label{der-scale}
\begin{gathered}
|a'| \sim \frac{\l}{\e}|\vt'| \sim \frac{\l}{\e}|\ut'| \sim \frac{\l}{\e} |\tht'|, \quad \text{and} \quad
\abs{a'(\x)} \le C\e\l e^{-C\e\abs{\x}} \quad \text{ for all } \x\in\RR.
\end{gathered}
\end{equation}

\subsection{Maximization in terms of $u-\ut$}\label{sec:mini}

In order to estimate the right-hand side of \eqref{ineq-0}, we will use Proposition \ref{prop:main3}, i.e., a sharp estimate with respect to \(1/v-1/\vt\) and \(\th-\tht\) when \(\abs{1/v-1/\vt} \ll 1\) and \(|\th-\thtil| \ll 1\), for which we need to rewrite $\Jcal^{bad}$ on the right-hand side of \eqref{ineq-0} only in terms of \(v\) and \(\th\) near \(\vt\) and \(\tht\), respectively, by separating \(u-\ut\) from the first term of $\Jcal^{bad}$. 
Therefore, we will rewrite $\Jcal^{bad}$ into the maximized representation in terms of \(u-\ut\) in the following lemma. 
However, we will keep all terms of $\Jcal^{bad}$ in a region \(\{1/v>1/\vt + \d\}\) for small values of $v$, since we use the estimate \eqref{bo1m} to control the first term of \(\Jcal^{bad}\) in that region.     
Notice that we have not set the value of \(\d\) yet.

\begin{lemma}\label{lem-max}
Let $a\colon \RR\to\RR^+$ be as in \eqref{weight-a}, and $\Ut=(\vt, \ut, \tht)$ be the viscous shock in \eqref{shock_0} of size \(\e = \abs{v_+-v_-}\). Let $\d$ be any positive constant.
Then, for any $U\in \Xcal$, 
\begin{equation}\label{ineq-1}
\Jcal^{bad} (U) -\Jcal^{good} (U)= \Bcal_\d(U)- \Gcal_\d(U),
\end{equation}
where
\begin{equation}
\begin{aligned}\label{badgood}
\Bcal_\d(U) &\coloneqq \frac{1}{2\s_\e}\int_\RR a'(p-\pt)^2\one{1/v\le 1/\vt+\d} d\x 
+\int_\RR a'(p-\pt)(u-\ut)\one{1/v>1/\vt+\d} d\x\\
&\qquad
-\int_\RR av(p-\pt)\Big(\frac{1}{v}-\frac{1}{\vt}\Big)\ut'd\x
-\s_\e\int_\RR a\Big(R\Phif{v}{\vt}+\frac{R}{\g-1}\Phif{\th}{\tht}\Big)\tht'd\x, \\
\Gcal_\d(U) &\coloneqq \frac{\s_\e}{2}\int_\RR a'\Big(u-\ut-\frac{p-\pt}{\s_\e}\Big)^2\one{1/v\le 1/\vt+\d}d\x
+\frac{\s_\e}{2}\int_\RR a'(u-\ut)^2 \one{1/v>1/\vt+\d}d\x \\
&\qquad
+R\s_\e\int_\RR a'\tht\Phif{v}{\vtil}d\x + \frac{R\s_\e}{\g-1}\int_\RR a'\tht\Phif{\th}{\tht}d\x.
\end{aligned}
\end{equation}
\begin{remark}\label{rem:0}
Since $\s_\e a' >0$, $-\Gcal_\d$ consists of four good terms. 
\end{remark}
\end{lemma}

\begin{proof}
The proof is very simple and general, using the identity \(\alpha x^2+ \beta x =\alpha(x+\frac{\beta}{2\alpha})^2-\frac{\beta^2}{4\alpha}\), and is essentially the same as \cite[Lemma 4.3]{KV-Inven}.
Thus, we omit it.
\end{proof}

\subsection{Main proposition}
The main proposition for the proof of Theorem \ref{thm_main} is as follows.

\begin{proposition}\label{prop:main}
There exist $\e_0,\d_0, \d_3\in(0,1/2)$ such that for any $\e<\e_0$ and $\d_0^{-1}\e<\lambda<\d_0$, the following is true.\\
For any $U\in\Xcal \cap \{U \mid \abs{Y(U)}\le\e^2 \}$,
\begin{equation}
\begin{aligned}\label{prop:est}
\Rcal(U)&\coloneqq -\frac{1}{\e^4}Y^2(U) +\Bcal_{\d_3}(U)+\d_0\frac{\e}{\l} \abs{\Bcal_{\d_3}(U)} + \d_0\frac{\e}{\l}\Bcal_1^+(U) + \Jcal^{para}(U) + \d_0\abs{\Jcal^{para}(U)} \\
&\qquad -\Gcal_u^-(U)-\frac{1}{2}\Gcal_u^+(U) -\Big(1-\d_0\frac{\e}{\l}\Big)(\Gcal_v(U)+\Gcal_\th(U)) -(1-\d_0)\Dcal(U) \le 0,
\end{aligned}
\end{equation}
where \(Y\), \(\Bcal_{\d_3}\), \(\Jcal^{para}\), and \(\Dcal\) are as in \eqref{ybg-first} and \eqref{badgood}, \(\Bcal_1^+\) denotes 
\[
\Bcal_1^+(U) \coloneqq \frac{1}{2\s_\e}\int_\RR a'(p-\pt)^2 \one{1/v \le 1/\vt + \d_3} d\x,
\]
and \(\Gcal_u^-, \Gcal_u^+, \Gcal_v\) and \(\Gcal_\th\) denote the four terms of \(\Gcal_{\d_3}\) as follows:
\begin{equation}
\begin{aligned}\label{ggd}
&\Gcal_u^-(U)\coloneqq \frac{\s_\e}{2}\int_\RR a'(u-\ut)^2 \one{1/v>1/\vt+\d_3} d\x,\\
&\Gcal_u^+(U)\coloneqq\frac{\s_\e}{2}\int_\RR a'\Big(u-\ut-\frac{p-\pt}{\s_\e}\Big)^2\one{1/v\le 1/\vt+\d_3} d\x,\\
&\Gcal_v(U)\coloneqq R\s_\e\int_\RR a'\tht\Phif{v}{\vt} d\x,
\qquad \Gcal_\th(U)\coloneqq \frac{R\s_\e}{\g-1}\int_\RR a'\tht\Phif{\th}{\tht} d\x.\\
\end{aligned}
\end{equation}

\end{proposition}

\subsection{Proof of Theorem~\ref{thm_main} from Proposition \ref{prop:main}}
In this subsection, we show how Proposition \ref{prop:main} implies Theorem \ref{thm_main}.\\

For any fixed $\e>0$, we consider a continuous function $\Phi_\e$ defined by
\begin{equation}\label{Phi-d}
\Phi_\e (y)=
\left\{ \begin{array}{ll}
      \frac{1}{\e^2},\quad \mbox{if}~ y\le -\e^2, \\
      -\frac{1}{\e^4}y,\quad \mbox{if} ~ \abs{y}\le \e^2, \\
       -\frac{1}{\e^2},\quad \mbox{if}  ~y\ge \e^2. \end{array} \right.
\end{equation}
Let $\e_0,\d_0, \d_3$ be the constants in Proposition \ref{prop:main}. Then, let $\e, \l$ be any constants such that $0<\e<\e_0$ and $\d_0^{-1}\e<\l<\d_0<1/2$.\\

We define a shift function $X(t)$ as a solution of the following nonlinear ODE:
\begin{equation}\label{X-def}
\left\{ \begin{array}{ll}
    \dot X(t) = \Phi_\e (Y(U^X)) \Big(2|\Jcal^{bad}(U^X)|+2\abs{\Jcal^{para}(U^X)}+1\Big),\\
X(0)=0, \end{array} \right.
\end{equation}
where $Y$, $\Jcal^{bad}$ and \(\Jcal^{para}\) are as in \eqref{ybg-first}.\\
Then, for the solution \(U\in \Xcal_T\), there exists a unique absolutely continuous shift $X$ on $[0,T]$. 
Indeed, since \(\vt', \ut', \tht', a'\) are bounded, smooth and integrable, using $U\in \Xcal_T$ with the change of variables \(\x\mapsto \x-X(t)\) as in \eqref{move-X}, we find that there exist \(a, b\in L^2(0,T)\) such that
\[
\sup_{x\in\RR} \abs{F(t,x)} \le a(t)\quad\mbox{and}\quad \sup_{x\in\RR} \abs{\rd_x F(t,x)} \le b(t),\quad \forall t\in[0,T],
\]
where $F(t,X)$ denotes the right-hand side of the ODE in \eqref{X-def}. For more details on the existence and uniqueness theory of the ODE, we refer to \cite[Lemma A.1]{CKKV}. \\


Based on \eqref{ineq-0} and \eqref{X-def}, to get the contraction estimate \eqref{cont_main}, it is enough to prove that for almost every time \(t>0\), 
\begin{multline}\label{contem0}
\Phi_\e(Y(U^X))\left(2|\Jcal^{bad}(U^X)|+2\abs{\Jcal^{para}(U^X)}+1\right)Y(U^X) \\
+\Jcal^{bad}(U^X)+\Jcal^{para}(U^X)-\Jcal^{good}(U^X)-\Dcal(U^X) \le 0.
\end{multline}
For any \(U\in\Xcal\), we define 
\begin{multline*}
\Fcal(U) \coloneqq \Phi_\e(Y(U))\left(2|\Jcal^{bad}(U)|+2\abs{\Jcal^{para}(U)}+1\right)Y(U) \\
+\Jcal^{bad}(U)+\Jcal^{para}(U)-\Jcal^{good}(U)-\Dcal(U).
\end{multline*}
From \eqref{Phi-d}, we have 
\begin{equation}\label{XY}
\Phi_\e(Y)\left(2|\Jcal^{bad}|+2\abs{\Jcal^{para}}+1\right)Y
\le \begin{cases}
    -2|\Jcal^{bad}| -2\abs{\Jcal^{para}}, & \text{ if } \abs{Y}\ge \e^2, \\
    -\frac{1}{\e^4}Y^2, & \text{ if } \abs{Y}\le \e^2.
\end{cases}
\end{equation}
Hence, for all \(U\in\Xcal\) satisfying \(\abs{Y(U)}\ge \e^2\), we have 
\[
\Fcal(U) \le -|\Jcal^{bad}(U)|-\abs{\Jcal^{para}(U)}-\Jcal^{good}(U)-\Dcal(U)
\le 0. 
\]
Using \eqref{ineq-1}, \eqref{XY} and Proposition \ref{prop:main}, we find that for all \(U\in\Xcal\) satisfying \(\abs{Y(U)}\le\e^2\), 
\begin{align*}
\Fcal(U)&\le -\d_0\frac{\e}{\l}\left(\abs{\Bcal_{\d_3}(U)}+\Bcal_1^+(U)\right)
-\d_0\abs{\Jcal^{para}(U)} \\
&\qquad-\frac{1}{2}\Gcal_u^+(U)
-\d_0\frac{\e}{\l}(\Gcal_v(U)+\Gcal_\th(U))-\d_0\Dcal(U) \le 0.
\end{align*}
Since \(\d_0<1/2\), these two estimates together with the definition of \(\Jcal^{good}\) show that for every \(U\in\Xcal\), 
\begin{multline*}
\Fcal(U) \le -|\Jcal^{bad}(U)|\one{\abs{Y(U)}\ge \e^2}
-\d_0\frac{\e}{\l}\left(\abs{\Bcal_{\d_3}(U)}+\Bcal_1^+(U)\right)\one{\abs{Y(U)}\le \e^2} \\
-\d_0\abs{\Jcal^{para}(U)}
-\frac{1}{2}\Gcal_u^+(U)\one{\abs{Y(U)}\le \e^2}
-\d_0\frac{\e}{\l}(\Gcal_v(U)+\Gcal_\th(U))-\d_0\Dcal(U).
\end{multline*}
Thus, using the above estimate with \(U = U^X\), together with \eqref{ineq-0} and \eqref{contem0}, we find that for a.e. \(t>0\), 
\begin{equation} \label{111}
\begin{aligned}
&\frac{d}{dt}\int_\RR a\et(U^X|\Ut)d\x + \d_0\frac{\e}{\l}(\Gcal_v(U^X)+\Gcal_\th(U^X))+\d_0\Dcal(U^X) \\
&= \Fcal(U^X)+ \d_0\frac{\e}{\l}(\Gcal_v(U^X)+\Gcal_\th(U^X))+\d_0\Dcal(U^X) \\
&\le -|\Jcal^{bad}(U^X)|\one{\abs{Y(U^X)}\ge \e^2} - \d_0\abs{\Jcal^{para}(U^X)}
 -\frac{1}{2}\Gcal_u^+(U^X)\one{\abs{Y(U^X)}\le \e^2} \\
&\qquad-\d_0\frac{\e}{\l}\left(\abs{\Bcal_{\d_3}(U^X)}+\Bcal_1^+(U^X)\right)\one{\abs{Y(U^X)}\le \e^2}
\le  0.
\end{aligned}
\end{equation}
Therefore, we have 
\begin{equation}\label{cont-pre}
\int_\RR a\et(U^X|\Ut)d\x + \d_0\frac{\e}{\l}\int_0^t \Gcal_v(U^X)+\Gcal_\th(U^X) ds
+ \d_0\int_0^t \Dcal(U^X) ds
\le \int_\RR a\et(U_0|\Ut)d\x < \infty,
\end{equation}
which completes \eqref{cont_main}. 

To estimate \(\dot{X}\), we first observe that \eqref{Phi-d} and \eqref{X-def} yield 
\begin{equation}\label{contx}
|\dot{X}|\le \frac{1}{\e^2}\left(2|\Jcal^{bad}(U^X)|+2\abs{\Jcal^{para}(U^X)}+1\right), \qquad \text{ for a.e. } t\in (0, T).
\end{equation}
It follows from \eqref{111} and \(1 \le a\le 3/2\) by \(\d_0<1/2\) that 
\begin{multline}\label{jb-cont}
\int_0^T \Big(|\Jcal^{bad}(U^X)|\one{\abs{Y(U^X)}\ge \e^2} +\d_0\abs{\Jcal^{para}(U^X)}
+\Gcal_u^+(U^X)\one{\abs{Y(U^X)}\le \e^2} \\
+ \d_0\frac{\e}{\l}\left(\abs{\Bcal_{\d_3}(U^X)}+\Bcal_1^+(U^X)\right)\one{\abs{Y(U^X)}\le \e^2}\Big)dt 
\le 2\int_\RR \et(U_0|\Ut) d\x.
\end{multline}
To estimate \(|\Jcal^{bad}(U^X)|\) globally in time, using \eqref{ineq-1}, the definitions of \(\Jcal^{good}\) and \(\Gcal_{\d_3}\), \(\abs{a'}\le C\d_0\) and \(\d_0\le1/2\le a\), we find that
\begin{align*}
\begin{aligned}
&|\Jcal^{bad}(U^X)|=|\Jcal^{bad}(U^X)| \one{\abs{Y(U^X)}\ge \e^2} +|\Jcal^{bad}(U^X)| \one{\abs{Y(U^X)}\le \e^2}  \\
&=|\Jcal^{bad}(U^X)| \one{\abs{Y(U^X)}\ge \e^2} + |\Jcal^{good}(U^X) + \Bcal_{\d_3}(U^X)-\Gcal_{\d_3}(U^X)| \one{\abs{Y(U^X)}\le \e^2}\\
&\le |\Jcal^{bad}(U^X)|\one{\abs{Y(U^X)}\ge \e^2} + \abs{\Bcal_{\d_3}(U^X)}\one{\abs{Y(U^X)}\le \e^2} \\
&\qquad + \frac{\s_\e}{2}\int_{\{1/v\le 1/\vt+\d_3\}} a' \Big(\big(u^X-\ut \big)^2 - \Big(u^X-\ut-\frac{p^X-\pt}{\s_\e}\Big)^2\Big) d\x \, \one{\abs{Y(U^X)}\le \e^2} \\
&\le |\Jcal^{bad}(U^X)|\one{\abs{Y(U^X)}\ge \e^2} + \abs{\Bcal_{\d_3}(U^X)}\one{\abs{Y(U^X)}\le \e^2} \\
&\qquad  +C \int_{\{1/v\le 1/\vt+\d_3\}} a' \Big( \Big(u^X-\ut-\frac{p^X-\pt}{\s_\e}\Big)^2 + \big(p^X-\pt\big)^2 \Big) d\x \, \one{\abs{Y(U^X)}\le \e^2} \\
&\le |\Jcal^{bad}(U^X)|\one{\abs{Y(U^X)}\ge \e^2} + \Big(\abs{\Bcal_{\d_3}(U^X)} + C\Bcal_1^+(U^X) + C\Gcal_u^+(U^X)\Big)\one{\abs{Y(U^X)}\le \e^2}.
\end{aligned}
\end{align*}
This together with \eqref{cont-pre} and \eqref{contx} implies that
\begin{multline*}
|\dot {X}| \le \frac{1}{\e^2}\Big[ 2\Big( |\Jcal^{bad}(U^X)|\one{\abs{Y(U^X)}\ge \e^2}\Big)
+ 2\abs{\Jcal^{para}(U^X)} + 1\\
+ 2\left(\abs{\Bcal_{\d_3}(U^X)}+ C \Bcal_1^+(U^X) + C\Gcal_u^+(U^X)\right)\one{\abs{Y(U^X)}\le \e^2} 
\Big].
\end{multline*}
Moreover, it follows from \eqref{jb-cont} with \(\e<\e_0\) and \(\d_0^{-1}\e < \l < \d_0 < 1/2\) that
\begin{multline*}
\int_0^{T} \Big(|\Jcal^{bad}(U^X)|\one{\abs{Y(U^X)}\ge \e^2} + \abs{\Jcal^{para}(U^X)} \\
+ \left(\abs{\Bcal_{\d_3}(U^X)}+ \Bcal_1^+(U^X) + \Gcal_u^+(U^X)\right)\one{\abs{Y(U^X)}\le \e^2}\Big)dt 
\le C\frac{\l}{\d_0\e}\int_\RR \et(U_0|\Ut) d\x,
\end{multline*}
for some constant \(C>0\) does not depend on \(\d_0,\e/\l\) and \(T\).
This ends the proof of \eqref{est-shift}. \qed


\vskip0.1cm The next section is dedicated to the proof of Proposition \ref{prop:main}.

\section{Proof of Proposition \ref{prop:main}}
\setcounter{equation}{0}
\subsection{A nonlinear Poincar\'e type inequality}
For any \(\d>0\), and any function \(W\in L^2(0, 1)\) such that \(\sqrt{y(1-y)}\rd_y W\in L^2(0, 1)\), we define 
\begin{align*}
\Rcal_\d(W) &\coloneqq -\frac{1}{\d}\Big(\int_0^1 W^2 dy + 2\int_0^1 W dy\Big)^2 
+ (1+\d)\int_0^1 W^2 dy \\
&\qquad
+ \frac{2}{3}\int_0^1 W^3 dy + \d\int_0^1\abs{W}^3 dy - (0.9-\d)\int_0^1 y(1-y)\abs{\rd_y W}^2 dy.
\end{align*}
We have improved \cite[Proposition 3.3]{KV21} by the following version, as mentioned in Section \ref{sec:idea}.
\begin{proposition}\label{prop_nl_Poincare}
For a given \(C_1 >0\), there exists \(\d_2>0\), such that for any \(\d\le\d_2\) the following is true. \\
For any \(W\in L^2(0, 1)\) such that \(\sqrt{y(1-y)}\rd_y W\in L^2(0, 1)\), if \(\int_0^1 \abs{W(y)}^2 dy\le C_1\), then 
\[
\Rcal_\d(W) \le 0. 
\]
\end{proposition}
\begin{proof}
See Appendix \ref{appendix_prop_nl_Poincare}.
\end{proof}

\subsection{Inside truncation}\label{section-expansion}

We define the following functionals:
\begin{equation} \label{note-in}
\begin{split}
&\begin{aligned} 
&\Ical_g^Y(v,\th) \coloneqq
-R \int_\RR a' \tht \Phif{v}{\vt}d\x -\frac{R}{\g-1} \int_\RR a'\tht \Phif{\th}{\tht} d\x
- \frac{1}{2\s_\e^2}\int_\RR a' (p-\pt)^2 d\x \\
&\hspace{20mm}
+R \int_\RR a \frac{\tht}{\vt^2}(v-\vt) \vt' d\x + \frac{R}{\g-1}\int_\RR a \frac{\th-\tht}{\tht} \tht' d\x
+ \frac{1}{\s_\e} \int_\RR a (p-\pt) \ut' d\x \\
&\hspace{20mm}
-R \int_\RR a \Phif{v}{\vt}\tht' d\x - \frac{R}{\g-1}\int_\RR a \Phif{\th}{\tht} \tht' d\x, 
\end{aligned} \\
&\begin{aligned}
&\Ical_1(v,\th) \coloneqq 
\frac{1}{2\s_\e} \int_\RR a' (p-\pt)^2 d\x,
&&\Ical_2(v,\th) \coloneqq
\int_\RR a (p-\pt)\frac{(v-\vt)}{\vt} \ut' d\x, \\
&\Ical_3(v,\th) \coloneqq
-R\s_\e \int_\RR a \Phif{v}{\vt}\tht' d\x, 
&&\Ical_4(v,\th) \coloneqq
-\frac{R\s_\e}{\g-1} \int_\RR  a \Phif{\th}{\tht} \tht' d\x,\\
&\Gcal_v(v) \coloneqq
R\s_\e \int_\RR a' \tht \Phif{v}{\vt} d\x,
&&\Gcal_\th(\th) \coloneqq
\frac{R\s_\e}{\g-1} \int_\RR a' \tht \Phif{\th}{\tht} d\x,
\end{aligned}\\
&\Dcal_v^1(v) \coloneqq
R \t_0 \int_\RR a \tht v \abs{\Big(\frac{1}{v}-\frac{1}{\vt}\Big)_\x}^2 d\x.
\end{split}
\end{equation}
Note that all of these quantities depend only on \(v\) and \(\th\), not on \(u\).

\begin{proposition}\label{prop:main3}
For any constant \(C_2>0\), there exist \(\e_0, \d_3>0\), such that for any \(\e, \l\in(0,\e_0)\) and \(\d\in(0,\d_3)\) with \(\e/\l\in (0, \e_0)\), the following is true.\\
For any function \(v\colon\RR\to \RR^+\) and \(\th\colon\RR\to \RR^+\) such that \(\Dcal_v^1(v)+\Gcal_v(v)+\Gcal_\th(\th)\)
is finite, if
\begin{equation}\label{assYp}
\abs{\Ical_g^Y(v, \th)} \le C_2 \frac{\e^2}{\l},\qquad  \norm{\frac{1}{v}-\frac{1}{\vt}}_{L^\infty(\RR)}\le \d_3, \qquad \|\th-\tht\|_{L^\infty(\RR)}\le \d_3,
\end{equation}
then
\begin{align}
\begin{aligned}\label{redelta}
\Rcal_{\e,\d}(v, \th)&\coloneqq -\frac{1}{\e\d}\abs{\Ical_g^Y(v, \th)}^2 +\Ical_1(v, \th)+\d\left(\frac{\e}{\l}\right)\abs{\Ical_1(v, \th)}\\
&\quad\quad+(\Ical_2(v, \th)+\Ical_3(v, \th)+\Ical_4(v, \th))+\d(\abs{\Ical_2(v, \th)}+\abs{\Ical_3(v, \th)}+\abs{\Ical_4(v, \th)}) \\
&\quad\quad-\Big(1-\d\Big(\frac{\e}{\l}\Big)\Big)(\Gcal_v(v)+\Gcal_\th(\th))-(1-\d)\Dcal_v^1(v)\le 0.
\end{aligned}
\end{align}
\end{proposition}

\begin{proof}
To ensure the existence of a viscous shock, we choose \(\e_0\) to be smaller than the \(\e_0\) in Lemma \ref{lem-VS}, and we impose on \(\d_3\) that \(\d_3 < \d_*\) where \(\d_*\) is the constant in Lemma \ref{lem:local}.

We rewrite the above functionals  with respect to the following variables
\[
    y \coloneqq \frac{\vt(\x)-v_-}{[v]}, \,\, w_1 \coloneqq v_-^2 \Big(\frac{1}{v}-\frac{1}{\vt}\Big) \circ y^{-1}, \,\, w_2 \coloneqq \frac{v_-}{\th_-}(\th-\tht)\circ y^{-1}, \,\,
    W_1 \coloneqq \frac{\l}{\e} w_1, \,\, W_2 \coloneqq \frac{\l}{\e} w_2.
\]
Notice that \(\vt(\x)\) is increasing in \(\x\), we use a change of variable \(\x\in\RR \mapsto y\in[0,1]\).
Then, it follows from \eqref{weight-a} that \(a=1+\l y\) and
\begin{equation} \label{change-d}
a'(\x)=\frac{\l}{[v]}\vt'(\x)=\frac{\l}{\e}\vt'(\x), \qquad \frac{dy}{d\x}=\frac{\vt'}{[v]}, \qquad \abs{a-1}\le \l.
\end{equation}

\(\bullet\) \textbf{Change of variable for \(\Ical_g^Y\):} 
We refer to the eight terms of \(\Ical_g^Y\) as \(I_1\) through \(I_8\), and analyze them individually in what follows.
Using \eqref{assYp} and \eqref{change-d} with \eqref{Phi-est1-v} and \eqref{Phi-est1-th},
since \(|\tht(\x)-\th_-| \le C \e\) for all \(\x\), we have
\begin{align*}
\abs{Y_1 + Y_2 + \frac{\l R \th_-}{2v_-^2} \int_0^1 w_1^2 + \frac{w_2^2}{\g-1} dy} \le C (\e_0+\d_3) \l \Big(\int_0^1 w_1^2 dy + \int_0^1 w_2^2 dy\Big).
\end{align*}
For \(Y_3\), it is required to rewrite \(p-\pt\) as a linear combination of \(\left(\frac{1}{v}-\frac{1}{\vt}\right)\) and \(\th-\tht\):
\[
p-\pt = \frac{R}{v}(\th-\tht) + R\tht \Big(\frac{1}{v}-\frac{1}{\vt}\Big),
\]
which together with \eqref{sm1} gives
\begin{align*}
\abs{Y_3+\frac{\l R\th_-}{2\g v_-^2} \int_0^1 (w_1+w_2)^2 dy} \le C (\e_0+\d_3)\l \Big(\int_0^1 w_1^2 dy+\int_0^1 w_2^2 dy\Big).
\end{align*}
We then compute the linear terms \(Y_4,Y_5\) and \(Y_6\) with \eqref{ratio-vu}: using \eqref{change-d} and for \(Y_6\), using the same linear combination for \(p-\ptil\) as \(Y_3\), we find
\begin{align*}
\abs{Y_4 + Y_5 + Y_6 + \frac{2\e R \th_-}{v_-^2}\int_0^1 w_1+w_2 dy} \le C(\e_0+\d_3) \e \Big(\int_0^1 \abs{w_1}dy + \int_0^1 \abs{w_2}dy\Big).
\end{align*}
It remains to compute \(Y_7\) and \(Y_8\): using \eqref{Phi-est1-v}, \eqref{Phi-est1-th}, and \eqref{ratio-vu} again, we obtain
\begin{align*}
\abs{Y_7+Y_8+\frac{\e R\th_-}{2v_-^3}\int_0^1 (\g-1)w_1^2+w_2^2 dy} \le C(\e_0+\d_3)\e \Big(\int_0^1 w_1^2 dy+\int_0^1 w_2^2 dy\Big).
\end{align*}
Then, combining all the terms on \(\Ical_g^Y\) and writing it in terms of \(W_1\) and \(W_2\), we obtain
\begin{align}
\begin{aligned} \label{insidey}
&\abs{\frac{\l}{\e^2}\frac{v_-^2}{R\th_-}\Ical_g^Y 
+ \frac{1}{2} \int_0^1 W_1^2 +\frac{W_2^2}{\g-1} dy
+ \frac{1}{2\g} \int_0^1 (W_1+W_2)^2 dy + 2\int_0^1 (W_1+W_2) dy} \\
&\qquad \le
C(\e_0+\d_3) \left( \int_0^1 (W_1^2 + W_2^2) dy + \int_0^1 \abs{W_1} + \abs{W_2} dy\right).
\end{aligned}
\end{align}
Notice that \(Y_7\) and \(Y_8\) have smaller scale compared to \(Y_1, Y_2\) and \(Y_3\).

\(\bullet\) \textbf{Change of variable for \(\Ical_1 - \Gcal_v - \Gcal_\th\):}
To handle this, we write down altogether:
\begin{equation} \label{main-cancel}
\Ical_1 -\Gcal_v -\Gcal_\th 
= \frac{1}{2\s_\e} \int_\RR a' \Big[(p-\pt)^2 -2R\s_\e^2 \tht \Phif{v}{\vt} -\frac{2R\s_\e^2 \tht}{\g-1}\Phif{\th}{\tht}\Big] d\x
\end{equation}
Here, we need to express \((p-\pt)\) in the following way:
\[
p-\pt=\frac{R}{\vt}(\th-\tht)+R\tht\Big(\frac{1}{v}-\frac{1}{\vt}\Big)+R\Big(\frac{1}{v}-\frac{1}{\vt}\Big)(\th-\tht),
\]
and so, we get
\begin{equation}
\begin{aligned}\label{b1expansion}
&(p-\pt)^2
=\frac{R^2}{\vt^2}(\th-\tht)^2
+R^2\tht^2\Big(\frac{1}{v}-\frac{1}{\vt}\Big)^2
+\frac{2R^2\tht}{\vt}\Big(\frac{1}{v}-\frac{1}{\vt}\Big)(\th-\tht) \\
&\,
+\frac{2R^2}{\vt}\Big(\frac{1}{v}-\frac{1}{\vt}\Big)(\th-\tht)^2
+2R^2\tht\Big(\frac{1}{v}-\frac{1}{\vt}\Big)^2(\th-\tht)
+C\d_3 \Big(\Big|\frac{1}{v} - \frac{1}{\vt}\Big|^3 + |\th - \tht|^3 \Big).
\end{aligned}
\end{equation}
On the other hand, from \eqref{Phi-est1-v}-\eqref{Phi-est1-th}, we have
\begin{equation}
\begin{aligned}\label{gexpansion}
&2R\s_\e^2 \tht \Phif{v}{\vt} +\frac{2R\s_\e^2 \tht}{\g-1}\Phif{\th}{\tht} \\
&= R\s_\e^2 \vt^2 \tht \Big(\frac{1}{v}-\frac{1}{\vt}\Big)^2 
+\frac{2}{3} R\s_\e^2 \vt^3 \tht \Big(\frac{1}{v}-\frac{1}{\vt}\Big)^3
-2R\s_\e^2 \vt^3 \tht\Big(\frac{1}{v}-\frac{1}{\vt}\Big)^3 \\
&\qquad
+\frac{R\s_\e^2}{\g-1}\frac{1}{\tht}(\th-\tht)^2 
-\frac{2R\s_\e^2}{3(\g-1)}\frac{1}{\tht^2}(\th-\tht)^3
+C\d_3 \Big(\Big|\frac{1}{v} - \frac{1}{\vt}\Big|^3+|\th-\tht|^3 \Big).
\end{aligned}
\end{equation}
Then, we gather all the quadratic terms in \eqref{main-cancel} with \(\Ocal(\e)\) error:
\[
- \frac{R^2\th_-^2}{2\s_*v_-^4} \Big((\g-1)w_1^2 - 2 w_1 w_2 + \frac{1}{\g-1} w_2^2\Big) +C\e \left(w_1^2 + w_2^2 \right).
\]
We also gather all the cubic terms in \eqref{main-cancel}:
\[
\frac{R^2 \g \th_-^2}{2\s_*v_-^5} \Big( \frac{2}{\g}w_1w_2^2 + \frac{2}{\g}w_1^2w_2 + \frac{4}{3}w_1^3 + \frac{2}{3(\g-1)}w_2^3 \Big)
+ C(\e_0+\d_3) \left(\abs{w_1}^3+\abs{w_3^3}\right),
\]
which implies
\begin{align}
\begin{aligned} \label{insideg}
&\frac{\l^2}{\e^3} \frac{v_-^3}{\s_*R\th_-}(\Ical_1 - \Gcal_v - \Gcal_\th) 
\le \int_0^1 \Big(\frac{2}{3}W_1^3 + \frac{1}{\g}W_1^2 W_2 + \frac{1}{\g}W_1 W_2^2 + \frac{1}{3(\g-1)}W_2^3\Big)dy \\ 
&\qquad\qquad\qquad\qquad\qquad
+ C \d_3 \int_0^1 (W_1^2 + W_2^2) dy + C(\e_0+\d_3) \int_0^1 (\abs{W_1}^3 +\abs{W_2}^3) dy \\
&\qquad\qquad\qquad\qquad\qquad
-\frac{\l}{\e}\frac{v_-}{2\g(\g-1)} \int_0^1 ((\g-1)W_1-W_2)^2 dy.
\end{aligned}
\end{align}
Note that the last term of \eqref{insideg} has a greater scale (as $\l/\e\gg1$) than the other terms, which we name as 
\[
T(W_1, W_2) \coloneqq \frac{\l}{\e}\frac{v_-}{2\g(\g-1)}\int_0^1 ((\g-1)W_1-W_2)^2 dy.
\]
At the end of this proof, we will use the good term $-T(W_1,W_2)$ to convert all terms including $W_2$ into terms written in $W_1$ only.\\
Due to the perturbation of coefficient in the definition of \(\Rcal_{\e, \d}\), we also need the following simple estimate, which is a direct consequence of the expansion \eqref{b1expansion} and \eqref{gexpansion}:
\begin{equation}\label{insideg-perturb}
\frac{\l^2}{\e^3}\abs{\Ical_1} + \frac{\l^2}{\e^3}\abs{\Gcal_v} + \frac{\l^2}{\e^3}\abs{\Gcal_\th}
\le C\frac{\l}{\e}\abs{\int_0^1 (W_1^2+W_2^2) dy}.
\end{equation}

\(\bullet\) \textbf{Change of variable for \(\Ical_2,\Ical_3,\Ical_4\):}
For \(\Ical_2\), we first rewrite it as follows:
\[
\Ical_2(v,\th) = R \int_\RR a \Big[\frac{1}{v}(\th-\tht) +\tht\Big(\frac{1}{v}-\frac{1}{\vt}\Big) \Big] \frac{(v-\vt)}{\vt} \ut' d\x.
\]
Hence, using \eqref{assYp}, \eqref{change-d}, and \eqref{ratio-vu}, we get
\begin{align} \label{I2}
\abs{\Ical_2 - \frac{\e \s_* R \th_-}{v_-^3}\int_0^1 (w_1^2+w_1 w_2) dy}
\le C(\e_0+\d_3)\e \Big(\int_0^1 w_1^2 dy + \int_0^1 w_2^2 dy \Big).
\end{align}
Moreover, \eqref{Phi-est1-v}, \eqref{Phi-est1-th}, \eqref{ratio-vu}, and \eqref{sm1} give
\begin{align} \label{I34}
\abs{\Ical_3 +\Ical_4 - \frac{\e \s_* R \th_-}{2 v_-^3}\int_0^1 (\g-1)w_1^2 + w_2^2 dy}
\le C(\e_0+\d_3)\e \Big(\int_0^1 w_1^2 dy +\int_0^1 w_2^2 dy \Big).
\end{align}
Thus, summing up \eqref{I2} and \eqref{I34}, we obtain
\begin{align}
\begin{aligned} \label{insidei}
\frac{\l^2}{\e^3}\frac{v_-^3}{\s_* R\th_-}(\Ical_2+\Ical_3+\Ical_4)
&\le \int_0^1 (W_1^2+W_1W_2) dy + \frac{\g-1}{2}\int_0^1 W_1^2 dy + \frac{1}{2} \int_0^1 W_2^2 dy \\
&\qquad
+ C(\e_0+\d_3)\Big(\int_0^1 W_1^2 dy + \int_0^1 W_2^2 dy\Big).
\end{aligned}
\end{align}

\(\bullet\) \textbf{Change of variable for \(\Dcal_v^1\):}
By the change of variable, we obtain that 
\begin{align*}
\Dcal_v^1 &\ge (1-C(\e_0 + \d_3))R\t_0\frac{\th_-}{v_-^3}\int_\RR (1+\l y)\abs{\rd_\x w_1}^2 d\x \\
&\ge (1-C(\e_0 + \d_3))R\t_0\frac{\th_-}{v_-^3}\int_0^1\abs{\rd_y w_1}^2 \left(\frac{dy}{d\x}\right)dy.
\end{align*}
To estimate $\frac{dy}{d\x}$, we use the system \eqref{shock_0} for viscous shocks to show Lemma \ref{sharp-d}, as in \cite{KVW-NSF}.
Thanks to Lemma \ref{sharp-d} below, we use the condition \eqref{cold} to have 
\begin{align*}
\Dcal_v^1 
&\ge \frac{\e}{2}\frac{R\t_0\th_-}{v_-^3}\frac{\s_*\g(\g+1)R}{R(\t_0+(\g+1)\th_-^2) + (\g-1)^2\th_-^2}(1-C(\e_0 + \d_3))
\int_0^1 y(1-y)\abs{\rd_y w_1}^2 dy \\
&\ge 0.9\e \frac{\s_*R\th_-}{v_-^3}\frac{\g^2+\g}{2}(1-C(\e_0 + \d_3))
\int_0^1 y(1-y)\abs{\rd_y w_1}^2 dy.
\end{align*}
This is where the condition \eqref{cold} is crucially used. After rescaling, we get
\begin{equation}\label{insided}
-\frac{\l^2}{\e^3}\frac{v_-^3}{\s_*R\th_-}\Dcal_v^1 \le -(0.9-C(\e_0+\d_3))\frac{\g^2+\g}{2}\int_0^1 y(1-y)\abs{\rd_y W_1}^2 dy.
\end{equation}

\(\bullet\) \textbf{Control on \(W_1\) and \(W_2\):} To use Proposition \ref{prop_nl_Poincare}, we need \(L^2\) bounds of \(W_1\) and \(W_2\). For this, we use \eqref{assYp} and \eqref{insidey} to observe that 
\begin{align*}
&\frac{1}{2} \int_0^1 W_1^2 dy + \frac{1}{2(\g-1)} \int_0^1 W_2^2 dy + \frac{1}{2\g} \int_0^1 (W_1+W_2)^2 dy \\
&\le
C_2\frac{v_-^2}{R\th_-} + 2\abs{\int_0^1 (W_1+W_2) dy}
+C(\e_0+\d_3) \Big( \int_0^1 (W_1^2 + W_2^2) dy + \int_0^1 (\abs{W_1} + \abs{W_2}) dy\Big).
\end{align*}
Using 
\[
\abs{ \int_0^1 (W_1 + W_2) dy} \le \int_0^1 \abs{W_1} dy + \int_0^1 \abs{W_2} dy
\le \frac{1}{8}\int_0^1 W_1^2 dy + \frac{1}{8(\g-1)}\int_0^1 W_2^2 dy + C,
\]
we obtain that for any small enough \(\e_0\) and \(\d_3\),
\[
\frac{1}{2} \int_0^1 W_1^2 dy + \frac{1}{2(\g-1)} \int_0^1 W_2^2 dy 
\le C + \frac{1}{3}\int_0^1 W_1^2 dy + \frac{1}{3(\g-1)}\int_0^1 W_2^2 dy.
\]
Hence, there is a constant \(C_1>0\) depending on \(C_2\), but not on \(\e\) and \(\e/\l\), such that 
\begin{equation}\label{L2normW1W2}
\int_0^1 W_1^2 dy + \int_0^1 W_2^2 dy \le C_1.
\end{equation}
Note that we do not expect any smallness on this constant. 

\(\bullet\) \textbf{Control on the \(\abs{\Ical_g^Y}^2\) term:}
We have 
\[
-\frac{\l^2}{\e^3}\frac{v_-^3}{\s_*R\th_-}\frac{\abs{\Ical_g^Y}^2}{\e\d_3}
= -\frac{R\th_-}{\d_3\s_*v_-}\abs{\frac{\l}{\e^2}\frac{v_-^2}{R\th_-}\Ical_g^Y}^2.
\]
From the inequality \(-a^2 \le -\frac{b^2}{2} + \abs{b-a}^2\) with 
\[
a = \frac{\l}{\e^2}\frac{v_-^2}{R\th_-}\Ical_g^Y, \quad
b = -\frac{1}{2} \int_0^1 W_1^2+\frac{W_2^2}{\g-1} dy - \frac{1}{2\g} \int_0^1 (W_1+W_2)^2 dy - 2\int_0^1 (W_1+W_2) dy,
\]
and using \eqref{insidey}, we obtain that 
\begin{align*}
-\frac{\l^2}{\e^3}\frac{v_-^3}{\s_*R\th_-}\frac{\abs{\Ical_g^Y}^2}{\e\d_3}
&\le -\frac{R\th_-}{2\d_3\s_*v_-}\abs{\frac{1}{2}\int_0^1 W_1^2+ \frac{W_2^2}{\g-1}+ \frac{1}{\g}(W_1+W_2)^2 dy + 2\int_0^1 W_1+W_2 dy}^2 \\
&\qquad
+ \frac{C}{\d_3}(\e_0+\d_3)^2\Big(\int_0^1 (W_1^2 +W_2^2)dy + \int_0^1 (\abs{W_1}+\abs{W_2})dy\Big)^2.
\end{align*}
Thus, using \eqref{L2normW1W2} and taking \(\e_0<\d_3\), we obtain that 
\begin{equation}
\begin{aligned}\label{insidey2}
-\frac{\l^2}{\e^3}\frac{v_-^3}{\s_*R\th_-}\frac{\abs{\Ical_g^Y}^2}{\e\d_3}
&\le -\frac{R\th_-}{2\d_3\s_*v_-}\abs{\frac{1}{2}\int_0^1 W_1^2+ \frac{W_2^2}{\g-1}+ \frac{1}{\g}(W_1+W_2)^2 dy + 2\int_0^1 W_1+W_2 dy}^2 \\
&\qquad
+ C\d_3\int_0^1 (W_1^2+W_2^2)dy.
\end{aligned}
\end{equation}

\(\bullet\) \textbf{Representation in \(W_1\):}
We here use $T(W_1,W_2)$ to represent all the previous terms in $W_1$ based on the following inequalities.\\
For any large constant $C'$ to be determined independently from  \(\e, \l\) or \(\d_3\), 
\begin{equation}
\begin{aligned}\label{transquad}
\int_0^1 W_1W_2 dy &= (\g-1)\int_0^1 W_1^2 dy - \int_0^1 W_1((\g-1)W_1-W_2) dy \\
&\le (\g-1)\int_0^1 W_1^2 dy + CC'\frac{\e}{\l}\int_0^1 W_1^2 dy
+ \frac{1}{C'}T(W_1, W_2) \\
\int_0^1 W_2^2 dy
&\le (\g-1)^2\int_0^1 W_1^2 dy + CC'\frac{\e}{\l}\int_0^1 (W_1^2+W_2^2) dy
+ \frac{1}{2C'}T(W_1, W_2) \\
\int_0^1 W_2^2 dy 
&\le \left(1 + CC'\frac{\e}{\l}\right)(\g-1)^2\int_0^1 W_1^2 dy
+ \frac{1}{C'}T(W_1, W_2).
\end{aligned}
\end{equation}
To represent the cubic term in $W_1$, we need an additional argument. 
Note that we have an \(L^\infty\) bound of \(W_1\) and \(W_2\) from the assumption \eqref{assYp}: 
\begin{align*}
\norm{W_1}_{L^\infty} = \frac{\l}{\e}\norm{w_1}_{L^\infty} \le C\frac{\l}{\e}\d_3, \qquad
\norm{W_2}_{L^\infty} = \frac{\l}{\e}\norm{w_2}_{L^\infty} \le C\frac{\l}{\e}\d_3.
\end{align*}
Therefore, using the same idea of replacing \(W_2\) by \((\g-1)W_1-((\g-1)W_1-W_2)\), we rewrite all terms in $W_1$ as follows.
\begin{equation}
\begin{aligned}\label{transcubic}
\int_0^1 W_1^2W_2 dy 
&\le (\g-1)\int_0^1 W_1^3 dy + CC'\frac{\e}{\l}\int_0^1 W_1^4 dy + \frac{1}{C'}T(W_1, W_2) \\
&\le (\g-1)\int_0^1 W_1^3 dy + CC'\d_3\int_0^1 \abs{W_1}^3 dy + \frac{1}{C'}T(W_1, W_2) \\
\int_0^1 \abs{W_2}^3 dy 
&\le (\g-1)^3\int_0^1 \abs{W_1}^3 dy + CC'\d_3\int_0^1 \abs{W_1}^3\hspace{-1mm}+\abs{W_2}^3 dy + \frac{1}{2C'}T(W_1, W_2) \\
\int_0^1 \abs{W_2}^3 dy &\le \left(1+CC'\d_3\right)(\g-1)^3\int_0^1 \abs{W_1}^3 dy + \frac{1}{C'}T(W_1, W_2)\\
\int_0^1 W_1W_2^2 dy 
&\le (\g-1)^2\int_0^1 W_1^3 dy + CC'\d_3\int_0^1 \abs{W_1}^3 dy + \frac{1}{C'}T(W_1, W_2) \\
\int_0^1 W_2^3 dy 
&\le (\g-1)^3\int_0^1 W_1^3 dy + CC'\d_3\int_0^1 \abs{W_1}^3 dy + \frac{1}{C'}T(W_1, W_2).
\end{aligned}
\end{equation}
This allows us to rewrite all terms in \eqref{insideg} and \eqref{insidei}: 
\begin{equation}
\begin{aligned}\label{insidegi2}
&\frac{\l^2}{\e^3}\frac{v_-^3}{\s_*R\th_-}(\Ical_1-\Gcal_v-\Gcal_\th) \le \frac{\g^2+\g}{3}\int_0^1 W_1^3 dy + C\d_3\int_0^1 W_1^2 dy \\
&\qquad\qquad\qquad\qquad\qquad\qquad
+ CC'(\e_0+\d_3)\int_0^1 \abs{W_1}^3 dy
+ \frac{C}{C'}T(W_1, W_2) - T(W_1, W_2), \\
&\frac{\l^2}{\e^3}\frac{v_-^3}{\s_*R\th_-}(\Ical_2+\Ical_3+\Ical_4)
\le \frac{\g^2+\g}{2}\int_0^1 W_1^2 dy + CC'(\e_0+\d_3)\int_0^1 W_1^2 dy + \frac{C}{C'}T(W_1, W_2).
\end{aligned}
\end{equation}
For \eqref{insidey}, we use the inequality \(-a^2 \le -\frac{1}{2}b^2 + \abs{a-b}^2\) once again with 
\begin{align*}
a = \frac{1}{2} \int_0^1 W_1^2+\frac{W_2^2}{(\g-1)}+\frac{1}{\g} (W_1+W_2)^2 dy  + 2\int_0^1 W_1+W_2 dy, \quad b = \g\int_0^1 W_1^2+2W_1 dy. 
\end{align*}
Factoring out \((\g-1)W_1-W_2\) from \(a-b\), we use H\"older's inequality with \eqref{L2normW1W2} to get
\[
\abs{a-b}^2
\le C\frac{\e}{\l}T(W_1, W_2).
\]
Using this to \eqref{insidey2} together with \eqref{transquad}, we obtain that 
\begin{equation}\label{insidey3}
-\frac{\l^2}{\e^3}\frac{v_-^3}{\s_*R\th_-}\frac{\abs{\Ical_g^Y}^2}{\e\d_3}
\le -\frac{\g^2R\th_-}{4\d_3\s_*v_-}\abs{\int_0^1 W_1^2+2W_1 dy}^2 + CC'\d_3\int_0^1 W_1^2dy + \frac{C}{\d_3}\frac{\e}{\l}T(W_1, W_2).
\end{equation}

\(\bullet\) \textbf{Conclusion:}
For any \(\d \in (0,\d_3)\), we have 
\begin{align*}
\Rcal_{\e,\d}(v, \th)&\le -\frac{1}{\e\d_3}\abs{\Ical_g^Y(v, \th)}^2 +\Ical_1(v, \th)+\d_3\left(\frac{\e}{\l}\right)\abs{\Ical_1(v, \th)}\\
&\quad\quad+(\Ical_2(v, \th)+\Ical_3(v, \th)+\Ical_4(v, \th))+\d_3(\abs{\Ical_2(v, \th)}+\abs{\Ical_3(v, \th)}+\abs{\Ical_4(v, \th)}) \\
&\quad\quad-\left(1-\d_3\left(\frac{\e}{\l}\right)\right)(\Gcal_v(v)+\Gcal_\th(\th))-(1-\d_3)\Dcal_v^1(v).
\end{align*}
From \eqref{insideg-perturb}, \eqref{insided}, \eqref{insidegi2}, and \eqref{insidey3}, we obtain that 
\begin{align*}
&\frac{\l^2}{\e^3}\frac{v_-^3}{\s_*R\th_-}\Rcal_{\e, \d}(v, \th)
\le -\frac{\g^2R\th_-}{4\d_3\s_*v_-}\abs{\int_0^1 (W_1^2+2W_1)dy}^2
+ \frac{\g^2+\g}{2}\int_0^1 W_1^2 + \frac{2}{3}W_1^3 dy \\
&\qquad
+ CC'(\e_0+\d_3)\int_0^1 W_1^2dy + CC'(\e_0+\d_3)\int_0^1 \abs{W_1}^3dy \\
&\qquad
-(0.9-C(\e_0+\d_3))\frac{\g^2+\g}{2}\int_0^1 y(1-y)\abs{\rd_y W_1}^2 dy 
+ \Big(\frac{C}{C'}+\frac{C}{\d_3}\frac{\e}{\l} - 1\Big)T(W_1, W_2).
\end{align*}
We here fix \(C'\) large enough so that \(\frac{C}{C'} + C\frac{\e_0}{\d_3} \le 1\) with smallness \(\e_0\ll\delta_3\). 
In this case, the last term in the inequality above is negative.  
Then, factoring out \(\frac{\g^2+\g}{2}\), we obtain that
\begin{align*}
&\frac{\l^2}{\e^3}\frac{2}{\g^2+\g}\frac{v_-^3}{\s_*R\th_-}\Rcal_{\e, \d}(v, \th) \\
&\le -\frac{R\g \th_-}{2(\g+1)\d_3\s_*v_-}\abs{\int_0^1 (W_1^2+2W_1)dy}^2
+(1+CC'(\e_0+\d_3))\int_0^1 W_1^2 dy + \frac{2}{3}\int_0^1 W_1^3 dy \\
&\qquad
+ CC'(\e_0+\d_3)\int_0^1 \abs{W_1}^3dy 
-(0.9-C(\e_0+\d_3))\int_0^1 y(1-y)\abs{\rd_y W_1}^2 dy.
\end{align*}
Then, let us fix the value of \(\d_2\) in Proposition \ref{prop_nl_Poincare}, with \(C_1\) given in \eqref{L2normW1W2}. 
Taking \(\d_3\) and \(\e_0\) small enough so that 
\[
\frac{2(\g+1)\d_3\s_*v_-}{R\g \th_-} \le \d_2, \quad \text{ and } \quad CC'(\e_0+\d_3) \le \d_2, 
\]
we obtain from Proposition \ref{prop_nl_Poincare} that 
\begin{align*}
\frac{\l^2}{\e^3}\frac{2}{\g^2+\g}\frac{v_-^3}{\s_*R\th_-}\Rcal_{\e, \d}(v, \th)
&\le -\frac{1}{\d_2}\abs{\int_0^1 (W_1^2+2W_1)dy}^2
+(1+\d_2)\int_0^1 W_1^2 dy + \frac{2}{3}\int_0^1 W_1^3 dy \\
&\qquad
+ \d_2\int_0^1 \abs{W_1}^3dy 
-(0.9-\d_2)\int_0^1 y(1-y)\abs{\rd_y W_1}^2 dy \le 0,
\end{align*}
which completes the proof.
\end{proof}

The following lemma provides a Jacobian estimate, used to estimate the diffusion term.
\begin{lemma}
It holds that 
\begin{equation} \label{sharp-d}
\bigg| \frac{1}{y(1-y)}\frac{dy}{d\x}
- \frac{\e}{2}\frac{\s_* \g(\g+1)R}{R(\t_0+ (\g+1)\th_-^2) + (\g-1)^2\th_-^2} \bigg| \le C \e^2.
\end{equation}
\end{lemma}
\begin{proof}
This can be proved similarly to \cite{KVW-NSF}, using the viscous shock equations \eqref{shock_0}.
Thus, we omit the details here and refer to \cite{EE-DBNSF}.
\end{proof}


\subsection{Truncation and its properties}
In order to prove Proposition \ref{prop:main} based on Proposition \ref{prop:main3}, we need to show that all bad terms for values for \(v\) and \(\th\) such that \(\abs{1/v-1/\vt}\ge \d_3\) and \(|\th-\tht|\ge\d_3\) 
could be controlled by the remaining good terms.
However, the value of \(\d_3\) is itself conditioned to the constant \(C_2\) in Proposition \ref{prop:main3}. 
Therefore, we need first to find a uniform bound on \(\Ical_g^Y\) which is not yet conditioned on the level of the truncation size \(k\).

We consider a truncation on \(\abs{1/v-1/\vt}\) and \(|\th-\tht|\) with a constant \(k>0\). Later we will consider the case \(k=\d_3\) as in Proposition \ref{prop:main3}. But for now, we consider the general case \(k\) to estimate the constant \(C_2\). For that, let \(\psi_k\) be a continuous function defined by
\begin{equation} \label{psi}
\psi_k(y)=\inf \left(k,\sup(-k,y)\right).
\end{equation}
We then define the functions \(\vbar_k\) and \(\thbar_k\) uniquely as
\begin{align}
1/\vbar_k-1/\vt =\psi_k\left(1/v-1/\vt\right), \qquad
\thbar_k - \tht =\psi_k (\th-\tht). \label{trunc-v-def}
\end{align}
We have the following lemma.

\begin{lemma} \label{lemmeC2}
For a fixed \((v_-,u_-,\th_-)\in \RR^+ \times \RR \times \RR^+ \), there exist \(C_2,k_0,\e_0,\d_0>0\) such that for any \(\e < \e_0, \e/\l < \d_0\) with \(\l<1/2\), the following is true whenever \(\abs{Y(U)}\le\e^2 :\)
\begin{align}
& \int_\RR a' \Big( \Phif{v}{\vt}+\Phif{\th}{\tht}+(u-\ut)^2 \Big) d\x \le C\frac{\e^2}{\l}, \label{locE} \\
& \abs{\Ical_g^Y(\vbar_k,\thbar_k)} \le C_2 \frac{\e^2}{\l}, \qquad \text{for every } k\le k_0. \label{lbis}
\end{align}
\end{lemma}

\begin{proof}
Since $Y(U)$ consists of the localized quadratic perturbations as in \eqref{locE} and the localized linear perturbation, the assumption \(\abs{Y(U)}\le\e^2\) would imply the smallness \eqref{locE}, which implies \eqref{lbis}.
The proof follows the same argument as in the proof of \cite[Lemma 4.4]{KV-Inven}, which utilizes the quadratic structures \eqref{rel_Phi_v}\(_1\) and \eqref{rel_Phi_th}\(_1\) on compact sets with H\"older's inequality, and the linear structures \eqref{rel_Phi_v}\(_2\) and \eqref{rel_Phi_th}\(_2\) on their complements with the scale difference \eqref{der-scale}.
Thus, we omit the details.
\end{proof}

We now fix the constant \(\d_3\) of Proposition \ref{prop:main3} associated to the constant \(C_2\) of Lemma \ref{lemmeC2}.
Without loss of generality, we may assume that \(\d_3<k_0\) (since Proposition \ref{prop:main3} holds for any smaller \(\d_3\).)
From now on, we set
\[
\vbar \coloneqq \vbar_{\d_3}, \quad \thbar \coloneqq \thbar_{\d_3}, \quad \Ubar \coloneqq \left(\vbar,u,\thbar\right), \quad \Bcal \coloneqq \Bcal_{\d_3}, \quad \Gcal \coloneqq \Gcal_{\d_3}.
\]
Note from Lemma \ref{lemmeC2} that
\begin{equation} \label{YC2}
\abs{\Ical_g^Y(\vbar,\thbar)} \le C_2 \frac{\e^2}{\l}.
\end{equation}

For the sake of technicality, we consider one-sided truncations : more specifically, we will control the bad terms in different ways for each case of small or big values of \(v\) and \(\th\). Thus, we define \(\vbar_s\) and \(\vbar_b\) as follows:
\begin{equation} \label{trunc-v-sb}
1/\vbar_s - 1/\vt \coloneqq \psi^b_{\d_3} (1/v-1/\vt), \quad
1/\vbar_b - 1/\vt \coloneqq \psi^s_{\d_3} (1/v-1/\vt),
\end{equation}
where \(\psi^b_{\d_3}\) and \(\psi^s_{\d_3}\) are one-sided truncations, i.e.,
\[
\psi^b_{\d_3}(y) \coloneqq \sup(-\d_3,y), \quad
\psi^s_{\d_3}(y) \coloneqq \inf(\d_3,y).
\]
Then, we also define \(\thbar_s\) and \(\thbar_b\)
\begin{equation} \label{trunc-th-sb}
\thbar_s - \tht \coloneqq \psi^s_{\d_3} (\th-\tht), \quad
\thbar_b - \tht \coloneqq \psi^b_{\d_3} (\th-\tht).
\end{equation}
Notice that the function \(\vbar_s\) and \(\thbar_s\) (resp. \(\vbar_b\) and \(\thbar_b\)) represent the truncation of big (resp. small) values of \(v\) and \(\th\).
Then, by definitions of truncation \eqref{trunc-v-def} and \eqref{trunc-v-sb}, we obtain
\begin{align}
\begin{aligned} \label{eq_D}
\Dcal_v^1(v)
&=R\t_0\int_\RR a v\tht \abs{\Big(\frac{1}{v}-\frac{1}{\vt} \Big)_\x}^2(\one{\abs{1/v-1/\vt}\le\d_3}+\one{1/v-1/\vt>\d_3}+\one{1/v-1/\vt<-\d_3}) d\x \\
&=\Dcal_v^1(\vbar)
+ R\t_0\int_\RR a v\tht \abs{\Big(\frac{1}{v}-\frac{1}{\vbar_b} \Big)_\x}^2 d\x
+ R\t_0\int_\RR a v\tht \abs{\Big(\frac{1}{v}-\frac{1}{\vbar_s} \Big)_\x}^2 d\x \\
&\ge
R\t_0\int_\RR a v\tht \abs{\Big(\frac{1}{v}-\frac{1}{\vbar_b} \Big)_\x}^2 d\x
+ R\t_0\int_\RR a v\tht \abs{\Big(\frac{1}{v}-\frac{1}{\vbar_s} \Big)_\x}^2 d\x.
\end{aligned}
\end{align}
Especially, this implies the monotonicity property:
\begin{equation}\label{keyD}
D_v^1(v) \ge D_v^1(\vbar).
\end{equation}
On the other hand, since it holds from \eqref{trunc-v-def} that \(\Phi(v/\vt) \ge \Phi(\vbar/\vt)\) and \(\Phi(\th/\tht) \ge \Phi(\thbar/\tht)\),
using \eqref{locE} yields the following:
\begin{equation} \label{l2-v}
0 \le R\s_\e \int_\RR a' \Big( \Phif{v}{\vt}-\Phif{\vbar}{\vt} \Big) d\x
= \Gcal_v(U)-\Gcal_v(\Ubar) \le \Gcal_v(U) \le C\frac{\e^2}{\l}
\end{equation}
and
\begin{equation} \label{l2-th}
0 \le \frac{R\s_\e}{\g-1} \int_\RR a' \Big( \Phif{\th}{\tht}-\Phif{\thbar}{\tht} \Big) d\x
= \Gcal_\th(U)-\Gcal_\th(\Ubar) \le \Gcal_\th(U) \le C\frac{\e^2}{\l}.
\end{equation}

Additionally, we need to consider a truncation regarding the values of \(u\) for technical issues as well. Hence, we define \(\ubar\) as follows:
\[
\ubar-\ut \coloneqq \psi_{\d_3}(u-\ut).
\]
Notice that the truncation size does not necessarily have to be \(\d_3\), but we use \(\d_3\) just to give a sense of unity.
It is also worth mentioning that for the case of \(u\), it is not necessary to define \(\ubar_s\) and \(\ubar_b\).
Moreover, most importantly, \(\Ubar\) still stands for \((\vbar,u,\thbar)\), not \((\vbar,\ubar,\thbar)\).
The truncation \(\ubar\) will be only used to control parabolic bad terms in Section \ref{section-parabolic}.

In what follows, for simplicity, we use the following notation:
\[
\O \coloneqq \{\x \mid \left(1/v-1/\vt\right) \le \d_3 \}.
\]
This domain is for big values of $v$ or inside truncation.

\subsection{Control of hyperbolic terms outside truncation}\label{section_hyperbolic}
For the hyperbolic bad terms $\Bcal_{\delta_3}$ in \eqref{badgood}, we will use the following notations:
\begin{equation}\label{bad0}
\Bcal_{\delta_3} = \Bcal^+_1 + \Bcal^-_1 + \Bcal_2 + \Bcal_3 + \Bcal_4
\end{equation}
where
\begin{align*}
&\Bcal_1^-(U) \coloneqq \int_{\O^c} a'(p-\pt)(u-\ut) d\x, &
&\Bcal_1^+(U) = \frac{1}{2\s_\e}\int_\O a'(p-\pt)^2 d\x, \\
&\Bcal_2(U) \coloneqq \int_\RR av(p-\pt)\Big(\frac{1}{v}-\frac{1}{\vt}\Big)\ut' d\x, \\
&\Bcal_3(U) \coloneqq -R\s_\e\int_\RR a\Phif{v}{\vt}\tht' d\x, &
&\Bcal_4(U) \coloneqq -\frac{R}{\g-1}\s_\e\int_\RR a\Phif{\th}{\tht}\tht' d\x.
\end{align*}
For the diffusion terms \(\Dcal(U)\) in \eqref{ybg-first}, we will use the following notations:
\[
\Dcal_v(U) \coloneqq \Dcal_v^1(U) + \Dcal_v^2(U), \,\,
\Dcal_u(U) \coloneqq \int_\RR a\tht\frac{\th}{v}\abs{(u-\ut)_\x}^2 d\x, \,\,
\Dcal_\th(U) \coloneqq \int_\RR a\frac{\tht}{v}\abs{(\th-\tht)_\x}^2 d\x,
\]
where
\beq\label{d1vdef}
\Dcal_v^1(U) \coloneqq 
R\t_0\int_\RR a\tht v\abs{\Big(\frac{1}{v}-\frac{1}{\vt}\Big)_\x}^2 d\x, \,\,\,\,
\Dcal_v^2(U) \coloneqq 
R\int_\RR a\tht v\th^2\abs{\Big(\frac{1}{v}-\frac{1}{\vt}\Big)_\x}^2 d\x.
\eeq
We also recall the notations \(\Gcal_u^-(U)\), \(\Gcal_v(U)\), \(\Gcal_\th(U)\) in \eqref{ggd} for the good terms.
In Proposition \ref{prop_hyperbolic_out}, we will show that a small amount of good terms controls the bad terms outside truncation.

\begin{proposition}\label{prop_hyperbolic_out}
There exist constants $\e_0, \deo, C, C^*>0$ (in particular, $C$ depends on the constant $\d_3$) such that for any $\e<\e_0$ and $\d_0^{-1}\e<\lambda<\d_0<1/2$, the following statements hold true.
For any $U$ such that $\abs{Y(U)}\le \e^2$,
\begin{align}
&\abs{\Bcal_1^+(U)-\Ical_1(\vbar, \thbar)}
\le C\Big(\frac{\e}{\l}(\Dcal_v(U)+\Dcal_\th(U))+ \e^2\Gcal_v(U)\Big),\quad \mbox{for  $\Ical_1$ in \eqref{note-in}},
\label{bo1p}\\
&\abs{\Bcal_1^-(U)}
\le C\Big(\sqrt{\frac{\e}{\l}}\Gcal_u^-(U) + \sqrt{\frac{\e}{\l}}(\Dcal_v(U)+\Dcal_\th(U)) + \e(\Gcal_v(U)+\Gcal_\th(U))\Big) 
\label{bo1m}\\
&\begin{aligned}\label{bo2}
&\abs{\Bcal_2(U)-\Bcal_2(\Ubar)}
\le C\frac{\e}{\l}\Big(\frac{\e}{\l}(\Dcal_v(U)+\Dcal_\th(U)) + \e^2(\Gcal_v(U)+\Gcal_\th(U)) \\
&\hspace{60mm}
+ (\Gcal_v(U)-\Gcal_v(\Ubar)+\Gcal_\th(U)-\Gcal_\th(\Ubar))\Big)
\end{aligned}
\\
&\abs{\Bcal_3(U)-\Bcal_3(\Ubar)}\le C\frac{\e}{\l}(\Gcal_v(U)-\Gcal_v(\Ubar))
\label{bo3}\\
&\abs{\Bcal_4(U)-\Bcal_4(\Ubar)}\le C\frac{\e}{\l}(\Gcal_\th(U)-\Gcal_\th(\Ubar))
\label{bo4}\\
&\abs{\Bcal_{\d_3}(U)}+\abs{\Bcal_1^+(U)}\le C^*\Big(\frac{\e^2}{\l} + \sqrt{\frac{\e}{\l}}(\Dcal_v(U)+\Dcal_\th(U))\Big)
\label{bo}
\end{align}
\end{proposition}

In the following lemma, we first present pointwise estimates. The proof of Proposition \ref{prop_hyperbolic_out} will be presented at the end of this subsection.
\begin{lemma}\label{lemma_hyp_pw}
Under the same assumptions as Proposition \ref{prop_hyperbolic_out}, we have
\begin{align}
&\abs{\sqrt{v}\Big(\frac{1}{v}-\frac{1}{\vbar_b}\Big)(\x)} 
\le C\bigg(\sqrt{\int_{\x_0}^\x \one{\abs{1/v-1/\vbar_b}>0}d\z} \sqrt{\Dcal_v(U)} + \sqrt{\frac{\e^2}{\l}}\sqrt{\Gcal_v(U)}\bigg) \label{pwv1}\\
&\begin{aligned}
&\abs{\sqrt{v}\th\Big(\frac{1}{v}-\frac{1}{\vbar_b}\Big)(\x)} 
\le C\bigg(\sqrt{\int_{\x_0}^\x \one{\abs{1/v-1/\vbar_b}>0}d\z}\left(\sqrt{\Dcal_v(U)+\Dcal_\th(U)}\right) \\
&\hspace{50mm} + \frac{\e}{\l}(\Gcal_v(U) + \Gcal_\th(U))
+ \sqrt{\frac{\e^2}{\l}}\sqrt{\int_\RR a'v\Big(\frac{1}{v}-\frac{1}{\vbar_b}\Big)^2 d\x}\bigg) \label{pwv2}
\end{aligned}\\
&\abs{\Big(\frac{1}{v}-\frac{1}{\vbar_s}\Big)}
\le C\sqrt{\int_{\x_0}^\x \one{\abs{1/v-1/\vbar_s}>0}d\z}\sqrt{\Dcal_v(U)} \label{pwv3'} \\
&\begin{aligned}
&\abs{\th\Big(\frac{1}{v}-\frac{1}{\vbar_s}\Big)}
\le C\bigg(\sqrt{\int_{\x_0}^\x \one{\abs{1/v-1/\vbar_s}>0}d\z}\sqrt{\Dcal_v(U)} \\
&\hspace{50mm}+\sqrt{\Dcal_\th(U)\int_{\x_0}^{\x}v\one{\abs{1/v-1/\vbar_s}>0}d\z} + \frac{\e}{\l}\Gcal_v(U)\bigg) \label{pwv3}
\end{aligned}\\
&\abs{(\th-\thbar)(\x)} 
\le C\sqrt{\Dcal_\th(U)}\sqrt{\int_{\x_0}^{\x}v\one{|\th-\thbar|>0}d\z} \label{pwth1} \\
&\abs{\frac{1}{\sqrt{v}}(\th-\thbar)(\x)}
\le C\bigg(\sqrt{\int_{\x_0}^\x \one{\abs{\th-\thbar}>0}d\z}\sqrt{\Dcal_v(U) + \Dcal_\th(U)} 
+ \int_\RR \sqrt{v}|\th-\thbar|\vt'd\x\bigg) \label{pwth2}
\end{align}
\end{lemma}
\begin{proof}
From \eqref{lower-v} and \eqref{der-a},
\begin{align*}
\frac{\e}{2} \int_{-1/\e}^{1/\e} \tht \eta(U|\Ut) d\x
&\le \frac{\e}{2\inf_{[-1/\e,1/\e]}\abs{a'}} \int_\RR a'  \tht \eta(U|\Ut) d\x \le C \frac{\e}{\l\e}\frac{\e^2}{\l} = C \left(\frac{\e}{\l}\right)^2.
\end{align*}
Hence, there exists \(\x_0\in[-1/\e,1/\e]\) such that
\[
\eta(U(\x_0)|\Ut(\x_0)) \le C \left(\frac{\e}{\l}\right)^2
\]
which implies 
\[
\Phif{v}{\vt}(\x_0)+\Phif{\th}{\tht}(\x_0)+(u-\ut)^2(\x_0) \le C \left(\frac{\e}{\l}\right)^2.
\]
For \(\d_0\) small enough, we use Lemma \ref{lem-quad} to find
\[
\abs{\left(1/v-1/\vt\right)(\x_0)}+|(\th-\tht)(\x_0)|+\abs{(u-\ut)(\x_0)} \le C \frac{\e}{\l}.
\]
Thus, if \(\d_0\) is small enough so that \(C\e/\l \le \d_3\), then we obtain from the definitions of \(\vbar,\thbar\) and \(\ubar\) that
\[
(v-\vbar)(\x_0)=(\th-\thbar)(\x_0)=(u-\ubar)(\x_0)=0.
\]

On the other hand, from the definition of truncations, it holds that \(1/v>1/\vt+\d_3\) for any \(\x\) with \(\abs{(1/v-1/\vbar_b)(\x)}> 0\).
Then, using \eqref{rel_Phi_v}, there is a lower bound of \(\Phi(v/\vt)\) on the region \(\{\abs{(1/v-1/\vbar_b)(\x)}> 0\}\), which depends only on \(\d_3\). 
Similarly, on \(\{\abs{(1/v-1/\vbar_s)(\x)}> 0\}\), \(\Phi(v/\vt)\) is bounded from below, and on \(\{|\th-\thbar|>0\}\), \(\Phi(\th/\tht)\) is bounded from below as well. 
Say \(\a>0\) is a common lower bound, which depends on \(\d_3\): that is,
\begin{align}
\begin{aligned} \label{Qconst}
\one{\abs{1/v-1/\vbar_b}>0} \le \frac{\Phi(v/\vt)}{\a}, \quad
\one{\abs{1/v-1/\vbar_s}>0} \le \frac{\Phi(v/\vt)}{\a}, \quad
\one{|\th-\thbar|>0} \le \frac{\Phi(\th/\tht)}{\a}.
\end{aligned}
\end{align}
In addition, since \(\frac{\Phi(v/\vt)}{v}\) is increasing along \(v\) on \((\vt,\infty)\), 
\(v\one{\abs{1/v-1/\vbar_s}>0}\) can be bounded by \(\Phi(v/\vt)\) as well:
\begin{align} \label{Qlin}
v \one{\abs{1/v-1/\vbar_s}>0} = v \one{1/v <1/\vt -\d_3} \le C \Phif{v}{\vt}
\end{align}
for any constant \(C\) greater than \(\frac{v}{\Phi(v/\vt)}\) for \(v\) with \(1/v=1/\vt -\d_3\).
Obviously, the following holds as well:
\begin{align} \label{Qlin'}
\th \one{\th>\tht+\d_3} \le C \Phif{\th}{\tht}.
\end{align}

\bpf{pwv1}
For any \(\x\in\RR\), we use the fundamental theorem of calculus to obtain
\[
\abs{\sqrt{v}\Big(\frac{1}{v}-\frac{1}{\vbar_b}\Big)(\x)}
\le C\int_{\x_0}^{\x} \bigg[\abs{\sqrt{v}\Big(\frac{1}{v}-\frac{1}{\vbar_b}\Big)_\x} 
+ \abs{\sqrt{v}v\Big(\frac{1}{v}-\frac{1}{\vbar_b}\Big)\Big(\frac{1}{v}\Big)_\x}\bigg] d\x.
\]
For the first term, we can apply H\"older's inequality and \eqref{eq_D}: 
\[
\int_{\x_0}^{\x} \abs{\sqrt{v}\Big(\frac{1}{v}-\frac{1}{\vbar_b}\Big)_\x}d\x
\le \sqrt{\int_{\x_0}^\x \one{\abs{1/v-1/\vbar_b}>0}d\z}\sqrt{\Dcal_v(U)}.
\]
For the second term, since \(v\abs{\frac{1}{v}-\frac{1}{\vbar_b}} \le \one{\abs{1/v-1/\vbar_b}>0}\), H\"older inequality implies
\begin{align*}
\int_{\x_0}^\x \abs{\sqrt{v}v\Big(\frac{1}{v}-\frac{1}{\vbar_b}\Big)\Big(\frac{1}{v}\Big)_\x} d\x 
&\le \int_{\x_0}^{\x} \bigg(\abs{\sqrt{v}\Big(\frac{1}{v}-\frac{1}{\vtil}\Big)_\x} + \abs{\sqrt{v}\frac{\vtil'}{\vtil^2}}\bigg)\one{\abs{1/v-1/\vbar_b}>0} d\x \\
&\hspace{-5mm}\le \sqrt{\int_{\x_0}^\x \one{\abs{1/v-1/\vbar_b}>0} d\z}\sqrt{\Dcal_v(U)}
+ C\int_{\x_0}^{\x} \sqrt{v}\one{\abs{1/v-1/\vbar_b}>0}\vtil' d\x.
\end{align*}
Further, the second integral here can be bounded by the hyperbolic good term \(\Gcal_v(U)\), from \(v\one{\abs{1/v-1/\vbar_b}>0}\le \Phi(v/\vtil)\) and \eqref{der-a}, we have 
\[
\int_{\x_0}^{\x} \sqrt{v}\one{\abs{1/v-1/\vbar_b}>0}\vtil' d\x
\le \sqrt{\int_{\x_0}^{\x} \vtil' d\x}\sqrt{\int_{\x_0}^\x v\one{\abs{1/v-1/\vbar_b}>0}\vtil' d\x}
\le \sqrt{\frac{\e^2}{\l}}\sqrt{\Gcal_v(U)}.
\]
Summing up the inequalities above, we obtain the desired conclusion: 
\[
\abs{\sqrt{v}\Big(\frac{1}{v}-\frac{1}{\vbar_b}\Big)(\x)} 
\le C\sqrt{\int_{\x_0}^\x \one{\abs{1/v-1/\vbar_b}>0} d\z}\sqrt{\Dcal_v(U)} + C\sqrt{\frac{\e^2}{\l}}\sqrt{\Gcal_v(U)}.
\]

\bpf{pwv2}
Note by \eqref{Qconst} and \eqref{Qlin'} that 
\[
\th \one{1/v>\vtil+\d_3} \le C \one{1/v>\vtil+\d_3} + \th \one{\th>\thtil+\d_3} \le C \Phif{v}{\vt} + C\Phif{\th}{\tht}.
\]
Hence, this with \eqref{der-scale} gives that for any \(\x\in\RR\),
\begin{align*}
&\abs{\sqrt{v}\th\Big(\frac{1}{v}-\frac{1}{\vbar_b}\Big)(\x)} 
\le C\int_{\x_0}^{\x} \bigg[\abs{\sqrt{v}\th\Big(\frac{1}{v}-\frac{1}{\vbar_b}\Big)_\x}
+ \abs{\sqrt{v}\Big(\frac{1}{v}-\frac{1}{\vbar_b}\Big)(\th-\tht)_\x}
+ \abs{\sqrt{v}\Big(\frac{1}{v}-\frac{1}{\vbar_b}\Big)\tht'} \\
&\hspace{50mm}
+ \abs{\sqrt{v}v\Big(\frac{1}{v}-\frac{1}{\vbar_b}\Big)\th\Big(\frac{1}{v}-\frac{1}{\vt}\Big)_\x}
+ \abs{\sqrt{v}v\Big(\frac{1}{v}-\frac{1}{\vbar_b}\Big)\th\vt'}\bigg]d\x \\
&\hspace{25mm}\le 
C\int_{\x_0}^{\x} \bigg[\abs{\sqrt{v}\th\Big(\frac{1}{v}-\frac{1}{\vt}\Big)_\x} 
+ \abs{\frac{1}{\sqrt{v}}(\th-\tht)_\x}
+ \abs{\th \vt'}\bigg]\one{\abs{1/v-1/\vbar_b}>0} d\x \\
&\hspace{25mm}\qquad\qquad
+ C\sqrt{\frac{\e^2}{\l}}\sqrt{\int_\RR a'v\Big(\frac{1}{v}-\frac{1}{\vbar_b}\Big)^2 d\x} \\
&\hspace{25mm}\le 
C\sqrt{\int_{\x_0}^\x \one{\abs{1/v-1/\vbar_b}>0} d\z}\left(\sqrt{\Dcal_v(U)} + \sqrt{\Dcal_\th(U)}\right)
+ C\frac{\e}{\l}(\Gcal_\th(U) + \Gcal_v(U)) \\
&\hspace{25mm}\qquad\qquad
+ C\sqrt{\frac{\e^2}{\l}}\sqrt{\int_\RR a'v\Big(\frac{1}{v}-\frac{1}{\vbar_b}\Big)^2 d\x}.
\end{align*}

\bpf{pwv3'} For any \(\x\in\RR\),
\begin{align*}
\abs{\Big(\frac{1}{v}-\frac{1}{\vbar_s}\Big)(\x)}
&\le C\int_{\x_0}^{\x}\abs{\sqrt{v}\Big(\frac{1}{v}-\frac{1}{\vbar_s}\Big)_\x} d\x
\le C \sqrt{\int_{\x_0}^\x \one{\abs{1/v-1/\vbar_s}>0} d\z}\sqrt{\Dcal_v(U)}.
\end{align*}

\bpf{pwv3} For any \(\x\in\RR\), \eqref{Qconst} and \eqref{der-scale} yield that
\begin{align*}
\abs{\th\Big(\frac{1}{v}-\frac{1}{\vbar_s}\Big)(\x)}
&\le C\int_{\x_0}^{\x}\bigg[\abs{\th\Big(\frac{1}{v}-\frac{1}{\vbar_s}\Big)_\x}
+ \abs{\Big(\frac{1}{v}-\frac{1}{\vbar_s}\Big)(\th-\tht)_\x}
+ \abs{\Big(\frac{1}{v}-\frac{1}{\vbar_s}\Big)\tht'}\bigg]d\x \\
&\le C\int_{\x_0}^{\x}\bigg[\abs{\sqrt{v}\th\Big(\frac{1}{v}-\frac{1}{\vbar_s}\Big)_\x} + \abs{\frac{1}{\sqrt{v}}(\th-\tht)_\x\cdot \sqrt{v}} + \abs{\tht'}\bigg] \one{\abs{1/v-1/\vbar_s}>0} d\x \\
&\hspace{-26mm}\le C\sqrt{\Dcal_v(U)} \sqrt{\int_{\x_0}^\x \one{\abs{1/v-1/\vbar_s}>0} d\z}
+ C\sqrt{\Dcal_\th(U)}\sqrt{\int_{\x_0}^{\x}v\one{\abs{1/v-1/\vbar_s}>0}d\x} 
+ C\frac{\e}{\l}\Gcal_v(U).
\end{align*}

\bpf{pwth1} For any \(\x\in\RR\),
\[
\abs{(\th-\thbar)(\x)}
\le C\int_{\x_0}^{\x} \frac{1}{\sqrt{v}}(\th-\thbar)_\x\cdot \sqrt{v}\one{|\th-\thbar|>0} d\x
\le C\sqrt{\Dcal_\th(U)}\sqrt{\int_{\x_0}^{\x}v\one{|\th-\thbar|>0} d\x}.
\]

\bpf{pwth2} For any \(\x\in\RR\),
\begin{align*}
\abs{\frac{1}{\sqrt{v}}(\th-\thbar)(\x)}
&\le C\int_{\x_0}^{\x} \bigg[\abs{\frac{1}{\sqrt{v}}(\th-\thbar)_\x} +\abs{\sqrt{v}(\th-\thbar)\Big(\frac{1}{v}-\frac{1}{\vt}\Big)_\x} +\abs{\sqrt{v}(\th-\thbar)\vt'}\bigg] d\x \\
&\hspace{-26mm}\le C\sqrt{\int_{\x_0}^\x \one{\abs{\th-\thbar}>0} d\z}\sqrt{\Dcal_\th(U)}
+ C\sqrt{\int_{\x_0}^\x \one{\abs{\th-\thbar}>0} d\z}\sqrt{\Dcal_v(U)} 
+ C\int_\RR \sqrt{v} |\th-\thbar| \vt'd\x.
\end{align*}
\end{proof}

From Lemma \ref{lemma_hyp_pw}, we have the following.
\begin{lemma}\label{lemma_hyp_ineq}
Under the same assumptions as Proposition \ref{prop_hyperbolic_out}, there is a constant \(C>0\) depends only on \(\d_3\) such that the following inequalities are true: 
\begin{align}
&\int_\RR a'v\Big(\frac{1}{v}-\frac{1}{\vbar_b}\Big)^2 d\x
\le C\Big(\frac{\e}{\l}\Dcal_v(U) + \e^2\Gcal_v(U)\Big)
\label{ineq-b1} \\
&\begin{aligned}
&\abs{\sqrt{v}\th\Big(\frac{1}{v}-\frac{1}{\vbar_b}\Big)(\x)} 
\le C\bigg(\sqrt{\int_{\x_0}^\x \one{\abs{1/v-1/\vbar_b}>0} d\z}\sqrt{\Dcal_v(U)+ \Dcal_\th(U)}  \\
&\hspace{50mm} + \sqrt{\frac{\e^2}{\l}}\sqrt{\Dcal_v(U)}
+\sqrt{\frac{\e^2}{\l}}\sqrt{\Gcal_v(U)+\Gcal_\th(U)}\bigg) 
\end{aligned}
\label{pwv2'}\\
&\int_\RR a'v\th^2\Big(\frac{1}{v}-\frac{1}{\vbar_b}\Big)^2 d\x
\le C\Big(\frac{\e}{\l}(\Dcal_v(U)+\Dcal_\th(U))+ \e^2(\Gcal_v(U)+\Gcal_\th(U))\Big)
\label{ineq-b2} \\
&\int_\RR a'(\th-\thbar)^2 d\x
\le C\frac{\e}{\l}\Dcal_\th(U)
\label{ineq-b3} \\
&\abs{\frac{1}{\sqrt{v}}(\th-\thbar)(\x)}
\le C\bigg(\sqrt{\Dcal_v(U)+\Dcal_\th(U)}\sqrt{\int_{\x_0}^\x \one{\abs{\th-\thbar}>0}d\z}+\sqrt{\frac{\e^2}{\l}}\sqrt{\Dcal_\th(U)}\bigg)
\label{pwth2'} \\
&\int_\RR a'\Big(\frac{1}{v}-\frac{1}{\vbar_s}\Big)^2 d\x
\le C \frac{\e}{\l} \Dcal_v(U)
\label{ineq-b4'} \\
&\int_\RR a'\th^2\Big(\frac{1}{v}-\frac{1}{\vbar_s}\Big)^2 d\x
\le C\Big(\frac{\e}{\l} (\Dcal_v(U)+\Dcal_\th(U)) + \e^2 \Gcal_v(U)\Big)
\label{ineq-b4}
\end{align}
\end{lemma}
\begin{proof} Here, we use the pointwise estimates derived in \eqref{lemma_hyp_pw}. 

\bpf{ineq-b1}
Using \eqref{pwv1}, \eqref{Qconst}, \eqref{locE} and Lemma \ref{lemma_pushing} yields that 
\begin{align*}
\int_\RR a'v\Big(\frac{1}{v}-\frac{1}{\vbar_b}\Big)^2 d\x
&\le C \Dcal_v(U)\int_\RR a' \int_{\x_0}^\x \one{\abs{1/v-1/\vbar_b}>0} d\z d\x
+C\frac{\e^2}{\l}\Gcal_v(U)\int_\RR a' d\x \\
&\le C\frac{\e}{\l}\Dcal_v(U) + \e^2 \Gcal_v(U).
\end{align*}

\bpf{pwv2'} 
Thanks to \eqref{pwv2} with \eqref{locE}, we improve \eqref{pwv2} to \eqref{pwv2'}.

\bpf{ineq-b2}
Using \eqref{pwv2'}, \eqref{Qconst}, \eqref{locE} and Lemma \ref{lemma_pushing}, we proceed as follows:
\begin{equation*}
\begin{aligned}
\int_\RR a'v\th^2\Big(\frac{1}{v}-\frac{1}{\vbar_b}\Big)^2 d\x 
&\le C(\Dcal_v(U)+\Dcal_\th(U))\int_\RR a'\int_{\x_0}^\x \one{\abs{1/v-1/\vbar_b}>0} d\z d\x \\
&\qquad + C\e^2\Dcal_v(U) + C\e^2(\Gcal_v(U)+\Gcal_\th(U)) \\
&\le
C\frac{\e}{\l}(\Dcal_v(U)+\Dcal_\th(U))+ C\e^2(\Gcal_v(U)+\Gcal_\th(U)).
\end{aligned}
\end{equation*}

\bpf{ineq-b3} Note by \eqref{Qconst} and \eqref{Qlin} that 
\begin{equation}\label{vlocE}
v\one{|\th-\thbar|>0}
\le v\one{1/v<1/\vt-\d_3} + C\one{|\th-\thbar|>0}
\le C\Phif{v}{\vt} + C\Phif{\th}{\tht}.
\end{equation}
Using \eqref{pwth1} together with Lemma \ref{lemma_pushing} and \eqref{locE}, we have
\begin{equation*}
\begin{aligned}
\int_\RR a'(\th-\thbar)^2 d\x
&\le C\Dcal_\th(U)\int_\RR a'\Big(\int_{\x_0}^\x v\one{|\th-\thbar|>0}d\z\Big)d\x \\
&\le C\Dcal_\th(U)\frac{1}{\e}\int_\RR a'\Big(\Phif{v}{\vt}+\Phif{\th}{\tht}\Big) d\x 
\le C\frac{\e}{\l}\Dcal_\th(U).
\end{aligned}
\end{equation*}

\bpf{pwth2'}
We use \eqref{ineq-b3} with \eqref{vlocE} and \eqref{locE} to find that
\begin{align*}
&\int_\RR \sqrt{v} |\th-\thbar|\vt'd\x
\le \Big(\int_\RR v\one{|\th-\thbar|>0}\vt'd\x\Big)^{1/2}\Big(\int_\RR (\th-\thbar)^2\vt'd\x\Big)^{1/2} \\
&\le C\frac{\e}{\l}\Big(\int_\RR a' \Phif{v}{\vt} d\x
+ \int_\RR a' \Phif{\th}{\tht} d\x\Big)^{1/2} 
\Big(\int_\RR a'(\th-\thbar)^2d\x\Big)^{1/2} 
\le C\sqrt{\frac{\e^2}{\l}}\sqrt{\Dcal_\th(U)}.
\end{align*}
This combined with \eqref{pwth2} gives \eqref{pwth2'} as we desired.

\bpf{ineq-b4'} We use \eqref{pwv3'} with \eqref{Qconst}, \eqref{locE}, and Lemma \ref{lemma_pushing} to proceed as follows.
\[
\int_\RR a'\Big(\frac{1}{v}-\frac{1}{\vbar_s}\Big)^2 d\x 
\le C\Dcal_v(U)\int_\RR \int_{\x_0}^{\x} \one{\abs{1/v-1/\vbar_x}>0}d\z d\x
\le C\frac{\e}{\l}\Dcal_v(U).
\]

\bpf{ineq-b4} We use \eqref{pwv3} with \eqref{Qconst}, \eqref{Qlin}, \eqref{locE} and Lemma \ref{lemma_pushing} find that
\begin{equation*}
\begin{aligned}
\int_\RR a'\th^2\Big(\frac{1}{v}-\frac{1}{\vbar_s}\Big)^2 d\x 
&\le C\Dcal_v(U)\int_\RR a' \int_{\x_0}^\x \one{1/v-1/\vbar_s}d\z d\x \\
&\qquad
+C\Dcal_\th(U)\int_\RR a' \int_{\x_0}^\x v \one{1/v-1/\vbar_s}d\z d\x
+\frac{\e^2}{\l}\Gcal_v(U)^2 \\
&\le C\frac{\e}{\l} (\Dcal_v(U)+\Dcal_\th(U)) + C\e^2 \Gcal_v(U).
\end{aligned}
\end{equation*}
This completes the proof of Lemma \ref{lemma_hyp_ineq}.
\end{proof}

\subsubsection{Proof of Proposition \ref{prop_hyperbolic_out}.}
\bpf{bo1p} From \(p = R\th/v\), we have
\begin{align*}
\abs{\Bcal_1^+(U)-\Ical_1(\vbar, \thbar)}
&\le \frac{1}{2\s_\e}\int_\O a'\abs{(p-\ptil)^2 - (\pbar-\ptil)^2} d\x 
+ \frac{1}{2\s_\e}\int_{\O^c} a'(\pbar-\ptil)^2 d\x \\ 
&\le C\int_\O a'\abs{p-\pbar}(\abs{p-\pbar}+2\abs{\pbar-\ptil}) d\x
+ \frac{1}{2\s_\e}\int_{\O^c} a'(\pbar-\ptil)^2 d\x \\ 
&\le C\int_\O a'\Big(\frac{\th}{v}-\frac{\thbar}{\vbar}\Big)^2 d\x
+ C\d_3\int_\O a'\abs{\frac{\th}{v}-\frac{\thbar}{\vbar}} d\x
+ C\d_3^2\int_{\O^c} a'd\x.
\end{align*}
The last term appears because \(\Ical_1(\vbar, \thbar)\) is defined as an integral on the whole real line \(\RR\), but \(\Bcal_1^+\) is defined as an integral on \(\O\).
Note that \(\vbar = \vbar_s\) on \(\O\), so for the first term, \eqref{ineq-b3} and \eqref{ineq-b4} give the bound
\begin{align*}
\int_\O a'\Big(\frac{\th}{v}-\frac{\thbar}{\vbar}\Big)^2 d\x
&\le C\int_\O a'\th^2\Big(\frac{1}{v}-\frac{1}{\vbar_s}\Big)^2 d\x
+ C\int_\O a'(\th-\thbar)^2 d\x \\
&\le C\Big(\frac{\e}{\l}(\Dcal_v(U)+\Dcal_\th(U)) + \e^2\Gcal_v(U)\Big).
\end{align*}
For the second and last term, we temporarily use the truncation of \(v\) and \(\th\) by \(\d_3/2\) instead of \(\d_3\), which is denoted by \((\vbar_s)_{\d_3/2}\) and \((\thbar)_{\d_3/2}\) respectively, i.e.,
\begin{equation}\label{half-truncation}
\begin{aligned}
1/((\vbar_s)_{\d_3/2})-1/\vt \coloneqq \psi^b_{\d_3/2}(1/v-1/\vt), \qquad
(\thbar)_{\d_3/2} - \tht \coloneqq \psi_{\d_3/2}(\th-\tht),
\end{aligned}
\end{equation}
as we did in \eqref{trunc-v-sb} and \eqref{trunc-v-def}.
Since \((y-\d_3/2)_+\ge \frac{\d_3}{2}\) whenever \((y-\d_3)_+>0\), we have 
\[
(y-\d_3)_+ \le (y-\d_3/2)_+\one{y-d_3>0} \le (y-\d_3/2)_+\Big(\frac{(y-\d_3/2)_+}{\d_3/2}\Big) = \frac{2}{\d_3}(y-\d_3/2)_+^2.
\]
This gives the following pointwise inequalities:
\begin{equation}\label{nonlinear-v1}
\abs{\frac{1}{v}-\frac{1}{\vbar_s}}\one{1/v\le 1/\vt+\d_3} \le C\Big(\frac{1}{v}-\frac{1}{(\vbar_s)_{\d_3/2}}\Big)^2,
\end{equation}
\begin{equation}\label{nonlinear-th}
|\th-\thbar| \le C (\th-(\thbar)_{\d_3/2})^2,
\end{equation}
\begin{equation}\label{nonlinear-v2}
\one{1/v>1/\vt+\d_3} \le C v\Big(\frac{1}{v}-\frac{1}{(\vbar_b)_{\d_3/2}}\Big)^2.
\end{equation}
Notice from \eqref{ineq-b1}, \eqref{ineq-b3}, \eqref{ineq-b4'}, and \eqref{ineq-b4} in Lemma \ref{lemma_hyp_ineq} that 
\begin{equation} \label{half-ineq}
\begin{aligned}
\int_\RR a'v\Big(\frac{1}{v}-\frac{1}{(\vbar_b)_{\d_3/2}}\Big)^2 d\x
&\le C\Big(\frac{\e}{\l}\Dcal_v(U) + \e^2\Gcal_v(U)\Big), \\
\int_\RR a'(\th-(\thbar)_{\d_3/2})^2 d\x
&\le C\frac{\e}{\l}\Dcal_\th(U), \\
\int_\RR a'\Big(\frac{1}{v}-\frac{1}{(\vbar_s)_{\d_3/2}}\Big)^2 d\x
&\le C \frac{\e}{\l}\Dcal_v(U), \\
\int_\RR a'\th^2\Big(\frac{1}{v}-\frac{1}{(\vbar_s)_{\d_3/2}}\Big)^2 d\x
&\le C\Big(\frac{\e}{\l}(\Dcal_v(U)+\Dcal_\th(U)) + \e^2\Gcal_v(U)\Big).
\end{aligned}
\end{equation}
These together with \eqref{nonlinear-v1} and \eqref{nonlinear-th} yield
\begin{equation}\label{ineq-p-linear}
\begin{aligned}
\int_\O a'\abs{\frac{\th}{v}-\frac{\thbar}{\vbar}} d\x
&\le C\int_\O a'\th\abs{\frac{1}{v}-\frac{1}{\vbar_s}} d\x 
+ C\int_\O a' |\th-\thbar| d\x \\
&\le C\int_\O a'(1+\th^2)\Big(\frac{1}{v}-\frac{1}{(\vbar_s)_{\d_3/2}}\Big)^2 d\x 
+ C\int_\O a'(\th-(\thbar)_{\d_3/2})^2 d\x \\
&\le C\Big(\frac{\e}{\l}(\Dcal_v(U)+\Dcal_\th(U)) + \e^2\Gcal_v(U)\Big).
\end{aligned}
\end{equation}
For the last term, we use \eqref{nonlinear-v2} and \eqref{half-ineq}\(_1\) to obtain that 
\begin{equation}
\label{mainbad-one}
\int_{\O^c} a'd\x
\le C\int_{\O^c} a'v\Big(\frac{1}{v}-\frac{1}{(\vbar_b)_{\d_3/2}}\Big)^2 d\x
\le C\Big(\frac{\e}{\l}\Dcal_v(U) + \e^2\Gcal_v(U)\Big).
\end{equation}
Summing up, we obtain the desired bound \eqref{bo1p}.

\bpf{bo1m}
We first observe
\begin{equation} \label{B1-decompose}
\abs{\Bcal_1^-(U)} 
\le C\int_{\O^c} a'\th\abs{\frac{1}{v}-\frac{1}{\vt}}\abs{u-\ut} d\x
+ C\int_{\O^c} a'\frac{1}{\vt} |\th-\tht| \abs{u-\ut} d\x.
\end{equation}
We estimate each term individually. Thus, we separate the first term into two parts. 
\[
\int_{\O^c} a'\th\abs{\frac{1}{v}-\frac{1}{\vt}}\abs{u-\ut} d\x 
\le \underbrace{\int_{\O^c} a'\th\abs{\frac{1}{v}-\frac{1}{\vbar}}\abs{u-\ut} d\x}_{\eqqcolon J_1}
+ \underbrace{\int_{\O^c} a'\th\abs{\frac{1}{\vbar}-\frac{1}{\vt}}\abs{u-\ut} d\x}_{\eqqcolon J_2}.
\]
To estimate \(J_1\), we split it again into two parts: 
\[
J_1 = \underbrace{\int_{\O^c} a'\th\abs{\frac{1}{v}-\frac{1}{\vbar}}\abs{u-\ut} \one{\d_3<1/v-1/\vt<2\d_3} d\x}_{\eqqcolon J_{11}}
+ \underbrace{\int_{\O^c} a'\th\abs{\frac{1}{v}-\frac{1}{\vbar}}\abs{u-\ut} \one{1/v-1/\vt\ge 2\d_3} d\x}_{\eqqcolon J_{12}}.
\]
Since \(\vbar = \vbar_b\) on \(\O^c\) and \(1-v/\vbar\) has positive lower bound on \(\{1/v-1/\vt\ge 2\d_3\}\), we have
\begin{align*}
J_{12} &\le 
\bigg(\int_{\O^c} a'\th^2\Big(\frac{1}{v}-\frac{1}{\vbar}\Big)^2\one{1/v-1/\vt\ge 2\d_3} d\x\bigg)^{1/2}
\Big(\int_{\O^c} a'(u-\ut)^2d\x\Big)^{1/2} \\
&\le C\bigg(\int_{\O^c} a'v^2\th^2\Big(\frac{1}{v}-\frac{1}{\vbar_b}\Big)^4 d\x\bigg)^{1/2}
\Big(\int_{\O^c} a'(u-\ut)^2d\x\Big)^{1/2}.
\end{align*}
On the other hand, using \((z+1)^2 \le C e^{\frac{C}{2}\abs{z}}\), it follows from \eqref{Qconst} and \eqref{locE} that for sufficiently small \(\d_0\), we have
\begin{equation}
\begin{aligned}\label{int-xi2}
&\int_\RR a'\Big(\abs{\x}+\frac{1}{\e}\Big)^2\one{1/v>1/\vtil+\d} d\x \\
&= \int_{\abs{\x}\ge\frac{1}{\e}\sqrt{\frac{\l}{\e}}} a'\Big(\abs{\x}+\frac{1}{\e}\Big)^2\one{1/v>1/\vtil+\d} d\x 
+ \int_{\abs{\x}<\frac{1}{\e}\sqrt{\frac{\l}{\e}}} a'\Big(\abs{\x}+\frac{1}{\e}\Big)^2\one{1/v>1/\vtil+\d} d\x \\
&\le C\frac{\l}{\e^2} \int_{\abs{z}\ge\sqrt{\frac{\l}{\e}}} e^{-C\abs{z}}(\abs{z}+1)^2 dz + C\frac{\l}{\e^3}\int_\RR a'\one{1/v>1/\vtil+\d} d\x \\
&\le C\frac{\l}{\e^2}e^{-\frac{C}{2}\sqrt{\frac{\l}{\e}}} + \frac{C}{\e}\le \frac{C}{\e}.
\end{aligned}
\end{equation}
Then, from \eqref{pwv1}, \eqref{pwv2'}, \eqref{int-xi2} with \(\int_{\x_0}^\x 1 d\z \le \abs{\x}+\abs{\x_0}\le\abs{\x}+\frac{1}{\e}\), we have
\begin{align*}
&\int_{\O^c} a'v^2\th^2\Big(\frac{1}{v}-\frac{1}{\vbar_b}\Big)^4 d\x
\le C\int_{\O^c} a'v^2\Big(\frac{1}{v}-\frac{1}{\vbar_b}\Big)^4 d\x
+ C\int_{\O^c} a'v^2\th^4\Big(\frac{1}{v}-\frac{1}{\vbar_b}\Big)^4 d\x \\
&\le C\left(\Dcal_v(U)+\Dcal_\th(U)\right)^2\int_{\O^c} a'\Big(\abs{\x}+\frac{1}{\e}\Big)^2\one{1/v>1/\vtil+\d} d\x
+ C\frac{\e^4}{\l^2}\left(\Dcal_v(U) + \Gcal_v(U) + \Gcal_\th(U)\right)^2 \\
&\le 
\frac{C}{\e}(\Dcal_v(U)+\Dcal_\th(U))^2 + C\e^2\left(\Gcal_v(U)+\Gcal_\th(U)\right)^2.
\end{align*}
Summing up, we obtain that 
\[
J_{12}\le C\sqrt{\frac{\e}{\l}}(\Dcal_v(U)+\Dcal_\th(U)) + C\e(\Gcal_v(U)+\Gcal_\th(U)).
\]
For the remaining terms, from \eqref{ineq-b3} and \eqref{mainbad-one}, we have 
\begin{align*}
J_2+J_{11} &\le C\int_{\O^c}a'\th\abs{u-\ut} d\x 
\le C\sqrt{\frac{\e}{\l}}\int_{\O^c}a'(u-\ut)^2 d\x + C\sqrt{\frac{\l}{\e}}\int_{\O^c}a'\th^2 d\x \\
&\le C\sqrt{\frac{\e}{\l}}\Gcal_u^-(U) + C\sqrt{\frac{\l}{\e}}\int_{\O^c} a'(\th-\thbar)^2 d\x
+ C\sqrt{\frac{\l}{\e}}\int_{\O^c} a' d\x \\
&\le C\Big(\sqrt{\frac{\e}{\l}}\Gcal_u^-(U) + \sqrt{\frac{\e}{\l}}(\Dcal_\th(U)+\Dcal_v(U)) + \e\Gcal_v(U)\Big).
\end{align*}
Therefore, we get
\begin{equation}\label{mainbad-v}
\int_{\O^c} a'\th\abs{\frac{1}{v}-\frac{1}{\vt}}\abs{u-\ut} d\x 
\le C\Big(\sqrt{\frac{\e}{\l}}(\Gcal_u^-(U)+\Dcal_v(U)+\Dcal_\th(U)) + \e(\Gcal_v(U)+\Gcal_\th(U))\Big).
\end{equation}

For the second term on the right-hand side of \eqref{B1-decompose}, we divide it into two parts: 
\[
\int_{\O^c} a'\frac{1}{\vt} |\th-\tht|\abs{u-\ut}d\x
\le C\underbrace{\int_{\O^c} a' |\th-\thbar|\abs{u-\ut} d\x}_{\eqqcolon J_3}
+ C\underbrace{\int_{\O^c} a' |\thbar-\tht|\abs{u-\ut} d\x}_{\eqqcolon J_4}
\]
To estimate \(J_3\), we split it once again into two parts: 
\[
J_3 = \underbrace{\int_{\O^c} a'|\th-\thbar|\abs{u-\ut}\one{\d_3<|\th-\thtil|<2\d_3} d\x}_{\eqqcolon J_{31}}
+ \underbrace{\int_{\O^c} a' |\th-\thbar|\abs{u-\ut}\one{|\th-\thtil|\ge 2\d_3} d\x}_{\eqqcolon J_{32}}.
\]
Since \(|\th-\thbar|\) and \(1/v\) have positive lower bounds on  \(\{|\th-\tht|\ge 2\d_3\}\cap \O^c\), \eqref{locE} implies
\begin{align*}
J_{32} &\le \Big(\int_{\O^c} a'(\th-\thbar)^2\one{|\th-\thtil|\ge 2\d_3}d\x\Big)^{\frac{1}{2}}\Big(\int_{\O^c} a'(u-\ut)^2d\x\Big)^{\frac{1}{2}}
\le \sqrt{\frac{\e^2}{\l}} \Big(\int_{\O^c}a'\frac{(\th-\thbar)^4}{v^2}\Big)^{\frac{1}{2}}.
\end{align*}
Using \eqref{pwth2'}, a simple variation of \eqref{int-xi2} and \(\int_{\x_0}^\x 1 d\z \le\abs{\x}+\frac{1}{\e}\), we obtain
\begin{equation}\label{mainbad-thpw}
\begin{aligned}
\int_{\O^c} a'\frac{(\th-\thbar)^4}{v^2} d\x
&\le \frac{C}{\e}(\Dcal_v(U)+\Dcal_\th(U))^2
+ C\frac{\e^4}{\l^2}\Dcal_\th(U)^2
\le \frac{C}{\e}(\Dcal_v(U)+\Dcal_\th(U))^2.
\end{aligned}
\end{equation}
This together with \eqref{locE} gives 
\[
J_{32}\le C\sqrt{\frac{\e}{\l}}\left(\Dcal_v(U)+\Dcal_\th(U)\right).
\]
For \(J_4\) and \(J_{31}\), by using \eqref{mainbad-one}, we have
\begin{align*}
J_4 + J_{31} &\le C\int_{\O^c} a'\abs{u-\ut} d\x
\le C\sqrt{\frac{\e}{\l}}\int_{\O^c} a'(u-\ut)^2 d\x + C\sqrt{\frac{\l}{\e}}\int_{\O^c} a' d\x \\
&\le C\sqrt{\frac{\e}{\l}}\Gcal_u^-(U) + C\sqrt{\frac{\e}{\l}}\Dcal_v(U) + C\e\Gcal_v(U).
\end{align*}
Summing up, these two with \eqref{mainbad-v}, we obtain the desired estimate \eqref{bo1m}.

\bpf{bo2} We consider the following decomposition: 
\[
\abs{\Bcal_2(U)-\Bcal_2(\Ubar)}
\le \underbrace{\int_\RR aR\abs{\frac{\th}{v}-\frac{\thbar}{\vbar}}\abs{\ut'} d\x}_{\eqqcolon J_1}
+ \underbrace{\int_\RR a\frac{R}{\vt} |\th-\thbar| \abs{\ut'} d\x}_{\eqqcolon J_2}
+ \underbrace{\int_\RR a\frac{\pt}{\vt}\abs{v-\vbar}\abs{\ut'} d\x}_{\eqqcolon J_3}.
\]
For \(J_1\), we split the integral into four pieces and use \eqref{der-scale} as follows: 
\begin{align*}
J_1 
&\le \underbrace{C\frac{\e}{\l}\int_\RR a'\frac{|\th-\thbar|}{\vbar} d\x}_{\eqqcolon J_{11}}
+ \underbrace{C\frac{\e}{\l}\int_\RR a'\th\abs{\frac{1}{v}-\frac{1}{\vbar}}\one{1/v>1/\vt+2\d_3} d\x}_{\eqqcolon J_{12}} \\
&\qquad
+ \underbrace{C\frac{\e}{\l}\int_\RR a' |\th-\thbar|\one{1/v\le 1/\vt+2\d_3} d\x}_{\eqqcolon J_{13}}
+ \underbrace{C\frac{\e}{\l}\int_\RR a'\thbar\abs{v-\vbar}\one{1/v\le 1/\vt+2\d_3} d\x}_{\eqqcolon J_{14}}.
\end{align*}
For \(J_{12}\), since \(\vbar=\vbar_b\) on \(\{1/v>1/\vtil+2\d_3\}\), \eqref{ineq-b1} and \eqref{ineq-b2} with Young's inequality imply
\begin{align*}
J_{12} &\le C\frac{\e}{\l}\int_\RR a'v(1+\th^2)\Big(\frac{1}{v}-\frac{1}{\vbar_b}\Big)^2 d\x 
\le C\frac{\e}{\l}\Big(\frac{\e}{\l}(\Dcal_v(U)+\Dcal_\th(U)) + \e^2(\Gcal_v(U)+\Gcal_\th(U))\Big).
\end{align*}
For the other terms, we utilize \eqref{rel_Phi1} and \eqref{der-scale} to obtain that 
\begin{align}
J_{11} + J_{13} + J_2 &\le C\frac{\e}{\l}\int_\RR a'\Big[\Phif{\th}{\tht}-\Phif{\thbar}{\tht}\Big] d\x \le C\frac{\e}{\l}(\Gcal_\th(U)-\Gcal_\th(\Ubar)), \label{ineq-Phi_rel-th} \\
J_{14} + J_3 &\le C\frac{\e}{\l}\int_\RR a'\Big[\Phif{v}{\vt}-\Phif{\vbar}{\vt}\Big] d\x \le C\frac{\e}{\l}(\Gcal_v(U)-\Gcal_v(\Ubar)).
\label{ineq-Phi_rel-v}
\end{align}
Summing up, we obtain the desired bound.

\noindent{\textbf{Proof of}} \eqref{bo3} and \eqref{bo4}\textbf{:}
With \eqref{der-scale}, we proceed as in \eqref{ineq-Phi_rel-th} and \eqref{ineq-Phi_rel-v}: 
\[
\abs{\Bcal_3(U)-\Bcal_3(\Ubar)}\le C\frac{\e}{\l}\int_\RR a'\Big[\Phif{v}{\vt}-\Phif{\vbar}{\vt}\Big] d\x \le C\frac{\e}{\l}(\Gcal_v(U)-\Gcal_v(\Ubar)),
\]\[
\abs{\Bcal_4(U)-\Bcal_4(\Ubar)}\le C\frac{\e}{\l}\int_\RR a'\Big[\Phif{\th}{\tht}-\Phif{\thbar}{\tht}\Big] d\x \le C\frac{\e}{\l}(\Gcal_\th(U)-\Gcal_\th(\Ubar)).
\]

\bpf{bo} First, using \eqref{rel_Phi_v} and \eqref{rel_Phi_th}, we obtain that 
\begin{equation} \label{barsmall}
\abs{\Ical_1(\vbar,\thbar)} + \abs{\Bcal_2(\Ubar)} + \abs{\Bcal_3(\Ubar)} + \abs{\Bcal_4(\Ubar)} \le C\int_\RR a'\Phif{v}{\vt} d\x + C\int_\RR a'\Phif{\th}{\tht} d\x \le C\frac{\e^2}{\l}.
\end{equation}
Then, from \eqref{bo1p}-\eqref{bo4} with \eqref{locE} and \eqref{l2-v}-\eqref{l2-th}, we conclude that \eqref{bo} as we desired.
This completes the proof of Proposition \ref{prop_hyperbolic_out}.
\qed

\subsection{Control of parabolic terms} \label{section-parabolic}
We will control the parabolic bad terms  \(\Bcal^v(U), \Bcal^u(U)\), and \(\Bcal^\th(U)\) of \eqref{ybg-first}. 
Recall \(\Gcal_u^+(U)\) of \eqref{ggd} and that \(\Dcal(U)=\Dcal_v(U)+\Dcal_u(U)+\Dcal_\th(U)\).

\begin{proposition}\label{prop_parabolic_out}
Under the same assumptions as Proposition \ref{prop_hyperbolic_out}, we have
\begin{align}
&\begin{aligned} \label{para-v}
\abs{\Bcal^v(U)} \le C \left( \left(\l+\frac{\e}{\l}\right)\Dcal(U) + \e(\Gcal_v(U)+\Gcal_\th(U)) \right)
\end{aligned} \\
&\begin{aligned} \label{para-th}
\abs{\Bcal^\th(U)} \le C \left( \left(\l+\frac{\e}{\l}\right)\Dcal(U) + \e(\Gcal_v(U) + \Gcal_\th(U)) \right)
\end{aligned} \\
&\begin{aligned} \label{para-u}
&\abs{\Bcal^u(U)} \le C \Big( \left(\l+\frac{\e}{\l}\right)\Dcal(U)
+ \e\Gcal(U)\Big)
\end{aligned} \\
&\begin{aligned} \label{para-total}
&\abs{\Jcal^{para}(U)} \le C \Big( \left(\l+\frac{\e}{\l}\right)\Dcal(U)
+ \e\Gcal(U) \Big)
\end{aligned}
\end{align}
We also have \(C^*>0\) which depends on \(\d_3>0\) such that 
\begin{equation}\label{para-tot}
\abs{\Jcal^{para}(U)} \le C^* \Big(\frac{\e^2}{\l} + \Big(\l+\frac{\e}{\l}\Big)\Dcal(U)\Big).
\end{equation}
\end{proposition}
Note that \eqref{para-total} is a direct consequence of \eqref{para-v}-\eqref{para-u}.
We prove them in the following subsections, by using Lemma \ref{lemma_para_ineq}.
Moreover, \eqref{para-tot} follows from \eqref{para-total} combined with \eqref{locE} and \eqref{trans-bound}, which is also stated in Lemma \ref{lemma_para_ineq}. 

Among the aforementioned bounds that include \(\Dcal(U)\), the bounds in \eqref{para-v} and \eqref{para-th} are independent of \(\Dcal_u(U)\).
This will be demonstrated in the proof, though it does not carry any significance for the subsequent analysis.

\begin{lemma} \label{lemma_para_ineq}
Under the same assumptions as Proposition \ref{prop_hyperbolic_out}, there is a constant \(C>0\) (depending only on \(\d_3>0\)) which makes the following inequalities be true:
\begin{align}
&\int_\RR a' v \Big(\frac{1}{v}-\frac{1}{\vt}\Big)^2 d\x
\le C \Big( \Gcal_v(U) + \frac{\e}{\l}\Dcal_v(U) \Big)
\label{ineq-p1} \\
&\int_\RR \abs{a'}^2 v\th^2 \Big( \frac{1}{v}-\frac{1}{\vt} \Big)^2 d\x
\le C\e\l (\Dcal(U) + \Gcal_v(U)+\Gcal_\th(U))
\label{ineq-p2} \\
&\int_\RR \abs{a'}^2 v (\th-\tht)^2 d\x
\le C \Big( \e^2 \Dcal_\th(U)+\e\l(\Gcal_v(U)+\Gcal_\th(U))\Big)
\label{ineq-p3} \\
&\abs{\frac{\sqrt{\th}}{\sqrt{v}}(\th-\thbar_s)(\x)}
\le C\Big(\sqrt{\Dcal(U)}
\sqrt{\int_{\x_0}^{\x}\th\one{\th>\tht+\d_3}d\z+\frac{\e^2}{\l}}
+ \frac{\e}{\l}(\Gcal_v(U) + \Gcal_\th(U))\Big)
\label{pwth3} \\
&\int_\RR \abs{a'}^2 \frac{\th^2}{v}(\th-\tht)^2 d\x
\le C \Big( \e^2\Dcal(U) + \e\l(\Gcal_v(U)+\Gcal_\th(U)) \Big)
\label{ineq-p4} \\
&\int_\RR a' \frac{(\th-\tht)^2}{v} d\x
\le C\Big( \Gcal_\th(U) + \frac{\e}{\l}\Dcal(U) + \e^2 \Gcal_v(U) \Big)
\label{ineq-p5} \\
&\int_\RR a' (\th-\tht)^2 d\x
\le C \Big( \Gcal_\th(U) + \frac{\e}{\l} \Dcal_\th(U) \Big)
\label{ineq-p6} \\ 
&\int_\RR a'(u-\ut)^2 d\x \le \Gcal_u^-(U) + C\Big(\Gcal_u^+(U) + \Gcal_v(U) + \Gcal_\th(U) + \frac{\e}{\l}\Dcal(U)\Big)
\label{ineq-aux-u} \\
&\int_\RR \abs{a'}^2 \frac{\th^3}{v}(u-\ut)^2 d\x
\le C \Big( \e^2\Dcal(U) + \e\l\Gcal(U) \Big)
\label{ineq-p7}\\
&\begin{aligned} \label{pwu}
&\abs{\frac{\th}{\sqrt{v}}(u-\ubar)(\x)}
\le C \bigg(\sqrt{\Dcal_u(U)\int_{\x_0}^{\x} \th \one{\abs{u-\ut}>\d_3}d\z}
+ \sqrt{\Dcal(U)}\sqrt{\int_{\x_0}^\x (u-\ubar)^2d\z} \\
&\hspace{40mm}
+\e\sqrt{\frac{\e}{\l}\Gcal(U)}
\sqrt{\int_{\x_0}^\x (u-\ubar)^2d\z} \bigg).
\end{aligned} \\
&\begin{aligned} \label{ineq-p8}
&\int_\RR a' \frac{\th^2}{v}(u-\ut)^2 d\x
\le C\Big(\Gcal(U) + \frac{\e}{\l}\Dcal(U)\Big)
\end{aligned} \\
&\int_\RR \abs{a'}^2 \frac{\th}{v} (\th-\tht)^2 d\x
\le C \Big( \e^2\Dcal(U) + \e\l (\Gcal_v(U)+\Gcal_\th(U)) \Big)
\label{ineq-p9} \\
&\Gcal_u^+(U) \le C\Big(\frac{\e}{\l}\Dcal(U)+\frac{\e^2}{\l}\Big)
\label{trans-bound}
\end{align}
\end{lemma}
Note that among the bounds above that include \(\Dcal(U)\), the bound in \eqref{ineq-p8} is the only one that actually depends on \(\Dcal_u(U)\).

\begin{proof}
\bpf{ineq-p1} We decompose the left-hand side of \eqref{ineq-p1} as follows:
\begin{align*}
\int_\RR a'v\Big(\frac{1}{v}-\frac{1}{\vt}\Big)^2 d\x
\le \underbrace{\int_\O a'v\Big(\frac{1}{v}-\frac{1}{\vt}\Big)^2 d\x}_{\eqqcolon J_1} 
+ \underbrace{C\int_{\O^c} a'v\Big(\frac{1}{v}-\frac{1}{\vbar_b}\Big)^2 d\x}_{\eqqcolon J_2} 
+ \underbrace{C\int_{\O^c} a'v\Big(\frac{1}{\vbar_b}-\frac{1}{\vt}\Big)^2 d\x}_{\eqqcolon J_3}.
\end{align*}
Then, from \eqref{rel_Phi_v}, \eqref{Phi-sim} and \eqref{Qconst}, we have 
\[
J_1 + J_3 \le C\int_\RR a'\Phif{v}{\vt} d\x \le C\Gcal_v(U).
\]
This together with \eqref{ineq-b1} gives the desired result \eqref{ineq-p1}.

\bpf{ineq-p2} 
First of all, we obtain \eqref{Linfty-th1'}, an \(L^\infty\) bound of \(a'\th^2\), using Lemma \ref{lemma_Linfty}.
Since \(v\le C+Cv\one{1/v\le1/\vt+\d_3}\le C+C\Phi(v/\vt)\), it follows from Lemma \ref{lemma_Linfty} with \eqref{pwth1}, \eqref{der-a} and \eqref{locE} that
\begin{equation}\label{Linfty-th1}
a'(\th-\thbar)^2 (\x) \le C\Dcal_\th(U)a'(\x)\int_{\x_0}^\x v(\z) d\z
\le C\Dcal_\th(U) \int_\RR a'v d\x \le C\l\Dcal_\th(U). 
\end{equation}
This with \eqref{der-scale} also gives that 
\begin{equation}\label{Linfty-th1'}
a'\th^2(\x)\le Ca'(\thbar^2 + (\th-\thbar)^2)(\x) \le C\e\l + C\l\Dcal_\th(U).
\end{equation}
Now, we decompose the left-hand side of \eqref{ineq-p2}: 
\begin{align*}
&\int_\RR \abs{a'}^2 v\th^2\Big(\frac{1}{v}-\frac{1}{\vt}\Big)^2 d\x \\
&\le \underbrace{C\int_\O \abs{a'}^2 v\th^2\Big(\frac{1}{v}-\frac{1}{\vt}\Big)^2 d\x}_{\eqqcolon J_1} 
+ \underbrace{C\int_{\O^c} \abs{a'}^2 v\th^2\Big(\frac{1}{v}-\frac{1}{\vbar_b}\Big)^2 d\x}_{\eqqcolon J_2} 
+ \underbrace{C\int_{\O^c} \abs{a'}^2 v\th^2\Big(\frac{1}{\vbar_b}-\frac{1}{\vt}\Big)^2 d\x}_{\eqqcolon J_3}.
\end{align*}
Here, \eqref{rel_Phi_v}, \eqref{Phi-sim}, \eqref{locE} and \eqref{Qconst} together with the \(L^\infty\) bound \eqref{Linfty-th1'} give
\[
J_1 + J_3 \le C(\e\l + \l\Dcal_\th(U))\int_\RR a'\Phif{v}{\vt} d\x
\le C\e^2\Dcal_\th(U) + C\e\l\Gcal_v(U).
\]
This and \eqref{ineq-b2} combined with \eqref{der-scale} imply \eqref{ineq-p2}.

\bpf{ineq-p3} We decompose the left-hand side of \eqref{ineq-p3} as follows: 
\begin{align*}
\int_\RR \abs{a'}^2v(\th-\tht)^2 d\x
&\le \underbrace{\int_\RR \abs{a'}^2v(\th-\thbar)^2\one{1/v\ge 1/\vt-\d_3} d\x}_{\eqqcolon J_1}
+ \underbrace{\int_\RR \abs{a'}^2v(\thbar-\tht)^2\one{1/v\ge 1/\vt-\d_3} d\x}_{\eqqcolon J_2} \\
&\hspace{2mm}
+ \underbrace{\int_\RR \abs{a'}^2v(\th-\thbar)^2\one{1/v<1/\vt-\d_3} d\x}_{\eqqcolon J_3} 
+ \underbrace{\int_\RR \abs{a'}^2v(\thbar-\tht)^2\one{1/v<1/\vt-\d_3} d\x}_{\eqqcolon J_4}.
\end{align*}
For \(J_1\) we use \eqref{ineq-b3}, and for \(J_2\) \eqref{rel_Phi_th} and \eqref{Phi-sim} imply
\[
J_1 \le C\e\l\int_\RR a'(\th-\thbar)^2 d\x \le C\e^2\Dcal_\th(U), \qquad
J_2 \le C\e\l\Gcal_\th(U).
\] 
For \(J_3\), \eqref{Linfty-th1} with \eqref{Qlin} and \eqref{locE} gives 
\[
J_3 \le C\l\Dcal_\th(U)\int_\RR a'\Phif{v}{\vt} d\x
\le C\e^2\Dcal_\th(U).
\]
Finally, for \(J_4\), we use \eqref{der-scale} to have \(a'(\thbar-\tht)^2 \le C\e\l\d_3^2\) so that \eqref{Qlin} implies
\[
J_4 \le C\e\l\Gcal_v(U).
\]
Summing up, we obtain the desired bound. 

\bpf{pwth3} By the fundamental theorem of calculus, \eqref{der-scale}, and \eqref{Qconst}, for any \(\x\in\RR\),
\begin{align*}
\abs{\frac{\sqrt{\th}}{\sqrt{v}}(\th-\thbar_s)(\x)}
&\le C\int_{\x_0}^{\x} \bigg[\abs{\frac{\sqrt{\th}}{\sqrt{v}}(\th-\thbar_s)_\x}
+ \abs{\frac{1}{\sqrt{v\th}}(\th-\thbar_s)(\th-\thtil)_\x}
+ \abs{\frac{1}{\sqrt{v\th}}(\th-\thbar_s)\tht'} \\
&\qquad\qquad\qquad
+ \abs{\sqrt{v\th}(\th-\thbar_s)\Big(\frac{1}{v}-\frac{1}{\vtil}\Big)_\x}
+ \abs{\sqrt{v\th}(\th-\thbar_s)\vt'}\bigg] d\x \\
&\hspace{-3mm}\le C\sqrt{\Dcal(U)}\sqrt{\int_{\x_0}^{\x}\th\one{\th>\tht+\d_3}d\z}
+ C\frac{\e}{\l}\sqrt{\Gcal_\th(U)}\sqrt{\int_\RR a'\frac{\th}{v}\one{\th>\thtil+\d_3} d\x} \\
&\qquad
+ C\frac{\e}{\l}\sqrt{\int_\RR \abs{a'}^2v(\th-\thbar_s)^2d\x}\sqrt{\int_{\x_0}^\x\th\one{\th>\tht+\d_3}d\z}.
\end{align*}
To bound the integral in this estimate, we consider the inequality 
\begin{equation}\label{v-2delta}
\frac{1}{v} \le C + Cv\Big(\frac{1}{v}-\frac{1}{\vbar_b}\Big)^2.
\end{equation}
Then, this with \eqref{ineq-b2} and \eqref{Qlin'} implies 
\begin{align*}
\int_\RR a'\frac{\th}{v}\one{\th>\tht+\d_3} d\x
&\le C\int_\RR a'\th\one{\th>\tht+\d_3} d\x + C\int_\RR a' v\th^2 \Big(\frac{1}{v}-\frac{1}{\vbar_b}\Big)^2d\x \\
&\le C\left(\Gcal_\th(U) + \frac{\e}{\l}\Dcal(U) + \e^2 \Gcal_v(U)\right).
\end{align*}
Further, from the bound for \(J_1 + J_3\) in the proof of \eqref{ineq-p3}, we have 
\[
\int_\RR \abs{a'}^2v(\th-\thbar_s)^2 d\x \le C\e^2\Dcal_\th(U).
\]
This together with \eqref{locE} gives the desired bound:
\[
\abs{\frac{\sqrt{\th}}{\sqrt{v}}(\th-\thbar_s)(\x)}
\le C\bigg(\sqrt{\Dcal(U)}
\sqrt{\int_{\x_0}^{\x}\th\one{\th>\tht+\d_3}d\z+\frac{\e^2}{\l}}
+ \frac{\e}{\l}(\Gcal_v(U) + \Gcal_\th(U))\bigg).
\]

\bpf{ineq-p4} 
We first use Lemma \ref{lemma_Linfty} to establish the following pointwise estimates.\\
From \eqref{pwth3}, we use \eqref{locE} and \eqref{Qlin'} to find that for any \(\x\in\RR\),
\begin{equation}
\begin{aligned}\label{Linfty-th2}
a'\frac{\th}{v}(\th-\thbar_s)^2(\x)
&\le C\Dcal(U)\Big(\int_\RR a'\th\one{\th>\tht+\d_3}d\x + \e^3\Big)
+ C\e^2(\Gcal_v(U)+\Gcal_\th(U)) \\
&\le C\frac{\e^2}{\l}\Dcal(U) + C\e^2(\Gcal_v(U)+\Gcal_\th(U)).
\end{aligned}
\end{equation}
Moreover, \eqref{v-2delta}, \eqref{pwv1}, \eqref{Qconst}, \eqref{locE} and \eqref{der-scale} give another \(L^\infty\) bound:
\begin{equation}\label{Linfty-v}
a'\frac{1}{v}(\x) \le Ca'(\x) + C\Dcal(U)\int_{\x_0}^\x a'\Phif{v}{\vtil}(\z)d\z + Ca'(\x)\frac{\e^2}{\l}\Gcal_v(U) 
\le C\e\l + C\frac{\e^2}{\l}\Dcal(U).
\end{equation}
Combining these two results and using \eqref{locE}, we obtain several other bounds. 
\begin{align}
a'\frac{\th^3}{v}(\x) &\le Ca'\frac{1}{v}(\x) + Ca'\frac{\th}{v}(\th-\thbar_s)^2(\x)
\le C\e\l + C\frac{\e^2}{\l}\Dcal(U), 
\label{Linfty-th3v} \\
a'\frac{\th + \th^2}{v}(\x) &\le a'\frac{\th^3}{v}(\x) + a'\frac{1}{v}(\x)
\le C\e\l + C\frac{\e^2}{\l}\Dcal(U).
\label{Linfty-thv} 
\end{align}
Now we decompose the left-hand side of \eqref{ineq-p4} as follows: 
\begin{multline*}
\int_\RR \abs{a'}^2\frac{\th^2}{v}(\th-\tht)^2 d\x \\
\le \underbrace{C\int_\RR \abs{a'}^2\frac{\th^2}{v}(\th-\thbar_s)^2 d\x}_{\eqqcolon J_1}
+ \underbrace{C\int_\RR \abs{a'}^2\frac{\th^2}{v}(\thbar-\tht)^2\one{\th\ge \tht-\d_3} d\x}_{\eqqcolon J_2}
+ \underbrace{C\int_\RR \abs{a'}^2\frac{1}{v}\one{\th<\tht-\d_3} d\x}_{\eqqcolon J_3}.
\end{multline*}
Using \eqref{Linfty-th2}, \eqref{locE} and \eqref{Qlin'}, we have 
\[
J_1 \le C\frac{\e^2}{\l}\Big(\frac{\e^2}{\l}\Dcal(U) + \e^2(\Gcal_v(U)+\Gcal_\th(U))\Big).
\]
Using \eqref{Linfty-thv}, \eqref{rel_Phi_v} and \eqref{Phi-sim} for \(J_2\) and using \eqref{Linfty-v} and \eqref{Qconst} for \(J_3\), we have
\[
J_2 + J_3 \le C\left(\e\l\Gcal_\th(U) + \e^2\Dcal(U)\right).
\]
Summing up, we obtain the desired bound.

\bpf{ineq-p5} We decompose the left-hand side of \eqref{ineq-p5} with \eqref{v-2delta} as follows: 
\begin{equation*}
\int_\RR a'\frac{(\th-\tht)^2}{v} d\x 
\le \underbrace{C\int_\RR a'\frac{(\th-\thbar)^2}{v} d\x}_{\eqqcolon J_1}
+ \underbrace{C\int_\RR a'(\thbar-\tht)^2 d\x}_{\eqqcolon J_2}
+ \underbrace{C\int_\RR a'v\Big(\frac{1}{v}-\frac{1}{\vbar_b}\Big)^2 d\x}_{\eqqcolon J_3}.
\end{equation*}
For \(J_1\), using \eqref{pwth2'}, \eqref{Qconst}, \eqref{locE}, Lemma \ref{lemma_pushing} and \eqref{weight-a}, we obtain
\begin{equation}
\begin{aligned}\label{ineq-aux-th}
J_1 &\le C\Dcal(U)\int_\RR a' \int_{\x_0}^\x \one{|\th-\thbar|>0} d\z d\x
+ C\e^2 \Dcal(U) \le C\frac{\e}{\l}\Dcal(U).
\end{aligned}
\end{equation}
For \(J_2\), \eqref{rel_Phi_th} and \eqref{Phi-sim} imply \(J_2 \le C \Gcal_\th(U)\),
and for \(J_3\), \eqref{ineq-b1} provides a proper bound.
Thus, summing up, we establish the desired bound.

\bpf{ineq-p6} It follows from \(J_2\) in the proof of \eqref{ineq-p5} and \eqref{ineq-b3} that 
\[
\int_\RR a'(\th-\tht)^2 d\x
\le C\int_\RR a'(\thbar-\tht)^2 d\x + C\int_\RR a'(\th-\thbar)^2 d\x
\le C\Gcal_\th(U) + C\frac{\e}{\l}\Dcal_\th(U).
\]

\bpf{ineq-aux-u} Recall from \eqref{badgood} that the forms of hyperbolic good terms are different on the sets \(\O\) and \(\O^c\). Then, Young's inequality implies
\begin{align*}
\int_\O a'(u-\ut)^2 d\x &\le C\Gcal_u^+(U) + C\int_\O a'(p-\pt)^2 d\x \\
&\le C\Gcal_u^+(U) + C\abs{\Bcal_1^+(U)-\Bcal_1^+(\Ubar)} + C\abs{\Bcal_1^+(\Ubar)}.
\end{align*}
For the second term, we use \eqref{bo1p}, and for the last term, we use \eqref{rel_Phi_v}-\eqref{Phi-sim} to get 
\begin{equation}\label{B_1bar}
\abs{\Bcal_1^+(\Ubar)} \le C\int_\O a'\thbar^2\Big(\frac{1}{\vbar}-\frac{1}{\vt}\Big)^2 d\x 
+ C\int_\O a'\frac{(\thbar-\tht)^2}{\vt^2} d\x 
\le C\Gcal_v(U) + C\Gcal_\th(U).
\end{equation}
Summing up, we obtain the desired bound. 

\bpf{ineq-p7} We use \eqref{Linfty-th3v} with \eqref{locE} and \eqref{ineq-aux-u} to obtain that 
\begin{align*}
\int_\RR \abs{a'}^2\frac{\th^3}{v}(u-\ut)^2 d\x 
&\le C\Big(\e\l + \frac{\e^2}{\l}\Dcal(U)\Big)\int_\RR a'(u-\ut)^2 d\x 
\le C\e^2 \Dcal(U) +C\e\l\Gcal(U).
\end{align*}

\bpf{pwu} By the fundamental theorem of calculus and \eqref{der-scale}, for any \(\x\in\RR\),
\begin{align*}
\abs{\frac{\th}{\sqrt{v}}(u-\ubar)(\x)} 
&\le 
C\bigg(\sqrt{\Dcal_u(U)\int_{\x_0}^{\x}\th \one{\abs{u-\ut}>\d_3} d\z} 
+ \sqrt{\Dcal(U)\int_{\x_0}^{\x}(u-\ubar)^2d\z} \\
&\hspace{25mm}
+ \frac{\e}{\l}\sqrt{\int_\RR \abs{a'}^2\Big(\frac{1}{v}+v\th^2\Big) \one{\abs{u-\util}>\d_3}d\x}\sqrt{\int_{\x_0}^\x (u-\ubar)^2d\z}\bigg).
\end{align*}
By \eqref{v-2delta}, \eqref{ineq-aux-u}, \eqref{ineq-b1} and \eqref{der-scale}, we have
\begin{align*}
\int_\RR \abs{a'}^2\frac{1}{v}\one{\abs{u-\ut}>\d_3} d\x 
&\le C\e\l\int_\RR a'(u-\ut)^2 d\x + C\e\l\int_\RR a'v\Big(\frac{1}{v}-\frac{1}{\vbar_b}\Big)^2 d\x \\
&\le C\e\l\Gcal(U) + C\e^2\Dcal(U).
\end{align*}
Similarly, from \(v\one{\abs{u-\ut}>\d_3}\le C\Phi(v/\vt)+C(u-\ut)^2\) with \eqref{ineq-p3} and \eqref{ineq-aux-u} we have
\begin{align*}
\int_\RR \abs{a'}^2v\th^2\one{\abs{u-\ut}>\d_3} d\x
\le C\e\l\Gcal(U)+C\e^2\Dcal(U).
\end{align*}
This gives the desired inequality.

\bpf{ineq-p8} We decompose the left-hand side of \eqref{ineq-p8} into:
\[
\int_\RR a'\frac{\th^2}{v}(u-\ut)^2d\x 
\le C\int_\RR a'\frac{\th^2}{v}(u-\ubar)^2d\x 
+ C\int_\RR a'\frac{\th^2}{v}(\ubar-\ut)^2d\x
\eqqcolon J_1+J_2.
\]
For \(J_1\), we use \eqref{pwu}, \eqref{locE} and Lemma \ref{lemma_pushing} with \(\th\one{\abs{u-\ut}>\d_3}\le C\Phi(\th/\tht)+C(u-\ut)^2\):
\begin{align*}
J_1 
&\le C\Dcal_u(U)\frac{1}{\e}\int_\RR a'\th\one{\abs{u-\ut}>\d_3} d\x
+ C\Big(\Dcal_v(U)+\Dcal_\th(U)+\frac{\e^3}{\l}\Gcal(U)\Big)\frac{1}{\e}\int_\RR a'(u-\ubar)^2d\x \\
&\le C\frac{\e}{\l}\Dcal(U) + C\e^2\Gcal(U).
\end{align*}
For \(J_2\), we use 
\[
\frac{\th^2}{v}\le C + Cv\Big(\frac{1}{v}-\frac{1}{\vbar_b}\Big)^2 + C\frac{(\th-\tht)^2}{v},
\]
which together with \eqref{ineq-aux-u}, \eqref{ineq-b1}, and \eqref{ineq-p5} implies
\begin{align*}
J_2 
&\le C\Gcal(U) + C\frac{\e}{\l}\Dcal(U).
\end{align*}
Summing up, we obtain the desired inequality.

\bpf{ineq-p9} For this, we use Young's inequality with \eqref{ineq-p4}, \eqref{ineq-p5} and \eqref{der-scale}:
\begin{align*}
\int_\RR \abs{a'}^2 \frac{\th}{v}(\th-\tht)^2 d\x
&\le C\int_\RR \abs{a'}^2 \frac{\th^2}{v}(\th-\tht)^2 d\x
+ C\e\l\int_\RR \abs{a'}\frac{1}{v}(\th-\tht)^2 d\x \\
&\le C\e^2\Dcal(U) + C\e\l(\Gcal_v(U)+\Gcal_\th(U)).
\end{align*}

\bpf{trans-bound} From \eqref{bo1p}, \eqref{barsmall} and \eqref{locE}, we have
\begin{equation}
\begin{aligned}\label{Btot}
\abs{\Bcal_1^+(U)} &\le \abs{\Bcal_1^+(U)-\Ical_1(\vbar, \thbar)} + \abs{\Ical_1(\vbar, \thbar)} 
\le C\frac{\e}{\l}(\Dcal_v(U) + \Dcal_\th(U)) + C\frac{\e^2}{\l}.
\end{aligned}
\end{equation}
This with \eqref{locE} implies
\begin{align*}
\Gcal_u^+(U) &\le C\int_\O a'(p-\pt)^2 d\x + C\int_\O a'(u-\ut)^2 d\x
\le C\abs{\Bcal_1^+(U)} + C\frac{\e^2}{\l} 
\le C\frac{\e}{\l}\Dcal(U) + C\frac{\e^2}{\l}.
\end{align*}
This completes the proof of Lemma \ref{lemma_para_ineq}.
\end{proof}

\subsubsection{Proof of \eqref{para-v}}
Recall \eqref{RHS-v}. The four terms in the definition of \(\Bcal^v\) will be denoted, in order, by \(I_1\) through \(I_4\) for ease of reference.

\bci{1}
By Young's inequality, we get
\[
I_1 \le C \Big( \l \int_\RR a v (\t_0+\th^2) \abs{\Big(\frac{1}{v}-\frac{1}{\vt}\Big)_\x}^2 d\x
+ \frac{1}{\l} \int_\RR \abs{(a\tht)'}^2 v(\t_0+\th^2)\Big(\frac{1}{v}-\frac{1}{\vt}\Big)^2 d\x\Big).
\]
From \eqref{der-scale}, we have \(|(a\tht)'| \le C \abs{a'} \le C \e\l\) and
\[
I_1 \le C \Big( \l \Dcal_v(U)
+ \e \int_\RR a' v \Big(\frac{1}{v}-\frac{1}{\vt}\Big)^2 d\x
+\frac{1}{\l} \int_\RR \abs{a'}^2 v \th^2 \Big(\frac{1}{v}-\frac{1}{\vt}\Big)^2 d\x \Big).
\]
Hence, by \eqref{ineq-p1} and \eqref{ineq-p2}, we have
\begin{align}
\begin{aligned} \label{vI1}
I_1 
&\le C\big( \l\Dcal(U)+\e(\Gcal_v(U)+\Gcal_\th(U)) \big).
\end{aligned}
\end{align}

\bci{2}
This can be split as follows: 
\begin{align*}
I_2 &\le C\int_\RR \big(v\th |\th-\tht|\big) \Big(\frac{1}{v}-\frac{1}{\vt}\Big)_\x \vt' d\x
+ C\int_\RR \big( v |\th-\tht| \big) \Big(\frac{1}{v}-\frac{1}{\vt}\Big)_\x \vt' d\x \\
&\qquad
+ C\int_\RR \abs{v-\vt}\Big(\frac{1}{v}-\frac{1}{\vt}\Big)_\x \vt' d\x
\eqqcolon I_{21}+I_{22}+I_{23}.
\end{align*}
By Young's inequality together with \eqref{der-scale}, we get
\begin{align*}
I_{21} &\le C\frac{\e}{\l}\Big(\int_\RR a v\th^2 \abs{\Big(\frac{1}{v}-\frac{1}{\vt}\Big)_\x}^2 d\x 
+ \int_\RR \abs{a'}^2 v(\th-\tht)^2 d\x \Big), \\
I_{22}+I_{23} &\le C\frac{\e}{\l}\Big( \int_\RR a v \abs{\Big(\frac{1}{v}-\frac{1}{\vt}\Big)_\x}^2  d\x 
+ \int_\RR \abs{a'}^2 v (\th-\tht)^2 d\x 
+ \int_\RR \abs{a'}^2 v\Big(\frac{1}{v}-\frac{1}{\vt}\Big)^2  d\x\Big).
\end{align*}
Then, using \eqref{ineq-p1} and \eqref{ineq-p3} with \eqref{der-a} and \eqref{der-scale}, we have
\begin{align}
\begin{aligned} \label{vI2}
I_2 
&\le C \left( \frac{\e}{\l}\Dcal(U) + \e^2 (\Gcal_v(U)+\Gcal_\th(U)) \right).
\end{aligned}
\end{align}

\bci{3}
This can be split into three parts and use \eqref{der-scale} and Young's inequality:
\begin{align*}
&I_3 \le C\int_\RR a' v\th |\th-\tht|\aabs{\frac{1}{v}-\frac{1}{\vt}}\vt' d\x
+ C\int_\RR a' v\tht |\th-\tht| \aabs{\frac{1}{v}-\frac{1}{\vt}}\vt' d\x \\
&\qquad
+ C\int_\RR a' (\t_0+\tht^2)\abs{v-\vt} \aabs{\frac{1}{v}-\frac{1}{\vt}}\vt' d\x \\
&\le C\frac{\e}{\l}\Big(\int_\RR \abs{a'}^2 v\th |\th-\tht| \aabs{\frac{1}{v}-\frac{1}{\vt}} d\x
+ \int_\RR \abs{a'}^2 v |\th-\tht| \aabs{\frac{1}{v}-\frac{1}{\vt}} d\x 
+ \int_\RR \abs{a'}^2 v\aabs{\frac{1}{v}-\frac{1}{\vt}}^2 d\x\Big) \\
&\le C\frac{\e}{\l} \Big( \int_\RR \abs{a'}^2 v(\th-\tht)^2 d\x
+ \int_\RR \abs{a'}^2 v \th^2 \Big(\frac{1}{v}-\frac{1}{\vt}\Big)^2 d\x 
+ \e\l\int_\RR \abs{a'} v \Big(\frac{1}{v}-\frac{1}{\vt}\Big)^2 d\x \Big).
\end{align*}
Hence, by \eqref{ineq-p1}, \eqref{ineq-p2}, and \eqref{ineq-p3}, we have
\begin{equation}\label{vI3}
I_3 \le C\left( \e\l \Dcal(U) + \e^2(\Gcal_v(U)+\Gcal_\th(U))\right).
\end{equation}

\bci{4}
Note by \eqref{tail} and \eqref{der-scale} that
\(
((\t_0+\tht^2)\frac{\vt'}{\vt})' \le C \abs{\vt''} + C\abs{\vt'}^2 \le C \frac{\e^2}{\l} a'.
\)
Then, using \eqref{ineq-p1}, we have
\begin{align}
\begin{aligned} \label{vI4}
I_4 &\le C \frac{\e^2}{\l} \int_\RR a' v\Big(\frac{1}{v}-\frac{1}{\vt}\Big)^2 d\x 
\le C \Big( \frac{\e^3}{\l^2} \Dcal_v(U) + \frac{\e^2}{\l} \Gcal_v(U) \Big).
\end{aligned}
\end{align}

Therefore, by summing up all from \eqref{vI1} to \eqref{vI4}, we obtain
\eqref{para-v}.
\qed

\subsubsection{Proof of \eqref{para-th}}
Recall \eqref{RHS-th}. The seven terms in the definition of \(\Bcal^\th\) will be denoted, in order, by \(I_1\) through \(I_7\) for ease of reference.
These can be controlled in a similar way, by using Young's inequality and Proposition \ref{lemma_para_ineq}. 

\bci{1}
Using \eqref{ineq-p4}, we get
\begin{align}
\begin{aligned} \label{thI1}
I_1 &\le C \Big(\l \int_\RR a \frac{1}{v}\abs{(\th-\tht)_\x}^2 d\x + \frac{1}{\l} \int_\RR \abs{a'}^2 \frac{\th^2}{v}(\th-\tht)^2 d\x \Big) \\
&\le C \big(\l \Dcal(U) + \e(\Gcal_v(U)+\Gcal_\th(U)) \big).
\end{aligned}
\end{align}

\bci{2}
Using \eqref{der-scale} and \eqref{ineq-p5}, we get
\begin{align}
\begin{aligned} \label{thI2}
I_2 &\le C \Big( \frac{\e}{\l} \Dcal_\th(U) + \e^2 \int_\RR a' \frac{(\th-\tht)^2}{v} d\x \Big) 
\le C \left( \frac{\e}{\l} \Dcal(U) +\e^2 \Gcal_\th(U) +\e^4 \Gcal_v(U) \right).
\end{aligned}
\end{align}

\bci{3}
As above, using \eqref{der-scale} and \eqref{ineq-p5}, we have
\begin{align}
\begin{aligned} \label{thI3}
I_3 &\le C \Big( \frac{\e}{\l} \Dcal_v(U) + \e^2 \int_\RR a' \frac{(\th-\tht)^2}{v} d\x \Big) \le C \Big( \frac{\e}{\l} \Dcal(U) + \e^2 \Gcal_\th(U) +\e^4 \Gcal_v(U) \Big).
\end{aligned}
\end{align}

\bci{4}
Using \eqref{der-scale} and \eqref{ineq-p5}, we have
\begin{align}
\begin{aligned} \label{thI4}
    I_4 \le C \frac{\e^2}{\l} \int_\RR a' \frac{(\th-\tht)^2}{v} d\x
    \le C \frac{\e^2}{\l}\left(\frac{\e}{\l} \Dcal(U) + \Gcal_\th(U) + \e^2 \Gcal_v(U) \right).
\end{aligned}
\end{align}

\bci{5}
From \eqref{tail}, \eqref{der-scale}, \eqref{ineq-p1}, and \eqref{ineq-p5}
\begin{align}
\begin{aligned} \label{thI5}
I_5 &\le C \frac{\e^2}{\l}\Big( \int_\RR a' \frac{(\th-\tht)^2}{v} d\x
+ \int_\RR a' v \Big(\frac{1}{v}-\frac{1}{\vt}\Big)^2 d\x \Big) \\
&\le C \frac{\e^2}{\l}\Big(\frac{\e}{\l}\Dcal(U) 
+ (\Gcal_v(U)+\Gcal_\th(U)) \Big).
\end{aligned}
\end{align}

\bci{6}
Note that \((\tht')^2 \le C \frac{\e^3}{\l} a'\) holds because of \eqref{der-scale}. 
Hence, \eqref{thI5} gives
\begin{align}
\begin{aligned} \label{thI6}
I_6 &\le C\e I_5 \le C \frac{\e^3}{\l} \left( \frac{\e}{\l}\Dcal(U) 
+ (\Gcal_v(U)+\Gcal_\th(U)) \right).
\end{aligned}
\end{align}

\bci{7}
From \eqref{der-scale} and \eqref{ineq-p6}, we have
\begin{align}
\begin{aligned} \label{thI7}
I_7 \le C \frac{\e^3}{\l} \int_\RR a' (\th-\tht)^2 d\x 
\le C \frac{\e^3}{\l} \left( \frac{\e}{\l} \Dcal_\th(U) + \Gcal_\th(U) \right).
\end{aligned}
\end{align}
Therefore, by summing up all from \eqref{thI1} to \eqref{thI7}, we obtain \eqref{para-th}.
\qed

\subsubsection{Proof of \eqref{para-u}}
Recall \eqref{RHS-u}. The eight terms in the definition of \(\Bcal^u\) will be denoted, in order, by \(I_1\) through \(I_8\) for ease of reference.

\bci{1}
Using \eqref{ineq-p7}, we get
\begin{align}
\begin{aligned} \label{uI1}
I_1 &\le C \Big( \l \int_\RR a\frac{\th}{v} \abs{(u-\ut)_\x}^2 d\x + \frac{1}{\l} \int_\RR \abs{a'}^2 \frac{\th^3}{v}(u-\ut)^2 d\x \Big) \\
&\le C \left( \left(\l+\frac{\e}{\l}\right)\Dcal(U) + \e\Gcal(U) \right).
\end{aligned}
\end{align}

\bci{2}
Using \eqref{ineq-p9}, we get
\begin{align}
\begin{aligned} \label{uI2}
I_2 &\le C \Big( \frac{\e}{\l} \Dcal_u(U) + \frac{\e}{\l} \int_\RR \abs{a'}^2 \frac{\th}{v}(\th-\tht)^2 \Big) \le C \left( \frac{\e}{\l} \Dcal(U) 
+ \e^2 (\Gcal_v(U) + \Gcal_\th(U)) \right).
\end{aligned}
\end{align}

\bci{3}
Using \eqref{ineq-p8}, we get
\begin{align}
\begin{aligned} \label{uI3}
I_3 &\le C \Big( \frac{\e}{\l} \Dcal_\th(U) + \e^2 \int_\RR a' \frac{\th^2}{v}(u-\ut)^2 d\x \Big)
\le C \Big(\frac{\e}{\l} \Dcal(U) + \e^2\Gcal(U)\Big).
\end{aligned}
\end{align}

\bci{4}
Using \eqref{ineq-p8}, we get
\begin{align}
\begin{aligned} \label{uI4}
I_4 &\le C \Big( \frac{\e}{\l} \Dcal_v(U) + \e^2 \int_\RR a' \frac{\th^2}{v}(u-\ut)^2 d\x \Big)
\le C \Big(\frac{\e}{\l} \Dcal(U) + \e^2\Gcal(U)\Big).
\end{aligned}
\end{align}

\bci{5}
This can be split as follows:
\begin{align*}
I_5
&= \hspace{-1mm}\underbrace{\int_\RR a (u-\ut)\frac{\th}{v}(\th-\tht) \ut'' d\x}_{\eqqcolon I_{51}}
+ \hspace{-1mm}\underbrace{\int_\RR a (u-\ut)\tht \th \Big(\frac{1}{v}-\frac{1}{\vt}\Big) \ut'' d\x}_{\eqqcolon I_{52}} 
+ \hspace{-1mm}\underbrace{\int_\RR a (u-\ut)\frac{\tht}{\vt} (\th-\tht) \ut'' d\x}_{\eqqcolon I_{53}}.
\end{align*}
We use \eqref{tail}, \eqref{der-scale} and Young's inequality, and then \eqref{ineq-p5} and \eqref{ineq-p8} imply the following:
\begin{align*}
I_{51} 
\le C \frac{\e^2}{\l} \Big( \int_\RR a' \frac{(\th-\tht)^2}{v} d\x + \int_\RR a' \frac{\th^2}{v}(u-\ut)^2 d\x \Big)
\le C \frac{\e^2}{\l} \Big(\frac{\e}{\l}\Dcal(U)+\Gcal(U)\Big).
\end{align*}
Similarly, but using \eqref{ineq-p1} and \eqref{ineq-p8}, \eqref{ineq-p6} and \eqref{ineq-aux-u}, respectively, \(I_{52}\) and \(I_{53}\) can be controlled by the same bound. 
Thus, summing up, we obtain
\begin{align} \label{uI5}
I_5 \le C \frac{\e^2}{\l} \left(\frac{\e}{\l}\Dcal(U)+\Gcal(U)\right).
\end{align}

\bci{6}
This can be split as follows:
\begin{equation*}
I_6 = \underbrace{2\int_\RR a (u-\ut)\th \Big(\frac{1}{v}-\frac{1}{\vt}\Big) \tht' \ut' d\x}_{\eqqcolon I_{61}}
+ \underbrace{2\int_\RR a (u-\ut)\frac{1}{\vt}(\th-\tht) \tht' \ut' d\x}_{\eqqcolon I_{62}}.
\end{equation*}
As in the case above, using \eqref{der-scale} and Young's inequality, we apply \eqref{ineq-p1} and \eqref{ineq-p8} to get
\begin{align*}
I_{61} &\le C \frac{\e^3}{\l} \Big( \int_\RR a' v\Big(\frac{1}{v}-\frac{1}{\vt}\Big)^2 d\x + \int_\RR a' \frac{\th^2}{v}(u-\ut)^2 d\x \Big)
\le C \frac{\e^3}{\l}\Big( \frac{\e}{\l}\Dcal(U) +\Gcal(U)\Big).
\end{align*}
From \eqref{der-scale} together with the bound for \(I_{53}\), \(I_{62}\) can be controlled by the same bound as \(I_{61}\). 
Thus, summing up, we obtain
\begin{align} \label{uI6}
I_6 \le C\frac{\e^3}{\l} \left( \frac{\e}{\l}\Dcal(U) + \Gcal(U) \right).
\end{align}

\bci{7}
This can be rewritten as follows:
\[
I_7= \int_\RR a \frac{(\th-\tht)^2}{v} (\ut')^2 d\x + \int_\RR a (\th-\tht) \Big(\frac{1}{v}-\frac{1}{\vt}\Big) \tht (\ut')^2 d\x.
\]
Using \eqref{ineq-p1}, \eqref{ineq-p5} and \eqref{der-scale}, we have
\begin{align}
\begin{aligned} \label{uI7}
I_7 &\le C\frac{\e^3}{\l} \Big(\int_\RR a' \frac{(\th-\tht)^2}{v} d\x + \int_\RR a' v \Big(\frac{1}{v}-\frac{1}{\vt}\Big)^2 d\x \Big)\\
&\le C \frac{\e^3}{\l} \left( \frac{\e}{\l}\Dcal(U) 
+ (\Gcal_v(U)+\Gcal_\th(U)) \right).
\end{aligned}
\end{align}

\bci{8}
This can be split as follows:
\begin{align*}
I_8
&= \underbrace{\int_\RR a (u-\ut)\th(\th-\tht)\Big(\frac{1}{\vt}\Big)' \ut' d\x}_{\eqqcolon I_{81}}
+ \underbrace{\int_\RR a (u-\ut)\tht (\th-\tht)\Big(\frac{1}{\vt}\Big)' \ut' d\x}_{\eqqcolon I_{82}}.
\end{align*}
For \(I_{81}\), from Young's inequality and \eqref{der-scale} with \eqref{ineq-p3} and \eqref{ineq-p8}, we obtain 
\begin{align*}
I_{81} &\le C \frac{\e^2}{\l^2} \Big( \int_\RR \abs{a'}^2 v(\th-\tht)^2 d\x + \int_\RR \abs{a'}^2 \frac{\th^2}{v}(u-\ut)^2 d\x \Big)
\le C \frac{\e^2}{\l^2}\left(\e^2\Dcal(U)+\e\l\Gcal(U)\right).
\end{align*}
Similarly for \(I_{82}\), using \eqref{ineq-p6} and \eqref{ineq-aux-u}, we can control this term with the same bound.
Thus, summing up, we obtain
\begin{align} \label{uI8}
I_8 \le C \frac{\e^2}{\l^2}\left(\e^2\Dcal(U)+\e\l\Gcal(U)\right).
\end{align}

Therefore, summing up all from \eqref{uI1} to \eqref{uI8}, we establish \eqref{para-u} as we desired.
\qed

\subsection{Control of $Y(U)$ outside truncation}\label{section-shift}
Recall the functional $Y$ in \eqref{ybg-first} first. Then, we decompose $Y$ into four parts $Y_g$,  $Y_b$, $Y_l$ and $Y_s$ as follows: \\
\[
Y= Y_g +Y_b +Y_l + Y_s,
\]
where 
\begin{align*}
&Y_g(U) \coloneqq -\int_\O a'\Big(R\tht\Phif{v}{\vt}+\frac{R}{\g-1}\tht\Phif{\th}{\tht}+\frac{(p-\pt)^2}{2\s_\e^2}\Big) d\x \\
&- \int_\O a\Big(R\Phif{v}{\vt}+\frac{R}{\g-1}\Phif{\th}{\tht}\Big)\tht' d\x 
+\int_\O a\Big(R\tht\frac{v-\vt}{\vt^2}\vt'+\frac{R}{\g-1}\frac{\th-\tht}{\tht}\tht'+\frac{p-\pt}{\s_\e}\ut'\Big) d\x, \\
&Y_b(U) \coloneqq -\frac{1}{2}\int_\O a'\Big(u-\ut-\frac{p-\pt}{\s_\e}\Big)^2 d\x -\frac{1}{\s_\e}\int_\O a'(p-\pt)\Big(u-\ut-\frac{p-\pt}{\s_\e}\Big)d\x, \\
&Y_l(U) \coloneqq \int_\O a\Big(u-\ut-\frac{p-\pt}{\s_\e}\Big)\ut'd\x, \\
&Y_s(U) \coloneqq -\int_{\O^c} a'\Big(R\tht\Phif{v}{\vt}+\frac{R}{\g-1}\tht\Phif{\th}{\tht}+\frac{(u-\ut)^2}{2}\Big) d\x \\
&
- \int_{\O^c} a\Big(R\Phif{v}{\vt}
+ \frac{R}{\g-1}\Phif{\th}{\tht}\Big) \tht' d\x
+ \int_{\O^c} a\Big(R\tht\frac{v-\vt}{\vt^2}\vt'+\frac{R}{\g-1}\frac{\th-\tht}{\tht}\tht'+(u-\ut)\ut'\Big) d\x.
\end{align*}
Notice that $Y_g$ consists of the terms related to \(1/v-1/\vt\) and \(\th-\tht\),  while $Y_b$ and $Y_l$ consist of terms related to $u-\ut$. 
In addition, while $Y_b$ is quadratic, $Y_l$ is linear. 
Here, we need to show that $|Y_g(U)-\Ical_g^Y(\vbar,\thbar)|^2$ (for $\Ical_g^Y$ in \eqref{note-in}), $|Y_b(U)|^2$, $|Y_l(U)|^2$ and $|Y_s(U)|^2$ are negligible by the good terms. 
\begin{proposition}\label{prop_shift_out}
For any \(C^*>0\), there exist constants \(\e_0, \d_0, C>0\) (in particular, \(C\) depends on the constant \(C^*\)) such that for any $\e<\e_0$ and $\d_0^{-1}\e<\l<\d_0<1/2$, the following statements hold true.
For any $U$ such that $\abs{Y(U)}\le \e^2$ and $\Dcal_v(U)+\Dcal_\th(U)\le 10 C^*\frac{\e^2}{\l}$,
\begin{align}
&\abs{Y_l(U)}^2 \le C\frac{\e^2}{\l}\Gcal_u^+(U) \label{Y-l} \\
&\abs{Y_g(U)-\Ical_g^Y(\vbar,\thbar)} + \abs{Y_b(U)} + \abs{Y_s(U)} \le C\frac{\e^2}{\l} 
\label{Y-a1} \\
&\begin{aligned}\label{Y-a2}
&\abs{Y_g(U)-\Ical_g^Y(\vbar,\thbar)}+\abs{Y_b(U)}+\abs{Y_s(U)} \\
&\le C\Big(\sqrt{\frac{\l}{\e}}\Gcal_u^+(U) + \frac{\e}{\l}(\Dcal_v(U) + \Dcal_\th(U))
+ \Gcal_u^-(U) \\
&\qquad \qquad
+ (\Gcal_v(U)-\Gcal_v(\Ubar)+\Gcal_\th(U)-\Gcal_\th(\Ubar)) 
+ \sqrt[4]{\frac{\e}{\l}}(\Gcal_v(\Ubar) + \Gcal_\th(\Ubar))\Big)
\end{aligned}\\ 
&\begin{aligned}\label{Y-conc}
&\abs{Y_g(U)-\Ical_g^Y(\vbar,\thbar)}^2 +\abs{Y_b(U)}^2+\abs{Y_l(U)}^2+\abs{Y_s(U)}^2 \\
&\le C\frac{\e^2}{\l}\Big(\sqrt{\frac{\l}{\e}}\Gcal_u^+(U) + \frac{\e}{\l}(\Dcal_v(U) + \Dcal_\th(U))
+ \Gcal_u^-(U) \\
&\qquad \qquad
+ (\Gcal_v(U)-\Gcal_v(\Ubar)+\Gcal_\th(U)-\Gcal_\th(\Ubar)) 
+ \sqrt{\frac{\e}{\l}}(\Gcal_v(\Ubar) + \Gcal_\th(\Ubar))\Big)
\end{aligned}
\end{align}
\end{proposition}

\begin{proof}
\bpf{Y-l} Using H\"older's inequality with \eqref{der-scale}, we obtain the following: 
\begin{align*}
\abs{Y_l(U)}^2 &\le
\int_\O a\ut' d\x \cdot \int_\O a\Big(u-\ut-\frac{p-\pt}{\s_\e}\Big)^2 \ut' d\x 
\le C\frac{\e^2}{\l}\Gcal_u^+(U).
\end{align*}

\bpf{Y-a1} We control each term separately.
For \(\abs{Y_g(U)-\Ical_g^Y(\vbar,\thbar)}\), we have
\begin{equation}
\begin{aligned}\label{ineq-Yg}
\abs{Y_g(U)-\Ical_g^Y(\vbar,\thbar)} &\le C(\Gcal_v(U)-\Gcal_v(\Ubar)+\Gcal_\th(U)-\Gcal_\th(\Ubar)) + C\abs{\Bcal_1^+(U)-\Ical_1(\vbar,\thbar)} \\
&\quad
+ C\frac{\e}{\l}\int_\O a'\Big(\abs{v-\vbar}+|\th-\thbar|+\abs{\frac{\th}{v}-\frac{\thbar}{\vbar}}\Big) d\x 
+ C\d_3\int_{\O^c} a'd\x.
\end{aligned}
\end{equation}
Using \eqref{rel_Phi1}, \eqref{l2-v}, \eqref{l2-th}, \eqref{bo1p}, \eqref{ineq-p-linear} and \eqref{mainbad-one} together with the assumption \(\Dcal_v(U)+\Dcal_\th(U) \le 10C^*\frac{\e^2}{\l}\), we obtain that 
\[
\abs{Y_g(U)-\Ical_g^Y(\vbar,\thbar)} \le C\Big(\frac{\e}{\l}(\Dcal_v(U) + \Dcal_\th(U)) + \frac{\e^2}{\l}\Big) \le C\frac{\e^2}{\l}.
\]

For \(Y_b(U)\), we apply Young's inequality together with \eqref{trans-bound}, \eqref{Btot} and the assumption \(\Dcal_v(U)+\Dcal_\th(U) \le 10C^*\frac{\e^2}{\l}\) to obtain that
\[
\abs{Y_b(U)} \le C\Gcal_u^+(U) + C\abs{\Bcal_1^+(U)} \le C\frac{\e^2}{\l}.
\]

For \(Y_s(U)\), we decompose the terms as follows: 
\begin{equation}    
\begin{aligned}\label{Ys_decompose}
&Y_s(U) = -\int_{\O^c} a'\tht\et(U|\Ut) d\x
-\int_{\O^c}a\Big(R\Phif{v}{\vt}+\frac{R}{\g-1}\Phif{\th}{\tht}\Big)\tht' d\x \\
&+\int_{\O^c} \Big(R\tht\frac{v-\vt}{\vt^2}\vt' + \frac{R}{\g-1}\frac{\th-\tht}{\tht}\tht' + (u-\ut)\ut'\Big) d\x
\eqqcolon Y_s^1(U)+Y_s^2(U)+Y_s^3(U).
\end{aligned}
\end{equation}
For \(Y_s^1(U)\) and \(Y_s^2(U)\), we have \(\abs{Y_s^1(U)}+\abs{Y_s^2(U)} \le C\e^2/\l\) from \eqref{locE}. \\
For \(Y_s^3(U)\), we have
\begin{equation}\label{Ys3}
\abs{Y_s^3(U)} \le C\frac{\e}{\l}\int_{\O^c} a'(\abs{v-\vbar} + |\th-\thbar| + (u-\ut)^2 + 1) d\x, 
\end{equation}
which gives \(\abs{Y_s^3(U)}\le C\e^2/\l\) from \eqref{rel_Phi1}, \eqref{locE}, \eqref{l2-v}, \eqref{l2-th} and \eqref{mainbad-one} under the assumption \(\Dcal_v(U)+\Dcal_\th(U) \le 10C^*\frac{\e^2}{\l}\). Summing up, we obtain the desired bound. 

\bpf{Y-a2} Again we separately bound the three terms. 
For \(\abs{Y_g(U)-\Ical_g^Y(\vbar,\thbar)}\), we reuse \eqref{ineq-Yg}. From \eqref{rel_Phi1}, \eqref{bo1p}, \eqref{ineq-p-linear} and \eqref{mainbad-one}, we obtain that 
\begin{align*}
&\abs{Y_g(U)-\Ical_g^Y(\vbar,\thbar)} \\
&\le C\left(\frac{\e}{\l}(\Dcal_v(U) + \Dcal_\th(U)) + \e^2\Gcal_v(\Ubar) + (\Gcal_v(U)-\Gcal_v(\Ubar)+\Gcal_\th(U)-\Gcal_\th(\Ubar))\right).
\end{align*}

For \(Y_b(U)\), we use Young's inequality differently to obtain that 
\[
\abs{Y_b(U)} \le C\Big(\sqrt{\frac{\l}{\e}}\Gcal_u^+(U) + \sqrt{\frac{\e}{\l}}\abs{\Bcal_1^+(U)}\Big).
\]
Using \eqref{bo1p} and \eqref{rel_Phi_v}-\eqref{rel_Phi_th}, the second term on the right-hand side can be bounded:
\begin{align*}
\abs{\Bcal_1^+(U)} &\le 
\abs{\Bcal_1^+(U)-\Ical_1(\vbar, \thbar)} + \abs{\Ical_1(\vbar, \thbar)} \\
&\le C\Big(\frac{\e}{\l}(\Dcal_v(U) + \Dcal_\th(U))
+ \Gcal_v(\Ubar) + \Gcal_\th(\Ubar) + \e^2 (\Gcal_v(U)-\Gcal_v(\Ubar)) \Big).
\end{align*} 
Therefore, we obtain the desired bound
\[
\abs{Y_b(U)} \le C\Big(\sqrt{\frac{\l}{\e}}\Gcal_u^+(U) + \frac{\e}{\l}(\Dcal_v(U) + \Dcal_\th(U))
+ \sqrt{\frac{\e}{\l}}(\Gcal_v(\Ubar) + \Gcal_\th(\Ubar))
+ \e^2(\Gcal_v(U)-\Gcal_v(\Ubar))\Big).
\]

For \(Y_s(U)\), we consider the decomposition \eqref{Ys_decompose}. 
For \(Y_s^1(U)\), we use \eqref{mainbad-one} to get 
\begin{align*}
\abs{Y_s^1(U)} &\le
C\int_{\O^c} a'\Big(\Phif{v}{\vt}-\Phif{\vbar}{\vt}\Big) d\x
+ C\int_{\O^c} a'\Big(\Phif{\th}{\tht}-\Phif{\thbar}{\tht}\Big) d\x \\
&\qquad
+ C\int_{\O^c} a'(u-\ut)^2 d\x 
+ C\d_3\int_{\O^c} a' d\x \\
&\le C\left(\frac{\e}{\l}\Dcal_v(U) + \e^2\Gcal_v(\Ubar) + \Gcal_u^-(U)
+ (\Gcal_v(U)-\Gcal_v(\Ubar) + \Gcal_\th(U)-\Gcal_\th(\Ubar))\right).
\end{align*}
Similarly, for \(Y_s^2(U)\), we use \eqref{mainbad-one} again to get
\begin{align*}
\abs{Y_s^2(U)} &\le C\frac{\e}{\l}\Big(\int_{\O^c}a'\Big(\Phif{v}{\vt}-\Phif{\vbar}{\vt}\Big) d\x 
+ \int_{\O^c}a'\Big(\Phif{\th}{\tht}-\Phif{\thbar}{\tht}\Big) d\x + \d_3\int_{\O^c} a' d\x \Big) \\
&\le C\frac{\e}{\l}\Big(\frac{\e}{\l}\Dcal_v(U) + \e^2\Gcal_v(\Ubar)
+ (\Gcal_v(U)-\Gcal_v(\Ubar) + \Gcal_\th(U)-\Gcal_\th(\Ubar))\Big).
\end{align*}
For \(Y_s^3(U)\), we use \eqref{Ys3} with \eqref{rel_Phi1} and \eqref{mainbad-one} to obtain that 
\[
\abs{Y_s^3(U)} \le C \frac{\e}{\l} \left(\frac{\e}{\l}\Dcal_v(U) + \e^2\Gcal_v(\Ubar) + \Gcal_u^-(U) + (\Gcal_v(U)-\Gcal_v(\Ubar) + \Gcal_\th(U)-\Gcal_\th(\Ubar))\right).
\]
Summing up, we obtain 
\[
\abs{Y_s(U)} \le C\left(\frac{\e}{\l}\Dcal_v(U) + \e^2\Gcal_v(\Ubar)
+ \Gcal_u^-(U) + (\Gcal_v(U)-\Gcal_v(\Ubar) + \Gcal_\th(U)-\Gcal_\th(\Ubar))\right).
\]
These three inequalities give \eqref{Y-a2} as we desired.

\bpf{Y-conc}
From \eqref{Y-l}, \eqref{Y-a1}, and \eqref{Y-a2}, we obtain the desired bound. 
\end{proof}

\subsection{Proof of Proposition \ref{prop:main}}\label{section-main}
We now prove the main Proposition  \ref{prop:main}. We split the proof into two steps, depending on the strength of the dissipation term $\Dcal(U)$.

\step{1}
We first consider the case of \(\Dcal(U)\ge 10C^* \frac{\e^2}{\l}\), where the constant \(C^*\) is chosen to be the greater one between constants \(C^*\) given in Proposition \ref{prop_hyperbolic_out} and \ref{prop_parabolic_out}.
Then, using \eqref{bo} and taking \(\d_0\) small enough, we have 
\begin{align*}
\Rcal(U)
&\le
2\abs{\Bcal_{\d_3}(U)} + \abs{\Bcal_1^+(U)} + 2\abs{\Jcal^{para}(U)} - (1-\d_0)(\Dcal_v(U)+ \Dcal_\th(U) + \Dcal_u(U)) \\
&\le 5C^*\frac{\e^2}{\l} - (1-C\d_0-C\sqrt[4]{\d_0})(\Dcal_v(U)+ \Dcal_\th(U) + \Dcal_u(U)) \\
&\le 5C^*\frac{\e^2}{\l} - \frac{1}{2}(\Dcal_v(U)+ \Dcal_\th(U) + \Dcal_u(U)) \le 0,
\end{align*}
which gives the desired result. 

\step{2}
We now assume the other alternative, i.e., \(\Dcal(U)\le 10C^* \frac{\e^2}{\l}\). 
First of all, we have \eqref{lbis}, and for the small constant \(\d_3\) of Proposition \ref{prop:main3} associated to the constant \(C_2\) in \eqref{lbis}, we defined the truncation of size \(\d_3\).  
Using 
\[
\Ical_g^Y(\vbar,\thbar) = Y(U) - (Y_g(U)-\Ical_g^Y(\vbar,\thbar)) - Y_b(U) - Y_l(U) - Y_s(U),
\]
we have 
\[
-5\abs{Y(U)}^2 \le -\abs{\Ical_g^Y(\vbar,\thbar)}^2 + 5\abs{Y_g(U)-\Ical_g^Y(\vbar,\thbar)}^2 + 5\abs{Y_b(U)}^2 + 5\abs{Y_l(U)}^2 + 5\abs{Y_s(U)}^2.
\]
Now, taking \(\d_0\) small enough such that \(\d_0 \le \d_3^5\), we use \eqref{keyD} to find that for \(\e<\e_0(\le \d_3)\) and \(\e/\l<\d_0\), 
\begin{align*}
&\Rcal(U) \le -\frac{5}{\e\d_3}Y(U)^2 + \Bcal_{\d_3}(U) + \d_0\frac{\e}{\l}\abs{\Bcal_{\d_3}(U)} + \d_0\frac{\e}{\l}\Bcal_1^+(U) + \Jcal^{para}(U) + \d_0\abs{\Jcal^{para}(U)} \\
&\qquad\qquad -\Gcal_u^-(U)-\frac{1}{2}\Gcal_u^+(U) -\Big(1-\d_0\frac{\e}{\l}\Big)(\Gcal_v(U)+\Gcal_\th(U)) - (1-\d_0)\Dcal(U) \\
&\le -\frac{1}{\e\d_3}\abs{\Ical_g^Y(\vbar,\thbar)}^2 + \Ical_1(\vbar, \thbar) + \Bcal_2(\Ubar) + \Bcal_3(\Ubar) + \Bcal_4(\Ubar) + \d_0\frac{\e}{\l}\abs{\Ical_1(\vbar, \thbar)} \\
&\qquad
+ \d_0\frac{\e}{\l}\left(\abs{\Bcal_2(\Ubar)}+\abs{\Bcal_3(\Ubar)} + \abs{\Bcal_4(\Ubar)}\right) 
-\Big(1-\d_3\frac{\e}{\l}\Big)\left(\Gcal_v(\Ubar) + \Gcal_\th(\Ubar)\right) - (1-\d_3)\Dcal_v^1(\Ubar) \\
&\qquad
\underbrace{
\begin{aligned}
&+\left(1+\d_0\frac{\e}{\l}\right)\big(
\abs{\Bcal_1^+(U)-\Ical_1(\vbar, \thbar)} + \abs{\Bcal_1^-(U)} + \abs{\Bcal_2(U)-\Bcal_2(\Ubar)} \\
&\hspace{50mm}
+ \abs{\Bcal_3(U)-\Bcal_3(\Ubar)} + \abs{\Bcal_4(U)-\Bcal_4(\Ubar)}
\big)
\end{aligned}
}_{\eqqcolon J_1} \\
&\qquad
+\underbrace{\left(1+\d_0\right)\abs{\Jcal^{para}(U)}}_{\eqqcolon J_2} 
+\underbrace{\frac{5}{\e\d_3}\left(\abs{Y_g(U)-\Ical_g^Y(\vbar,\thbar)}^2 + \abs{Y_b(U)}^2 + \abs{Y_l(U)}^2 + \abs{Y_s(U)}^2\right)}_{\eqqcolon J_3} \\
&\qquad
-\Gcal_u^-(U)-\frac{1}{2}\Gcal_u^+(U)-\frac{1}{2}(\Gcal_v(U)-\Gcal_v(\Ubar)+\Gcal_\th(U)-\Gcal_\th(\Ubar)) \\
&\qquad
- (\d_3-\d_0)\frac{\e}{\l}(\Gcal_v(\Ubar)+\Gcal_\th(\Ubar))
- (\d_3-\d_0)\Dcal_v^1(U) - (1-\d_0)(\Dcal_v^2(U)+\Dcal_\th(U)+\Dcal_u(U)).
\end{align*}
We now claim that \(J_1, J_2, J_3\) are controlled by the remaining good terms above. 
Indeed, it follows from \eqref{bo1p}-\eqref{bo4} that for any \(\l, \e/\l<\d_0(\le \d_3^5)\),
\begin{align*}
J_1 
&\le \frac{\d_3}{4}(\Dcal_v(U)+\Dcal_\th(U)) + \d_3\Gcal_u^-(U) \\
&\qquad
+ \frac{\d_3}{4}(\Gcal_v(U)-\Gcal_v(\Ubar)+\Gcal_\th(U)-\Gcal_\th(\Ubar))
+ \frac{\d_3}{4}\frac{\e}{\l}(\Gcal_v(\Ubar)+\Gcal_\th(\Ubar))
\end{align*}
and from \eqref{para-total},
\begin{align*}
J_2 
&\le \frac{\d_3}{4}(\Dcal_v(U)+\Dcal_\th(U)+\Dcal_u(U)) + \d_0(\Gcal_u^-(U)+\Gcal_u^+(U)) \\
&\qquad
+ \frac{\d_3}{4}(\Gcal_v(U)-\Gcal_v(\Ubar)+\Gcal_\th(U)-\Gcal_\th(\Ubar))
+ \frac{\d_3}{4}\frac{\e}{\l}(\Gcal_v(\Ubar)+\Gcal_\th(\Ubar)).
\end{align*}
For \(J_3\), from \eqref{Y-conc}, we have
\begin{align*}
J_3
&\le \frac{\d_3}{4}\Big(\Dcal_v(U)+\Dcal_\th(U) + \Gcal_u^-(U)+\Gcal_u^+(U) \\
&\hspace{30mm}
+(\Gcal_v(U)-\Gcal_v(\Ubar)+\Gcal_\th(U)-\Gcal_\th(\Ubar)) 
+\frac{\e}{\l}(\Gcal_v(\Ubar)+\Gcal_\th(\Ubar))\Big).
\end{align*}
Therefore, we have 
\begin{align*}
&\Rcal(U) \le -\frac{1}{\e\d_3}\abs{\Ical_g^Y(\vbar,\thbar)}^2 
+ \left(\Ical_1(\vbar, \thbar) + \Bcal_2(\Ubar) + \Bcal_3(\Ubar) + \Bcal_4(\Ubar)\right)
+ \d_3\frac{\e}{\l} \abs{\Ical_1(\vbar, \thbar)} \\
&\qquad + \d_3\frac{\e}{\l} \left(\abs{\Bcal_2(\Ubar)}+\abs{\Bcal_3(\Ubar)}+ \abs{\Bcal_4(\Ubar)}\right)
-\Big(1-\d_3\frac{\e}{\l}\Big)\left(\Gcal_v(\Ubar) + \Gcal_\th(\Ubar)\right) - (1-\d_3)\Dcal_v^1(\Ubar).
\end{align*}
Since the above quantities \(\Ical_g^Y(\vbar,\thbar), \Bcal_2(\Ubar), \Bcal_3(\Ubar), \Bcal_4(\Ubar)\), \(\Gcal_v(\Ubar)\), \(\Gcal_\th(U)\), and \(\Dcal_v(\Ubar)\) depend only on \(\vbar\) and \(\thbar\), and 
\(
\Bcal_2(\Ubar) = \Ical_2(\vbar, \thbar), 
\Bcal_3(\Ubar) = \Ical_3(\vbar, \thbar), 
\Bcal_4(\Ubar) = \Ical_4(\vbar, \thbar),
\)
it holds by Proposition \ref{prop:main3} that \(\Rcal(U)\le 0\). 
This completes the proof of Proposition \ref{prop:main}. \qed


\section{Proof of Theorem \ref{thm_inviscid}}
\setcounter{equation}{0}

For the existence of well-prepared initial data, and the asymptotic analysis for \eqref{wconv} and \eqref{uni-est}, we follow the same argument as in \cite[Section 5]{KV-Inven}. So, we only present the proof of \eqref{X-control} on the control of shift,  that is subtle as mentioned in Section \ref{sec:idea}.

\subsection{Uniform estimates in \(\nu\)}
First of all, we obtain the uniform-in-$\nu$ estimates from the contraction estimate of Theorem \ref{thm_main} as follows.
Let \(\{ (v^\nu,u^\nu,\th^\nu) \}_{\nu>0}\) be a sequence of solutions on \((0,T)\) to \eqref{inveq} with the initial datum \((v_0^\nu,u_0^\nu,\th_0^\nu)\).
Moreover, throughout this section, let \(C\) be a positive constant that may vary from line to line but remains independent of \(\n\). (This may depend on \(\e,\l\) and \(T\) now.)

We apply Theorem \ref{thm_main} to the following functions:
\begin{align*}
    v(t,x)&=\vn(\nu t,\nu x), &u(t,x)&=\un(\nu t,\nu x), &\th(t,x)&=\thn(\nu t,\nu x), \\
    \vtil(x)&=\vtn(\nu x), &\util(x)&=\utn(\nu x), &\thtil(x)&=\thtn(\nu x),
\end{align*}
perform the change of variables \(t \mapsto t/\nu \) and \(x \mapsto x/\nu \), and make use of \eqref{ini_conv} so that we obtain for any \(\d\in(0,1)\), there exists \(\n_*>0\) such that for any \(\n<\n_*\) and any \(t\in[0,T]\), the following holds: (recall by \eqref{basic_ini} that \(\Ecal_0 \coloneqq \int_\RR \eta \left( (v^0,u^0,\th^0) | (\vbar,\ubar,\thbar) \right) dx\))
\begin{align}
\begin{aligned} \label{ineq-m}
    &\int_\RR \eta\left((\vn,\un,\thn)(t,x)|(\vtn,\utn,\thtn)(x-X_\nu(t))\right)dx \\
    &\qquad +\int_0^{T} \int_\RR \abs{(\vtn)'(x-X_\nu(t))} \Phi(\vn(t,x)/\vtn(x-X_\nu(t))) dxdt \\
    &\qquad +\nu\int_0^{T} \int_\RR \left(\vn(c+(\thn)^2)\right)(t,x)\abs{\rd_x\Big(\frac{1}{\vn(t,x)}-\frac{1}{\vtn(x-X_\nu(t))}\Big)}^2 dxdt \\
    &\qquad +\nu\int_0^{T} \int_\RR \frac{1}{\vn(t,x)}\abs{\rd_x\left( \thn(t,x)-\thtn(x-X_\nu(t)) \right)}^2 dxdt\\
    &\le C \Ecal_0 + \d.
\end{aligned}
\end{align}

\subsection{Proof of \eqref{X-control}: Control of the shift}
First, following the same argument as in \cite[Lemma 5.1]{KV-Inven}, we have  the $L^1$ convergence of \(\{X_\nu\}_{\nu>0}\) as follows.
\begin{lemma}\label{lem-X}
There exists \(X_\infty \in BV((0,T))\) such that
\begin{equation} \label{X-con}
\begin{aligned}
X_\nu \to X_\infty \quad \text{in } L^1(0,T), \quad \text{up to subsequence as } \nu \to 0.
\end{aligned}
\end{equation}
\end{lemma}
\begin{proof}
Since \(X_\nu(t)= \nu X(t/\nu)\) and \(X'_\nu(t)= X'(t/\nu)\), we use \eqref{est-shift} for the shift $X$ to have
\[
\abs{X'_\nu(t)} \le C \left( f_\nu(t)  +1 \right),    
\]
where  \(f_\nu(t)\coloneqq f(t/\nu)\). This and  \(X_\nu(0)=0\) yield
\begin{equation}\label{X-unif}
\abs{X_\nu(t)} \le Ct+C\int_0^t f_\nu(s) ds.
\end{equation}
Since \eqref{est-shift} and \eqref{ini_conv} imply the boundedness of \(\{f_\nu\}_{\nu>0}\) is in \(L^1(0,T)\), 
\(X_\nu'\) and \(X_\nu\) are uniformly bounded in \(L^1(0,T)\). Therefore, the compactness of BV \cite[Theorem 3.23]{AmbroFuscoPall00} gives the desired result.
\end{proof}

By \eqref{X-unif}, we choose a positive constant \(r = r(T)>1\) such that \(\norm{X_\n}_{L^\infty(0, T)}\le r/3\) for any \(\nu \in (0,\nu_*)\).
Note that \(r\) is a constant depends only on \(\Ecal_0\), not on \(\n\); we may take \(r\) to satisfy \(r \le C(\Ecal_0+T+1)\).
Then we consider a nonnegative smooth function \(\ps\colon \RR\to \RR\) such that \(\ps(x) = \ps(-x)\), and \(\ps'(x)\le 0\) for all \(x\ge 0\), and \(\abs{\ps'(x)}\le 2/r\) for all \(x\in \RR\), and
\[
\ps(x) = \begin{cases}
    1, &\text{ if } \abs{x}\le r, \\
    0, &\text{ if } \abs{x}\ge 2r.
\end{cases}
\]
On the other hand, let \(\th\colon \RR\to \RR\) be a nonnegative smooth function (not the absolute temperature) such that \(\th(s) = \th(-s)\), \(\int_\RR \th = 1\) and \(\supp \th \subseteq [-1, 1]\), and for any \(\e>0\), let 
\[
\th_\e(s)\coloneqq \frac{1}{\e}\th(\frac{s-\e}{\e}). 
\]
Here, \(\e\) is a parameter, not the shock strength. 
Then for a given \(t\in(0, T)\), and any \(\e<t/2\), we define a nonnegative smooth function 
\[
\phi_{t, \e}(s) \coloneqq \int_0^s\left(\th_\e(\t)-\th_\e(\t-t)\right)d\t. 
\]
Since \(v_t^\nu-u_x^\nu = \n((\t_0+(\thn)^2)\frac{v_x^\nu}{\vn})_x\), we have 
\begin{multline}\label{weak-eq}
\int_{[0, T]\times \RR} \left(\phi_{t, \e}'(s)\ps(x)\vn(s, x) - \phi_{t, \e}(s)\ps'(x)\un(s, x)\right)\, dsdx \\
= \int_{[0, T]\times \RR}\phi_{t, \e}(s)\ps'(x)\n(\t_0+(\thn)^2)\frac{v_x^\nu}{\vn}(s, x)\, dsdx.
\end{multline}
Note that \(\int_\RR \ps(x)\vn(s, x)dx\) is continuous with respect to \(s\); this is because our solutions lie in \(\Xcal_T\), which gives \(v_t^\n \in L^2(0,T;L^2(\RR))\), so that 
\begin{multline*}
\abs{\int_\RR \ps(x)\vn(s, x)dx - \int_\RR \ps(x)\vn(s', x)dx}
\le \int_{\abs{x}\le 2r}\int_{s}^{s'} \abs{v_t^\n(t, x)} dtdx \\
\le \sqrt{4r\abs{s-s'}}\norm{v_t^\n}_{L^2(0, T;L^2(\RR))}\to 0
\text{ as } \abs{s-s'}\to 0.
\end{multline*}
Then, since \(\un, \vn\), and \((\t_0+(\thn)^2)\frac{v_x^\n}{\vn}\) are all locally integrable, thanks to the dominated convergence theorem, we take \(\e\to 0\) to obtain 
\begin{multline}\label{eq-to-decompose}
\int_\RR \ps(x)v_0^\n(x)dx - \int_\RR \ps(x)\vn(t, x)dx - \int_0^t\int_\RR\ps'(x)\un(s, x) dxds \\
= \n\int_0^t\int_\RR \ps'(x)(\t_0+(\thn)^2)(s,x)\Big(\frac{v_x^\n}{\vn}\Big)(s,x) dxds\eqqcolon J_1.
\end{multline}
We now decompose the left-hand side of \eqref{eq-to-decompose} to obtain that 
\begin{align}
\begin{aligned} \label{eq-to-decompose1}
0 &= J_1
+\underbrace{\int_\RR \ps(x)\left[\vn(t, x)-\vtn(x-X_\n(t))\right] dx}_{\eqqcolon J_2} 
+\underbrace{\int_\RR \ps(x)\left[\vtn(x)-v_0^\n(x)\right] dx}_{\eqqcolon J_3} \\
&\qquad
+\underbrace{\int_0^t\int_\RR \ps'(x)\left[\un(s, x)-\utn(x-X_\n(s))\right]dxds}_{\eqqcolon J_4} \\
&\qquad
+\underbrace{\int_\RR \ps(x)\left[\vtn(x-X_\n(t))-\vtn(x)\right]dx
+\int_0^t\int_\RR \ps'(x)\utn(x-X_\n(s))dxds}_{\eqqcolon J_5}.
\end{aligned}
\end{align}
We will take \(\n\to 0\) to obtain a bound on \(\abs{X_{\infty}(t)-\s t}\). 
Note that, Lemma \ref{lem-X} guarantees the \(L^1\) convergence \(X_\n(t)\to X_\infty(t)\), and so up to a subsequence, we also have its pointwise convergence, namely, \(\abs{X_\n(t)-\s t}\to \abs{X_\infty(t)-\s t}\) almost everywhere.
In the below, the constant \(C\) does not depend on \(\Ecal_0\) or \(T\); in the case when we need constants depending on them, we will use \(C(\Ecal_0)\) or \(C(T)\) instead. 
Let \(\d\in(0, 1)\) be arbitrary and \(\n_*\) be the associated constant given by \eqref{ineq-m}.
Then, for \(\n<\n_*\), we have the following estimates. 

\(\bullet\) \textbf{Control of \(J_1\):}
Considering the support of \(\ps\) and the upper bound of \(\abs{\ps'}\), we have
\begin{align*}
\abs{J_1} &\le 
C\n\int_0^t\int_{[-2r, -r]\cup[r, 2r]}(\t_0+(\thn)^2)(s, x)\abs{\frac{v_x^\n}{\vn}(s, x)} dxds \\
&\le 
C\n\int_0^t\int_{[-2r, -r]\cup[r, 2r]}\left((\t_0+(\thn)^2)\vn\right)(s, x)\abs{\Big(\frac{1}{\vn(s, x)}-\frac{1}{\vtn(x-X_\n(s))}\Big)_x} dxds \\
&\qquad
+ C\n\int_0^t\int_{[-2r, -r]\cup[r, 2r]}\left((\t_0+(\thn)^2)\vn\right)(s, x)(\vtn)'(x-X_\n(s)) dxds
\eqqcolon J_{11}+J_{12}.
\end{align*}
For \(J_{11}\), we use Young's inequality to have the following: 
\begin{align*}
J_{11} &\le 
C\n\int_0^t\int_\RR \left((\t_0+(\thn)^2)\vn\right)(s, x)\abs{\Big(\frac{1}{\vn(s, x)}-\frac{1}{\vtn(x-X_\n(s))}\Big)_x}^2 dxds \\
&\qquad
+ C\n\int_0^t\int_{[-2r, -r]\cup[r, 2r]} \left((\t_0+(\thn)^2)\vn\right)(s, x) dxds.
\end{align*}
The first term can be bounded by \(C(\Ecal_0+\d)\), using \eqref{ineq-m}. 
For the second term, we need the following pointwise bound of \(\thn\). 
For any \(M>1\) and \(x\in [-M, M]\), we have
\begin{align*}
&\abs{2M\thn(t, x) - \int_{-M}^M\thn(t, y)dy} 
\le \int_{-M}^M \abs{\thn(t, x)-\thn(t, y)}dy \le \int_{-M}^M\int_{-M}^M \abs{\rd_x \thn(t, z)} dzdy \\
&\qquad\le 2M \int_{-M}^M \abs{\rd_x(\thn(t, z)-\thtn(z-X_\n(t)))} dz + 2M\int_{-M}^M \abs{\rd_x\thtn(z-X_\n(t))} dz \\
&\qquad\le 2M\sqrt{\int_{-M}^M \vn(t, z)dz}\sqrt{\int_{-M}^M \frac{1}{\vn}(t, z)\abs{(\thn(t, z)-\thtn(z-X_\n(t)))_x}^2dz} +2MC.
\end{align*}    
In addition, we use \eqref{rel_Phi_v}, \eqref{rel_Phi_th}, and \eqref{ineq-m} to obtain that 
\begin{equation}\label{ineq-m-loc-v}
\int_{-M}^M \vn(t, x) dx \le C\int_{-M}^M \left[\Phi\left(\vn(t, x)/\vtn(x-X_\n(t))\right) + 1\right] dx \le C\left(\Ecal_0 + 2M+ \d\right),
\end{equation}
and
\begin{equation}\label{ineq-m-loc-th}
\int_{-M}^M \thn(t, x) dx \le C\int_{-M}^M \left[\Phi\left(\thn(t, x)/\thtn(x-X_\n(t))\right) + 1\right] dx \le C\left(\Ecal_0 + 2M+ \d\right).
\end{equation}
Thus, 
\begin{equation}\label{pw-th-final}
\begin{aligned}
\thn(t, x) &\le \frac{1}{2M}\int_{-M}^M \thn(t, x)dx \\
&\qquad + \sqrt{\int_{-M}^M \vn(t, x)dx}\sqrt{\int_{-M}^M \frac{1}{\vn}(t, x)\abs{(\thn(t, x)-\thtn(x-X_\n(t)))_x}^2 dx} + C \\
&\le C\bigg(\Ecal_0 + 1 + \sqrt{\Ecal_0+2M}\sqrt{\int_\RR \frac{1}{\vn}(t, x)\abs{(\thn(t, x)-\thtn(x-X_\n(t)))_x}^2 dx}\bigg).
\end{aligned}
\end{equation} 
In the below, we fix \(M = 2r\). Note once again that \(r\) is independent of \(\n<\n_*\). 
With the pointwise estimate \eqref{pw-th-final}, we use \eqref{ineq-m} and \eqref{ineq-m-loc-v} to control the second term: 
\begin{equation}\label{int-th2}
\begin{aligned}
&\n\int_0^t\int_{[-2r, -r]\cup[r, 2r]} \left((\t_0+(\thn)^2)\vn\right)(s, x) dxds \\
&\le C\n(\Ecal_0+1)^2\int_0^t\int_\RR \vn(s, x) dxds \\
&\quad
+ C\n(\Ecal_0+1)\int_0^t\Big(\int_\RR \vn(s, x) dx\int_\RR \frac{1}{\vn}(s, x)\abs{(\thn(s, x)-\thtn(x-X_\n(s)))_x}^2 dx\Big)ds \\
&\le C(\Ecal_0, T)\n + C(\Ecal_0+1)^2(\Ecal_0+\d).
\end{aligned}
\end{equation} 
For \(J_{12}\), we first make an observation: from the second inequality in \eqref{tail} and
\[
\abs{x-X_\n(s)} \ge \abs{x}-r/3 \ge r/2 \text{ for } x\in[-2r, -r]\cup[r, 2r],
\] 
we observe that for any \(\n < \n_*\),
\[
\sup_{x\in [-2r, -r]\cup[r, 2r]} (\vtn)'(x-X_\n(s))
\le \sup_{x\in [-2r, -r]\cup[r, 2r]} \frac{1}{\n} \vt'\Big(\frac{x-X_\n(s)}{\n}\Big)
\le \frac{C}{\n}e^{-C\frac{r}{2\n}} \le C.
\]
Hence, \(J_{12}\) is bounded by the integral in \eqref{int-th2}. Thus, we conclude that  
\[
J_1 \le C(\Ecal_0, T)\n + C(\Ecal_0+1)^2(\Ecal_0+\d).
\]

\(\bullet\) \textbf{Control of \(J_2\):}
We split \(J_2\) into two parts:
\begin{align*}
J_2
&= \int_{\vn \in \left(\frac{1}{3}v_-,3v_-\right)} \ps(x)\left[\vn(t, x)-\vtn(x-X_\n(t))\right] dx \\
&\qquad + \int_{\vn \notin \left(\frac{1}{3}v_-,3v_-\right)} \ps(x)\left[\vn(t, x)-\vtn(x-X_\n(t))\right] dx
\eqqcolon J_{21}+J_{22}.
\end{align*}
From the definition of \(\ps\), we get
\[
\abs{J_{21}} \le \int_{\vn \in \left(\frac{1}{3}v_-,3v_-\right)} \abs{\vn(t, x)-\vtn(x-X_\n(t))} \one{\abs{x}\le 2r} dx.
\]
Hence, we use \eqref{rel_Phi_v} and H\"older's inequality together with \eqref{ineq-m} to obtain that for any \(\nu < \nu_*\),
\[
\abs{J_{21}} \le C \sqrt{\int_\RR \Phi(\vn(t,x)/\vtn(x-X_\nu(t)))dx} \le C \sqrt{\Ecal_0+\d}.   
\]
On the other hand, we apply \eqref{rel_Phi_v} more directly for \(J_{22}\).
That is, from \eqref{ineq-m}, it holds that for any \(\nu < \nu_*\),
\[
\abs{J_{22}} \le \int_\RR \Phi(\vn(t,x)/\vtn(x-X_\nu(t)))dx \le C(\Ecal_0 + \d).
\]
Thus, \(J_2\) can be controlled as follows:
\[
\abs{J_2} \le \abs{J_{21}} + \abs{J_{22}} \le C(\sqrt{\Ecal_0+\d}+\Ecal_0+\d).    
\]

\(\bullet\) \textbf{Control of \(J_3\):}
We also split \(J_3\) into two parts:
\begin{align*}
J_3
&= \int_{v_0^\nu \in \left(\frac{1}{3}v_-,3v_-\right)} \ps(x)\left[\vtn(x)-v_0^\n(x)\right] dx
+ \int_{v_0^\nu \notin \left(\frac{1}{3}v_-,3v_-\right)} \ps(x)\left[\vtn(x)-v_0^\n(x)\right] dx.
\end{align*}
Then, by the same argument as above, we obtain from \eqref{rel_Phi_v} that
\[
\abs{J_3} \le C \sqrt{\int_\RR \Phi(v_0^\nu(x)/\vtn(x)) dx} + \int_\RR \Phi(v_0^\nu(x)/\vtn(x)) dx.
\]
Hence, \eqref{ini_conv} gives that for any \(\nu < \nu_*\),
\[
\abs{J_3} \le C (\sqrt{\Ecal_0+\d}+\Ecal_0+\d).
\]

\(\bullet\) \textbf{Control of \(J_4\):}
Since \(\supp \ps' \subseteq [-2r,2r] \) and  \(\abs{\ps'} \le 2/r\),
H\"older's inequality gives that
\begin{align*}
\abs{J_4}
&\le \int_0^t\int_\RR \abs{\ps'(x)} \abs{\un(s, x)-\utn(x-X_\n(s))} \one{\abs{x}\le 2r} dxds \\
&\le C \int_0^t \sqrt{\int_\RR \abs{\un(s, x)-\utn(x-X_\n(s))}^2 dx} ds.
\end{align*}
On the other hand, from \eqref{ineq-m}, we have
\[
\int_\RR \frac{1}{2\thtn(x-X_\n(t))} \abs{\un(t,x)-\utn(x-X_\n(t))}^2
\le C\Ecal_0 +\d,
\]
and so, it simply follows from the boundedness of \(\thtn \in (\th_+,\th_-)\) that
\[
\int_\RR \abs{\un(t,x)-\utn(x-X_\n(t))}^2
\le C(\Ecal_0 +\d).
\]
Thus, we have
\[
\abs{J_4} \le CT\sqrt{\Ecal_0 + \d}.    
\]

\(\bullet\) \textbf{Control of \(J_5\):} We use the following lemma.
\begin{lemma}\label{lem-diff}
For any \(t\in(0,T)\) and any \(\nu<\nu_*\), the following holds:
\begin{equation}\label{X-control-5}
\abs{J_5-(X_\n(t)-\s t)(v_--v_+)} \le C \nu (t+1).
\end{equation}
\end{lemma}

\begin{proof}
First of all, we decompose \(J_5\) as follows:
\begin{align*}
&J_5
=\int_\RR \ps(x)\left[\vbar(x-X_\n(t))-\vbar(x)\right]dx
+\int_0^t\int_\RR \ps'(x)\ubar(x-X_\n(s))dxds \\
&\qquad
+\int_\RR \ps(x)\left[\vbar(x)-\vtn(x)\right]dx 
+\int_\RR \ps(x)\left[\vtn(x-X_\n(t))-\vbar(x-X_\n(t))\right]dx \\
&\qquad
+\int_0^t\int_\RR \ps'(x) \left[\utn(x-X_\n(s)) - \ubar(x-X_\nu(s)) \right]dxds 
\eqqcolon J_{51}+J_{52}+J_{53}+J_{54}+J_{55}.
\end{align*}
From the choice of \(r\), the definition of the inviscid shock \eqref{shock-0} and the Rankine-Hugoniot condition \eqref{end-con} imply that
\[
J_{51}+J_{52} = X_\n(t)(v_- - v_+) + t(u_- - u_+)
= (X_\n(t)-\s t)(v_- - v_+).
\]
On the other hand, we know from \eqref{tail} that
\[
\abs{J_{53}} \le \int_\RR \abs{\vtn(x)-\vbar(x)}dx
= \nu \int_\RR \abs{\vtil(x)-\vbar(x)}dx \le C \nu,
\]
and so, for any fixed \(s\in(0,T)\), we also have
\[
\abs{J_{54}} \le \int_\RR \abs{\vtn(x-X_\n(t))-\vbar(x-X_\n(t))} dx
= \int_\RR \abs{\vtn(x)-\vbar(x)} dx \le C \nu.
\]
Moreover, since \(\abs{\ps'}\le 2/r\), it follows from \eqref{tail} that
\begin{align*}
\abs{J_{55}}
\le C \int_0^t \int_\RR \abs{\utn(x-X_\n(s)) - \ubar(x-X_\nu(s))} dxds 
\le C \int_0^t \int_\RR \abs{\utn(x) - \ubar(x)} dxds = C \nu t.
\end{align*}
Thus, summing up all, we obtain the desired result.
\end{proof}

From \eqref{eq-to-decompose1}, Lemma \ref{lem-diff} gives a bound of \(\abs{X_\n-\s t}\). For any \(\d >0\), we have \(\n_*>0\) such that \(\n<\n_*\) implies the following: 
\[
\abs{X_\n(t)-\s t}\abs{v_+-v_-} \le C(\Ecal_0+1)^2(\Ecal_0+\d) + C(T)\sqrt{\Ecal_0+\d} + C(\Ecal_0, T)\nu.
\]
Therefore, for a.e. \(t\in(0, T)\), being a pointwise limit of \(X_\n(t)\), \(X_\infty(t)\) satisfies 
\[
\abs{X_\infty(t) -\s t}\abs{v_+-v_-} \le C(T) (\sqrt{\Ecal_0}+\Ecal_0^3).
\]
This completes the proof of \eqref{X-control}. \qed


\begin{appendix}
\section{Proof of Lemma \ref{lem-rel}} \label{appendix_RHS}
\setcounter{equation}{0}
To derive the desired quadratic structure, we use a change of variable \(\x\mapsto \x-X(t)\) as
\begin{equation}\label{move-X}
\int_{\RR} (a\tht)(\x)\eta(U^X(t,\x)|\Ut(\x)) d\x
=\int_{\RR} (a\tht)^{-X}(\x)\eta(U(t,\x)|\Ut^{-X}(\x)) d\x.
\end{equation}
For the sake of simplicity, we will omit the superscript \(-X\) in the shock variables in this proof.
Then, we use (\ref{relative_e}) to have 
\begin{equation} \label{RHS-0}
\begin{aligned}
&\frac{d}{dt}\int_{\RR} (a\tht)(\x)\eta(U(t,\x)|\Ut(\x)) d\x \\
&= -\dot{X}\int_\RR a' \tht \et(U|\Ut)d\x 
+\int_\RR a \rd_t\Big[R\tht\Phif{v}{\vt} + \frac{R\tht}{\g-1}\Phif{\th}{\tht} + \frac{(u-\ut)^2}{2}\Big]d\x.
\end{aligned}
\end{equation}
To compute the second term of the right-hand side, we first use the two systems (\ref{NS}) and (\ref{til-system}) (satisfied by \(\Ut^{-X}\)) to have
\begin{equation}\label{pereq1}
\left\{
\begin{aligned}
    &(v-\vt)_t -\s_\e(v-\vt)_\x -\dot{X}(t) \vt_\x -(u-\ut)_\x = Pv-P\vt, \\
    &(u-\ut)_t -\s_\e(u-\ut)_\x -\dot{X}(t)\ut_\x +(p-\pt)_\x = Pu-P\ut, \\
    &\frac{R}{\g-1}(\th-\tht)_t -\frac{R\s_\e}{\g-1}(\th-\tht)_\x - \frac{R}{\g-1}\dot{X}(t) \tht_\x +(pu_\x-\pt \ut_\x) = P\th-P\tht.
\end{aligned}
\right.
\end{equation}
Note that the NSF system and the BNSF system in Lagrangian mass coordinates share the same entropy, and their difference appears only in the parabolic effects.
Therefore, similar to \cite[Lemma 4.3]{KVW-NSF}, we can calculate the time derivative of the weighted relative entropy, handling only the parabolic terms separately.

Thanks to \cite[Lemma 4.3]{KVW-NSF}, we have
\begin{align*}
&\frac{d}{dt}\int_{\RR} (a\tht)(\x)\eta(U(t,\x)|\Ut(\x)) d\x 
= \dot{X}Y(U) + \Jcal^{bad}(U) - \Jcal^{good}(U)\\
&\quad
+ \int_\RR a \Big[R\tht v \Big(\frac{1}{\vt}-\frac{1}{v}\Big)\Big(\frac{Pv}{v}-\frac{P\vt}{\vt}\Big)
+ (\th-\tht)\Big(\frac{P\th}{\th}-\frac{P\tht}{\tht}\Big)
+ (u-\ut) (Pu-P\ut)\Big]d\x.
\end{align*}
Now, it remains to extract \(\Jcal^{para}(U^X) - \Dcal(U^X)\) from the last line.
To this end, we first gather all the terms with \(Pv\) and \(P\vt\):
\begin{align*}
&\int_\RR a R\tht v \Big(\frac{1}{\vt}-\frac{1}{v}\Big)\Big(\frac{Pv}{v}-\frac{P\vt}{\vt}\Big) d\x
=
R\int_\RR a\tht \Big(\frac{1}{v}-\frac{1}{\vt}\Big)
\Big(v(\t_0+\th^2)\Big(\frac{1}{v}-\frac{1}{\vt}\Big)_\x\Big)_\x d\x\\
&\qquad\qquad\qquad\qquad\qquad\qquad
+ R\int_\RR a\tht\Big(\frac{1}{v}-\frac{1}{\vt}\Big)
\Big(\Big(v(\t_0+\th^2)-\vt(\t_0+\tht^2)\Big)\Big(\frac{1}{\vt}\Big)_\x\Big)_\x d\x \\
&\qquad\qquad\qquad\qquad\qquad\qquad
-R\int_\RR a\tht v\Big(\frac{1}{v}-\frac{1}{\vt}\Big)^2
\Big((\t_0+\tht^2)\frac{\vt_\x}{\vt}\Big)_\x d\x.
\end{align*}
Then, integration by parts for the first two terms gives the first term of \(\Dcal(U)\) in \eqref{ybg-first} and \eqref{RHS-v}.
Now, we gather all the terms with \(Pu\) and \(P\ut\):
\[
\int_\RR a (u-\ut) (Pu-P\ut)d\x
=  \int_\RR a (u-\ut)\Big(\frac{\th^2}{v}(u-\ut)_\x +\Big(\frac{\th^2}{v}-\frac{\tht^2}{\vt}\Big)\ut_\x\Big)_\x d\x.
\]
Applying integration by parts for the first term,
\begin{align*}
&= - \int_\RR a \frac{\th^2}{v} \abs{(u-\ut)_\x}^2 d\x
- \int_\RR a' \frac{\th^2}{v}(u-\ut)(u-\ut)_\x d\x
+ \int_\RR a(u-\ut)\Big(\frac{\th^2}{v}-\frac{\tht^2}{\vt}\Big) \ut_{\x\x} d\x\\
&\qquad
+2 \int_\RR a(u-\ut)\frac{\th}{v}(\th-\tht)_\x \ut_\x d\x
+2 \int_\RR a(u-\ut)\Big(\frac{\th}{v}-\frac{\tht}{\vt}\Big) \tht_\x \ut_\x d\x\\
&\qquad
+ \int_\RR a(u-\ut)\th^2\Big(\frac{1}{v}-\frac{1}{\vt}\Big)_\x \ut_\x d\x
+ \int_\RR a(u-\ut)(\th^2-\tht^2)\Big(\frac{1}{\vt}\Big)_\x \ut_\x d\x.
\end{align*}
We gather terms containing \(u_\x\) and \(\ut_\x\) from \(P\th\) and \(P\tht\) as well:
\begin{align*}
& \int_\RR a(\th-\tht)\Big( \frac{\th}{v}(u_\x)^2 -\frac{\tht}{\vt}(\ut_\x)^2 \Big) d\x\\
&= \int_\RR a \frac{\th(\th-\tht)}{v} \abs{(u-\ut)_\x}^2 d\x
+2 \int_\RR a \frac{\th}{v}(\th-\tht)(u-\ut)_\x \ut_\x d\x 
+ \int_\RR a (\th-\tht)\Big(\frac{\th}{v}-\frac{\tht}{\vt}\Big) (\ut_\x)^2 d\x.
\end{align*}
Then, by summing up, we obtain the second term of \(\Dcal(U)\) in \eqref{ybg-first} and \eqref{RHS-u}.

Lastly, we gather all the terms from the remaining parts of \(P\th\) and \(P\tht\): 
\begin{align*}
& \int_\RR a(\th-\tht)\bigg( \frac{\big(\frac{\th^2}{v}\th_\x\big)_\x}{\th}-\frac{\big(\frac{\tht^2}{\vt}\tht_\x\big)_\x}{\tht}\bigg) d\x\\
&\quad
=  \int_\RR a(\th-\tht)\Big(\frac{\th}{v}(\th-\tht)_\x+\Big(\frac{\th}{v}-\frac{\tht}{\vt}\Big)\tht_\x\Big)_\x d\x
+  \int_\RR a(\th-\tht)\Big(\frac{(\th_\x)^2}{v}-\frac{(\tht_\x)^2}{\vt}\Big) d\x. \\
\end{align*}
Similarly, by applying integration by part for the first term, this expression can be simplified as the third term of \(\Dcal(U)\) in \eqref{ybg-first} with \eqref{RHS-th}.
This completes the proof.
\qed

\section{Proof of Proposition \ref{prop_nl_Poincare}} 
\setcounter{equation}{0}
\label{appendix_prop_nl_Poincare}
To prove Proposition \ref{prop_nl_Poincare}, we first need to obtain an estimate on a specific polynomial, which is the modified version of \cite[Proposition 3.2]{KV21}. 
Let \(\th \coloneqq \sqrt{5-\frac{\pi^2}{3}}\), and \(\d>0\) be any constant. 
We consider the following polynomial functionals. 
\begin{align*}
E(Z_1, Z_2) &\coloneqq Z_1^2 + Z_2^2 + 2Z_1, \\
P_\d(Z_1, Z_2) &\coloneqq (1+\d)(Z_1^2+Z_2^2) + 2Z_1Z_2^2 +\frac{2}{3}Z_1^3 + 6\d\abs{Z_1}(Z_1^2+Z_2^2) \\
&\qquad
- 2\Big(0.9-\d-\Big(\frac{2}{3}+\d\Big)\th Z_2\Big) Z_2^2.
\end{align*}
\begin{proposition}\label{prop_polynomial}
There exist constants \(\d_*, \d_1>0\) such that for any \(0<\d<\d_*\), the following is true. If \((Z_1, Z_2)\in \RR^2\) satisfies \(\abs{E(Z_1, Z_2)}\le \d_1\), then 
\begin{equation}
P_\d(Z_1, Z_2) - \abs{E(Z_1, Z_2)}^2 \le 0.
\end{equation}
\end{proposition}
\begin{remark}
Note that the change of \(1\) to \(0.9\) does not alter the conclusion and the proof: For the case with \(Z_1^2 + Z_2^2 \le r^2\) for sufficiently small \(r>0\), we have 
\begin{align*}
P_\d-\abs{E}^2 \hspace{-1mm}&\le\hspace{-1mm} -2Z_1^2 + (1+\d)\Big(Z_1^2 + Z_2^2 + \frac{r^2}{1+\d}(Z_1^2+Z_2^2) + \frac{(2+6\d)r}{1+\d}Z_2^2 + \frac{((2/3)+6\d)r}{1+\d}Z_1^2\Big) \\
&\qquad
-2\Big(0.9-\d-\frac{2}{3}(1+\d)\th r\Big)Z_2^2
\le 0
\end{align*}
by taking \(\d_*\) sufficiently small. For the case with \(\abs{E(Z_1, Z_2)}\le \d_1\) but \(Z_1^2 + Z_2^2 \ge r^2\), we need the following lemma. With this lemma, by arguing exactly the same with \cite[Proposition 3.2]{KV21}, we obtain the conclusion of Proposition \ref{prop_polynomial}. 
\end{remark}

\begin{lemma}
For all \(x\in [-2, 0)\), 
\[
1.6x - 2.2x^2 -\frac{4}{3}x^3 + \frac{4\th}{3}(-x^2-2x)^{3/2} < 0.
\]
\end{lemma}

\begin{proof}
Let \(h\) be the given polynomial. Then, we have
\begin{align*}
h'(x) &= 1.6-4.4x -4x^2 -4\th(x+1)\sqrt{1-(x+1)^2}, \\
h''(x) &= -4.4-8x + 4\th\frac{2(x+1)^2-1}{\sqrt{1-(x+1)^2}}.
\end{align*}
We prove the lemma with the following four steps. 
\step{1} Near the endpoints \(-2\) and \(0\), \(h''>0\) or \(h\) is convex.
In particular, \(h''>0\) on \((-2, -1-\sqrt{2}/2)\cup (-0.1, 0)\). 
For \(x\in (-2, -1-\sqrt{2}/2)\), we have \((x+1)^2 > 1/2\), so that \(h''(x) \ge -8x -4.4 >0\).
For \(x\in (-0.1, 0)\), we have \((x+1)^2 > 0.81\), so that 
\[
4\th\frac{2(x+1)^2-1}{\sqrt{1-(x+1)^2}} \ge 4\th\frac{0.62}{\sqrt{1-0.81}} \ge 7. 
\]
On the other hand, \(-8x-4.4 \ge -4.4\) on this interval, and hence \(h''(x)>0\). 

\step{2} There are three points that \(h'\) changes its sign, say \(x_1,x_2\) and \(x_3\), which satisfies \(x_1\in(-2, -1-\sqrt{2}/2)\), \(x_2\in(-1/3, -0.3)\), \(x_3\in\left(-0.1, 0\right)\). 
To achieve this, let \(x\in [-2,0)\) satisfy \(h'(x) = 0\). Then, we have 
\[
(1.6-4.4x-4x^2)^2 = 16(\th(x+1))^2(1-(x+1)^2),
\]
hence it should be a zero of \(p\), where \(p(x) = (1.6-4.4x-4x^2)^2 - 16(\th(x+1))^2(1-(x+1)^2)\).
Note that \(p\) has at most four zeros in \(\RR\). To spot them, we calculate \(p\) on various points:
\begin{align*}
&p(-2) = 5.4^2 - 0 > 0, \qquad
p\Big(-1-\frac{\sqrt{2}}{2}\Big) = (1.8\sqrt{2})^2 - 4\th^2 < 0, \qquad
p(-1) = 2^2 - 0 >0, \\
&p(-1/3) = \Big(\frac{118}{45}\Big)^2 - \frac{320}{81}\th^2 > 0, \qquad
p(-0.3) = 2.56^2 - (2.8\th)^2\times 0.51 < 0, \\
&p(-0.1) = 2^2 - (3.6\th)^2\times 0.19 < 0, \qquad
p(0) = 1.6^2 >0.
\end{align*}
Thus, \(h'\) has at most four zeros, and each of them should be in the distinct interval among 
\[
\Big(-2, -1-\frac{\sqrt{2}}{2}\Big),\quad
\Big(-1-\frac{\sqrt{2}}{2}, -1\Big), \quad
\left(-1/3, -0.3\right), \quad
(-0.1, 0).
\]
Since \(h'(-2) = -5.4 < 0\) and \(h'(-1-\sqrt{2}/2) = 2\th - 1.8\sqrt{2} >0\), there is indeed a zero of \(h'\) in the first interval, and \(h'\) changes its sign there. 
However, on the second interval \(h'\) does not change its sign: we have \(h'(-1) = 2 > 0\), so if \(h'\) changes sign here, then we have two zeros in this interval, which is a contradiction. 
Since \(h'(-1/3) = 118/45 - 8\th/3\sqrt{5/9} > 0\), \(h'(-0.3) = 2.56 - 2.8\th\sqrt{0.51} < 0\), \(h'(-0.1) = 2-3.6\th\sqrt{0.19} < 0\), and \(h'(0) = 1.6 >0\), we have two zeros of \(h'\) in each of the third and fourth intervals, and \(h'\) changes its sign there.

\step{3} \(h < 0\) on \([-1/3, -0.3]\): In Step 2, we observe that there is a unique local maximum of \(h\) in this interval. Let \(x_*\in(-1/3, -0.3)\) be the local maximum. 
Then, it suffices to show that \(h(x_*) < 0\). Since it holds from \(h'(x_*)=0\) that 
\[
1.6 - 4.4x_* - 4x_*^2 = 4\th(x_*+1)\sqrt{-x_*^2-2x_*},
\]
we have 
\[
h(x_*) = 1.6x_* - 2.2x_*^2 - \frac{4}{3}x_*^3 +\frac{(-x_*^2-2x_*)(1.6-4.4x_*-4x_*^2)}{3(x_*+1)}
= \frac{x_*(3x_*+1)(3x_*+8)}{15(x_*+1)} < 0,
\]
which verifies our claim. 

\step{4} \(h<0\) on \([-2, 0)\). Since \(h(-2) = -4/3 <0\), \(h(-1-\sqrt{2}/2) = -(47+20\sqrt{2})/30 + \sqrt{2}\th/3 < 0\), \(h(-0.1) = -271/1500 + 0.76\th\sqrt{0.19}/3 <0\), and \(h(0) = 0\), it follows from Step 1 that \(h(x) < 0\) on \(x\in [-2, -1-\sqrt{2}/2)\cup(-0.1, 0)\). 
Since \(h\) increases on \((-1-\sqrt{2}/2, -1/3)\) and decreases on \((-0.3, -0.1)\), it follows from Step 3 that \(h(x) < 0\) on \([-1-\sqrt{2}/2, -0.1]\).
This completes the proof. 
\end{proof}

Now we prove Proposition \ref{prop_nl_Poincare}.
\begin{proof}[Proof of Proposition \ref{prop_nl_Poincare}]
The proof is largely the same as \cite[Proposition 3.3]{KV21}.
The only thing we need to show is that the smaller coefficient \(0.9\) for the diffusion term is enough for the inequality. 
Let 
\[
Z_1 \coloneqq \WBar = \int_0^1 W dy, \quad 
Z_2 \coloneqq \Big(\int_0^1(W-\WBar)^2 dy\Big)^{1/2}, \quad
E(Z_1, Z_2) \coloneqq Z_1^2 + Z_2^2 + 2Z_1.
\]
Then, following the same calculation as \cite[(3.13)-(3.15)]{KV21}, we obtain that 
\begin{equation*}
\Rcal_\d(W) 
\le -\frac{1}{\d}\abs{E(Z_1, Z_2)}^2 + (1+\d)(Z_1^2+Z_2^2) + 2Z_1Z_2^2 + \frac{2}{3}Z_1^3 + 6\d\abs{Z_1}(Z_1^2+Z_2^2) + \Pcal, 
\end{equation*}
where
\begin{equation*}
\Pcal \coloneqq \Big(\frac{2}{3}+\d\Big)\int_0^1 \abs{W-\WBar}^3 dy - (0.9-\d)\int_0^1 y(1-y)\abs{\rd_y W}^2 dy. 
\end{equation*}
For the cubic term in \(\Pcal\), we use \cite[(3.17)]{KV21}, so we have 
\[
\Pcal \le -\Big(0.9-\d-\Big(\frac{2}{3}+\d\Big)\th Z_2\Big)\int_0^1 y(1-y)\abs{\rd_y W}^2 dy.
\]
Since \((Z_1+1)^2 + Z_2^2 = 1 + E(Z_1, Z_2)\), we have \(Z_2 \le \sqrt{1 + \abs{E(Z_1, Z_2)}}\).
In addition, since \(\frac{2}{3}\th\approx 0.88 < 0.9\), there exists a positive constant \(\d_\th\ll 1\) such that 
\[
\frac{2}{3}\th\sqrt{1+\d_\th} < 0.9.
\]
Then, we take \(\d_2 < 1\) such that \(\d<\d_2\) implies 
\[
0.9-\d - \Big(\frac{2}{3}+\d\Big)\th\sqrt{1+\d_\th} >0.
\]
The rest of the proof follows \cite[Proposition 3.3]{KV21}, which separately handles the cases where \(\abs{E(Z_1, Z_2)}\le\min\{\d_\th,\d_1\}\) and \(\abs{E(Z_1, Z_2)}\ge\min\{\d_\th,\d_1\}\).
Thus, the details are omitted.
\end{proof}
\end{appendix}

\subsection*{Declaration of competing interest}
The authors declared that they have no conflict of interest to this work.

\bibliographystyle{plain}
\bibliography{reference} 

\begin{thebibliography}{10}

\bibitem{AmbroFuscoPall00}
L.~Ambrosio, N.~Fusco, and D.~Pallara.
\newblock {\em Functions of bounded variation and free discontinuity problems}.
\newblock Oxford Mathematical Monographs. The Clarendon Press, Oxford University Press, New York, 2000.

\bibitem{arkilic2001mass}
E.~B. Arkilic, K.~S. Breuer, and M.~A. Schmidt.
\newblock Mass flow and tangential momentum accommodation in silicon micromachined channels.
\newblock {\em Journal of fluid mechanics}, 437:29--43, 2001.

\bibitem{BianchiniBressan05}
S.~Bianchini and A.~Bressan.
\newblock Vanishing viscosity solutions of nonlinear hyperbolic systems.
\newblock {\em Ann. of Math. (2)}, 161(1):223--342, 2005.

\bibitem{PaulJeffrey}
P.~Blochas and J.~Cheng.
\newblock Viscous destabilization for large shocks of conservation laws.
\newblock {\em arXiv preprint arXiv:2501.01537}, 2025.

\bibitem{brenner2005kinematics}
H.~Brenner.
\newblock Kinematics of volume transport.
\newblock {\em Physica A: Statistical Mechanics and its Applications}, 349(1-2):11--59, 2005.

\bibitem{brenner2005navier}
H.~Brenner.
\newblock Navier--stokes revisited.
\newblock {\em Physica A: Statistical Mechanics and its Applications}, 349(1-2):60--132, 2005.

\bibitem{brenner2006fluid}
H.~Brenner.
\newblock Fluid mechanics revisited.
\newblock {\em Physica A: Statistical Mechanics and its Applications}, 370(2):190--224, 2006.

\bibitem{Bressan20}
A.~Bressan.
\newblock {\em Hyperbolic systems of conservation laws}, volume~20 of {\em Oxford Lecture Series in Mathematics and its Applications}.
\newblock Oxford University Press, Oxford, 2000.
\newblock The one-dimensional Cauchy problem.

\bibitem{BressanCrastaPiccoli00}
A.~Bressan, G.~Crasta, and B.~Piccoli.
\newblock Well-posedness of the {C}auchy problem for {$n\times n$} systems of conservation laws.
\newblock {\em Mem. Amer. Math. Soc.}, 146(694):viii+134, 2000.

\bibitem{BressanDeLellis23}
A.~Bressan and C.~De~Lellis.
\newblock A remark on the uniqueness of solutions to hyperbolic conservation laws.
\newblock {\em Arch. Ration. Mech. Anal.}, 247(6):Paper No. 106, 12, 2023.

\bibitem{BressanGuerra24}
A.~Bressan and G.~Guerra.
\newblock Unique solutions to hyperbolic conservation laws with a strictly convex entropy.
\newblock {\em J. Differential Equations}, 387:432--447, 2024.

\bibitem{BressanLiuYang99}
A.~Bressan, T.-P. Liu, and T.~Yang.
\newblock {$L^1$} stability estimates for {$n\times n$} conservation laws.
\newblock {\em Arch. Ration. Mech. Anal.}, 149(1):1--22, 1999.

\bibitem{Fe1}
J.~B\v{r}ezina and E.~Feireisl.
\newblock Measure-valued solutions to the complete {E}uler system revisited.
\newblock {\em Z. Angew. Math. Phys.}, 69(3):Paper No. 57, 17, 2018.

\bibitem{ChapmanCowling90}
S.~Chapman and T.~G. Cowling.
\newblock The mathematical theory of non uniform gases, cambridge mathematical library, 1970.

\bibitem{ChenKangVasseur24arxiv}
G.~Chen, M.-J. Kang, and A.~F. Vasseur.
\newblock From navier-stokes to bv solutions of the barotropic euler equations.
\newblock {\em arXiv preprint arXiv:2401.09305}, 2024.

\bibitem{ChenKrupaVasseur22}
G.~Chen, S.~G. Krupa, and A.~F. Vasseur.
\newblock Uniqueness and weak-{BV} stability for {$2\times 2$} conservation laws.
\newblock {\em Arch. Ration. Mech. Anal.}, 246(1):299--332, 2022.

\bibitem{ChenFridLi02}
G.-Q. Chen, H.~Frid, and Y.~Li.
\newblock Uniqueness and stability of {R}iemann solutions with large oscillation in gas dynamics.
\newblock {\em Comm. Math. Phys.}, 228(2):201--217, 2002.

\bibitem{ChenPerepelitsa10}
G.-Q. Chen and M.~Perepelitsa.
\newblock Vanishing viscosity limit of the {N}avier-{S}tokes equations to the {E}uler equations for compressible fluid flow.
\newblock {\em Comm. Pure Appl. Math.}, 63(11):1469--1504, 2010.

\bibitem{ChenVasseurYu}
R.~M. Chen, A.~F. Vasseur, and C.~Yu.
\newblock Non-uniqueness for continuous solutions to 1d hyperbolic systems.
\newblock {\em arXiv preprint arXiv:2407.02927}, 2024.

\bibitem{ChioDeLellisKreml15}
E.~Chiodaroli, C.~De~Lellis, and O.~Kreml.
\newblock Global ill-posedness of the isentropic system of gas dynamics.
\newblock {\em Comm. Pure Appl. Math.}, 68(7):1157--1190, 2015.

\bibitem{ChioFeireislKreml15}
E.~Chiodaroli, E.~Feireisl, and O.~Kreml.
\newblock On the weak solutions to the equations of a compressible heat conducting gas.
\newblock {\em Ann. Inst. H. Poincar\'e{} C Anal. Non Lin\'eaire}, 32(1):225--243, 2015.

\bibitem{CKKV}
K.~Choi, M.-J. Kang, Y.-S. Kwon, and A.~F. Vasseur.
\newblock Contraction for large perturbations of traveling waves in a hyperbolic-parabolic system arising from a chemotaxis model.
\newblock {\em Math. Models Methods Appl. Sci.}, 30(2):387--437, 2020.

\bibitem{dadzie2008continuum}
S.~K. Dadzie, J.~M. Reese, and C.~R. McInnes.
\newblock A continuum model of gas flows with localized density variations.
\newblock {\em Physica A: Statistical Mechanics and its Applications}, 387(24):6079--6094, 2008.

\bibitem{Dafermos79}
C.~M. Dafermos.
\newblock The second law of thermodynamics and stability.
\newblock {\em Arch. Rational Mech. Anal.}, 70(2):167--179, 1979.

\bibitem{Dafermos96}
C.~M. Dafermos.
\newblock Entropy and the stability of classical solutions of hyperbolic systems of conservation laws.
\newblock In {\em Recent mathematical methods in nonlinear wave propagation ({M}ontecatini {T}erme, 1994)}, volume 1640 of {\em Lecture Notes in Math.}, pages 48--69. Springer, Berlin, 1996.

\bibitem{dafermos2005hyperbolic}
C.~M. Dafermos.
\newblock {\em Hyperbolic conservation laws in continuum physics}, volume~3.
\newblock Springer, 2005.

\bibitem{DeLellisLaszlo09}
C.~De~Lellis and L.~Sz\'ekelyhidi, Jr.
\newblock The {E}uler equations as a differential inclusion.
\newblock {\em Ann. of Math. (2)}, 170(3):1417--1436, 2009.

\bibitem{DeLellisLaszlo10}
C.~De~Lellis and L.~Sz{\'e}kelyhidi, Jr.
\newblock On admissibility criteria for weak solutions of the {E}uler equations.
\newblock {\em Arch. Ration. Mech. Anal.}, 195(1):225--260, 2010.

\bibitem{Diperna79}
R.~J. DiPerna.
\newblock Uniqueness of solutions to hyperbolic conservation laws.
\newblock {\em Indiana Univ. Math. J.}, 28(1):137--188, 1979.

\bibitem{dongari2009extended}
N.~Dongari, R.~Sambasivam, and F.~Durst.
\newblock Extended navier-stokes equations and treatments of micro-channel gas flows.
\newblock {\em Journal of Fluid Science and Technology}, 4(2):454--467, 2009.

\bibitem{EE-DBNSF}
S.~Eo and N.~Eun.
\newblock Traveling wave solutions to {B}renner-{N}avier-{S}tokes-{F}ourier system with temperature dependent transport coefficients.
\newblock {\em in preparation}.

\bibitem{EEK}
S.~Eo, N.~Eun, and M.-J. Kang.
\newblock Global existence of large strong solutions to {B}renner-{N}avier-{S}tokes-{F}ourier system.
\newblock {\em in preparation}.

\bibitem{EEKO24}
S.~Eo, N.~Eun, M.-J. Kang, and H.~Oh.
\newblock Traveling wave solutions to {B}renner-{N}avier-{S}tokes-{F}ourier system.
\newblock {\em Journal of Differential Equations}, 422:639--658, 2025.

\bibitem{Fe2}
E.~Feireisl, M.~Luk\'a\v{c}ov\'a-Medvid'ov\'a, and H.~Mizerov\'a.
\newblock A finite volume scheme for the {E}uler system inspired by the two velocities approach.
\newblock {\em Numer. Math.}, 144(1):89--132, 2020.

\bibitem{FeVa}
E.~Feireisl and A.~F. Vasseur.
\newblock New perspectives in fluid dynamics: Mathematical analysis of a model proposed by howard brenner.
\newblock {\em New Directions in Mathematical Fluid Mechanics: The Alexander V. Kazhikhov Memorial Volume}, pages 153--179, 2010.

\bibitem{fermi2012thermodynamics}
E.~Fermi.
\newblock {\em Thermodynamics}.
\newblock Courier Corporation, 2012.

\bibitem{Glimm65}
J.~Glimm.
\newblock Solutions in the large for nonlinear hyperbolic systems of equations.
\newblock {\em Comm. Pure Appl. Math.}, 18:697--715, 1965.

\bibitem{greenshields2007structure}
C.~J. Greenshields and J.~M. Reese.
\newblock The structure of shock waves as a test of brenner's modifications to the navier--stokes equations.
\newblock {\em Journal of Fluid Mechanics}, 580:407--429, 2007.

\bibitem{harley1995gas}
J.~C. Harley, Y.~Huang, .aim~H. Bau, and J.~N. Zemel.
\newblock Gas flow in micro-channels.
\newblock {\em Journal of fluid mechanics}, 284:257--274, 1995.

\bibitem{Kang18}
M.-J. Kang.
\newblock Non-contraction of intermediate admissible discontinuities for 3-{D} planar isentropic magnetohydrodynamics.
\newblock {\em Kinet. Relat. Models}, 11(1):107--118, 2018.

\bibitem{Kang19}
M.-J. Kang.
\newblock {$L^2$}-type contraction for shocks of scalar viscous conservation laws with strictly convex flux.
\newblock {\em J. Math. Pures Appl. (9)}, 145:1--43, 2021.

\bibitem{KO}
M.-J. Kang and H.~Oh.
\newblock {$L^2$} decay for large perturbations of viscous shocks for multi-{D} {B}urgers equation.
\newblock {\em Anal. Appl. (Singap.)}, 23(3):475--488, 2025.

\bibitem{KV16}
M.-J. Kang and A.~F. Vasseur.
\newblock Criteria on contractions for entropic discontinuities of systems of conservation laws.
\newblock {\em Arch. Ration. Mech. Anal.}, 222(1):343--391, 2016.

\bibitem{Kang-V-1}
M.-J. Kang and A.~F. Vasseur.
\newblock {$L^2$}-contraction for shock waves of scalar viscous conservation laws.
\newblock {\em Ann. Inst. H. Poincar\'{e} C Anal. Non Lin\'{e}aire}, 34(1):139--156, 2017.

\bibitem{KV21}
M.-J. Kang and A.~F. Vasseur.
\newblock Contraction property for large perturbations of shocks of the barotropic {N}avier-{S}tokes system.
\newblock {\em J. Eur. Math. Soc.}, 23(2):585--638, 2021.

\bibitem{KV-Inven}
M.-J. Kang and A.~F. Vasseur.
\newblock Uniqueness and stability of entropy shocks to the isentropic {E}uler system in a class of inviscid limits from a large family of {N}avier-{S}tokes systems.
\newblock {\em Invent. Math.}, 224(1):55--146, 2021.

\bibitem{KV-JDE}
M.-J. Kang and A.~F. Vasseur.
\newblock Well-posedness of the {R}iemann problem with two shocks for the isentropic {E}uler system in a class of vanishing physical viscosity limits.
\newblock {\em J. Differential Equations}, 338:128--226, 2022.

\bibitem{KVW}
M.-J. Kang, A.~F. Vasseur, and Y.~Wang.
\newblock {$L^2$}-contraction of large planar shock waves for multi-dimensional scalar viscous conservation laws.
\newblock {\em J. Differential Equations}, 267(5):2737--2791, 2019.

\bibitem{KVW-CMP}
M.-J. Kang, A.~F. Vasseur, and Y.~Wang.
\newblock Uniqueness of a planar contact discontinuity for 3{D} compressible {E}uler system in a class of zero dissipation limits from {N}avier-{S}tokes-{F}ourier system.
\newblock {\em Comm. Math. Phys.}, 384(3):1751--1782, 2021.

\bibitem{KVW-NSF}
M.-J. Kang, A.~F. Vasseur, and Y.~Wang.
\newblock Time-asymptotic stability of generic riemann solutions for compressible navier-stokes-fourier equations.
\newblock {\em arXiv preprint arXiv:2306.05604}, 2023.

\bibitem{kittel1980thermal}
C.~Kittel and H.~Kroemer.
\newblock {\em Thermal physics}.
\newblock Macmillan, 1980.

\bibitem{kittel2018introduction}
C.~Kittel and P.~McEuen.
\newblock {\em Introduction to solid state physics}.
\newblock John Wiley \& Sons, 2018.

\bibitem{klimontovich1992need}
Y.~L. Klimontovich.
\newblock On the need for and the possibility of a unified description of kinetic and hydrodynamic processes.
\newblock {\em Theoretical and Mathematical Physics}, 92(2):909--921, 1992.

\bibitem{klimontovich1993hamiltonian}
Y.~L. Klimontovich.
\newblock From the hamiltonian mechanics to a continuous media. dissipative structures. criteria of self-organization.
\newblock {\em Theoretical and Mathematical Physics}, 96:1035--1056, 1993.

\bibitem{KlinKremlOMM20}
C.~Klingenberg, O.~Kreml, V.~M\'acha, and S.~Markfelder.
\newblock Shocks make the {R}iemann problem for the full {E}uler system in multiple space dimensions ill-posed.
\newblock {\em Nonlinearity}, 33(12):6517--6540, 2020.

\bibitem{Krupa21}
S.~G. Krupa.
\newblock Finite time stability for the {R}iemann problem with extremal shocks for a large class of hyperbolic systems.
\newblock {\em J. Differential Equations}, 273:122--171, 2021.

\bibitem{Krupa24}
S.~G. Krupa.
\newblock Finite time bv blowup for liu-admissible solutions to $ p $-system via computer-assisted proof.
\newblock {\em arXiv preprint arXiv:2403.07784}, 2024.

\bibitem{KrupaNonT4}
S.~G. Krupa and L.~Sz\'ekelyhidi, Jr.
\newblock Nonexistence of {$T_4$} configurations for hyperbolic systems and the {L}iu entropy condition.
\newblock {\em Adv. Math.}, 454:Paper No. 109856, 49, 2024.

\bibitem{landau2013statistical}
L.~D. Landau and E.~M. Lifshitz.
\newblock {\em Statistical Physics: Volume 5}, volume~5.
\newblock Elsevier, 2013.

\bibitem{LegerVasseur11}
N.~Leger and A.~F. Vasseur.
\newblock Relative entropy and the stability of shocks and contact discontinuities for systems of conservation laws with non-{BV} perturbations.
\newblock {\em Arch. Ration. Mech. Anal.}, 201(1):271--302, 2011.

\bibitem{LiuYang99}
T.-P. Liu and T.~Yang.
\newblock {$L^1$} stability for {$2\times 2$} systems of hyperbolic conservation laws.
\newblock {\em J. Amer. Math. Soc.}, 12(3):729--774, 1999.

\bibitem{Riemann1860}
B.~Riemann.
\newblock {\em {\"U}ber die Fortpflanzung ebener Luftwellen von endlicher Schwingungsweite}, volume~8.
\newblock Verlag der Dieterichschen Buchhandlung, 1860.

\bibitem{DV}
D.~Serre and A.~F. Vasseur.
\newblock About the relative entropy method for hyperbolic systems of conservation laws.
\newblock {\em Contemporary Mathematics}, 658:237--248, 2016.

\bibitem{Vasseur16}
A.~F. Vasseur.
\newblock Relative entropy and contraction for extremal shocks of conservation laws up to a shift.
\newblock In {\em Recent advances in partial differential equations and applications}, volume 666 of {\em Contemp. Math.}, pages 385--404. Amer. Math. Soc., Providence, RI, 2016.

\bibitem{VincentiKruger66}
W.~G. Vincenti, C.~H. Kruger~Jr., and T.~Teichmann.
\newblock Introduction to physical gas dynamics, 1966.

\end{thebibliography}

\end{document}